\documentclass[11pt,a4paper,reqno]{amsart}
\usepackage{a4wide}
\usepackage{amsaddr}
\usepackage{amsmath}
\usepackage{amsfonts, amssymb, amsthm}
\usepackage{paralist}
\usepackage{enumerate}
\usepackage[colorlinks=true]{hyperref}
\usepackage{url}
\hypersetup{urlcolor=black}

\newtheorem{Thm}{Theorem}[section]
\newtheorem{Lem}[Thm]{Lemma}

\newtheorem{Cor}[Thm]{Corollary}
\newtheorem{Prop}[Thm]{Proposition}

\theoremstyle{definition}
\newtheorem{Def}[Thm]{Definition}
\newtheorem{Rem}{Remark}[section]

\numberwithin{equation}{section}

\usepackage[normalem]{ulem}
\usepackage{color}
\newcommand{\N}{{\mathbb N}}

\newcommand{\R}{{\mathbb R}}
\newcommand{\weakto}{{\rightharpoonup}}

\allowdisplaybreaks

\begin{document}

\title[The Brezis-Nirenberg problem]{Asymptotic behavior of multi-peak solutions to the Brezis-Nirenberg problem. The sub-critical perturbation case. }
\author{Jinkai Gao \ and \ Shiwang Ma$^{ \rm *}$}\email{jinkaigao@mail.nankai.edu.cn, shiwangm@nankai.edu.cn}
\address{School of Mathematical Sciences and LPMC, Nankai University\\ 
	Tianjin 300071, China}
\thanks{
  {$^{\rm *}$ Corresponding author.}
}



\date{}

\begin{abstract}
In this paper, we consider the following well-known Brezis-Nirenberg problem
\begin{equation*}
\begin{cases}
-\Delta u= u^{2^*-1}+\varepsilon u^{q-1}, \quad u>0,  &{\text{in}~\Omega},\\
\quad \ \ u=0, &{\text{on}~\partial \Omega},
\end{cases}
\end{equation*}
where $N\geq 3$, $\Omega$ is a smooth and bounded domain in $\R^{N}$, $\varepsilon>0$ is a small parameter, $q\in (2,2^*)$ and $2^*:=\frac{2N}{N-2}$ denotes the critical Sobolev exponent. The existence of solutions to the above problem has been obtained by many authors in the literature. However, as far as the authors know, the asymptotic behavior of solutions to  the above problem is still open. Here we first describe the asymptotic profile of solutions to the above problem as $\varepsilon\to 0$. Then, we derive the exact blow-up rate and characterize the concentration speed and the location of concentration points in the general case of multi-peak solutions. Finally, we prove the uniqueness, nondegeneracy and count the exact number of blow-up solutions. 

The main results in this paper give a complete picture of multi-peak blow-up phenomena in the framework of Brezis-Peletier conjecture in the case of sub-critical perturbation. On the other hand, compared with the special case $q=2$ previously studied in the literature, we observe that the exponent $q$ has a significant impact on the asymptotic behavior, uniqueness and nondegeneracy of solutions in addition to the geometry of domain $\Omega$ and space dimension $N$ which is already known in the literature.
\end{abstract}

\maketitle


\section{Introduction and main results}
\setcounter{equation}{0}
In this paper, we consider the following critical elliptic equation with a sub-critical perturbation term
\begin{equation}\label{p-varepsion}
\begin{cases}
-\Delta u= u^{2^*-1}+\varepsilon u^{q-1}, \quad u>0, &{\text{in}~\Omega},\\
\quad\ \ u=0, &{\text{on}~\partial \Omega},
\end{cases}
\end{equation}
where $N\geq 3$, $\Omega$ is a smooth and bounded domain in $\R^{N}$, $\varepsilon>0$ is a small parameter, $q\in (2,2^*)$ and $2^*:=\frac{2N}{N-2}$ denotes the critical Sobolev exponent.

Let us first recall some existence and non-existence results of (\ref{p-varepsion}). In 1983, in their celebrated paper \cite{Brezis1983CPAM},  Brezis and Nirenberg  proved that in the case $q=2$, if $N\geq 4$, the problem (\ref{p-varepsion}) admits a positive least energy solution for all $\varepsilon\in (0,\lambda_{1})$, where $\lambda_{1}$ is the first eigenvalue of $-\Delta$ with the homogeneous Dirichlet boundary condition; if $N=3$, there exists $\lambda_{*}\in (0,\lambda_{1})$ such that (\ref{p-varepsion}) has a positive least energy solution for all $\varepsilon\in (\lambda_{*},\lambda_{1})$ and no positive least energy solution exists for $\varepsilon \in (0,\lambda_{*})$. In particular, when $\Omega$ is a unit ball, \eqref{p-varepsion} has a positive least energy solution if and only if $\varepsilon\in (\frac{1}{4}\pi^{2},\pi^{2})$ and for general domain $\Omega$, we refer to \cite{Druet2002EllipticEW,Esposito}. However, in the case that $q\in (2,2^*)$, the situation is quite different in nature. Brezis and Nirenberg  \cite{Brezis1983CPAM}  proved that the problem \eqref{p-varepsion} has a positive least energy solution if  one of the following conditions holds
\begin{enumerate}
    \item $N\geq 3$, $q\in(\max\{2,\frac{4}{N-2}\}, 2^*)$ and every $\varepsilon>0$;
    \item $N=3$, $q\in (2,4]$ and $\varepsilon$ large enough.
\end{enumerate}
 In particular, when $N=3$ and $\Omega$ is a ball, numerical computation suggest that 
\begin{enumerate}
    \item If $q=4$, there is some $\varepsilon_{0}>0$ such that
    \begin{enumerate}
        \item for $\varepsilon>\varepsilon_{0}$, there is a unique solution;
        \item for $\varepsilon\leq \varepsilon_{0}$, there is no solution.
    \end{enumerate}
    \item If $2<q<4$, there is some $\varepsilon_{0}>0$ such that
    \begin{enumerate}
        \item for $\varepsilon>\varepsilon_{0}$, there are two solutions;
        \item for $\varepsilon=\varepsilon_{0}$, there is a unique solution;
        \item for $\varepsilon<\varepsilon_{0}$, there is no solutions.
    \end{enumerate}
\end{enumerate}
We remark that the second case $2<q<4$ has been confirmed by Atkinson-Peletier \cite{AtkinsonPeletier} afterwards.

In addition to the least energy solution obtained by Brezis and Nirenberg \cite{Brezis1983CPAM}, there are also some special solutions constructed by applying the Lyapounov-Schmidt procedure. Before going on, we need to introduce the notations of Green and Robin function of the domain $\Omega$ which plays a crucial role. Let $G(x,\cdot)$ denote the Green function of the negative Laplacian on $\Omega$, i.e.
\begin{equation}
\begin{cases}
-\Delta G(x,\cdot)= \delta_x, &{\text{in}~\Omega}, \\
\quad \  \ G(x,\cdot)=0, &{\text{on}~\partial\Omega},
\end{cases}
\end{equation}
where $\delta_x$ is the Dirac function at $x\in\Omega$. For $G(x,y)$, we have the following form
\begin{equation*}
G(x,y)=S(x,y)-H(x,y), ~~(x,y)\in \Omega\times \Omega,
\end{equation*}
where (if $N\geq 3$)
\begin{equation*}
S(x,y)=\frac{1}{(N-2)\omega_{N}|x-y|^{N-2}},
\end{equation*}
is the singular part, which is also the fundamental solution to negative Laplace equation, $\omega_{N}$ is a measure of the unit sphere in $\R^N$ and $H(x,y)$ is the regular part of $G(x,y)$ satisfying 
\begin{equation*}
 \begin{cases}
-\Delta H(x,\cdot)=0, &{\text{in}~\Omega}, \\
\quad \  \ H(x,\cdot)=S(x,\cdot), &{\text{on}~\partial\Omega}.
\end{cases}   
\end{equation*}
Particularly, we denote the leading term of $H$ as
\begin{equation}\label{definition of Robin function}
    R(x):=H(x,x), ~~x\in\Omega,
\end{equation}
which is called the Robin function of domain $\Omega$. Moreover, we know that $G(x,y)$ and $H(x,y)$ are symmetric in $x$ and $y$, and by using the maximum principle, we can obtain that $R(x)>0$ in $\Omega$ and $R(x)\to\infty$ if $x\to\partial\Omega$.

Now we recall some more known results on the existence of solutions to \eqref{p-varepsion}. For single-peak solutions: when $N\geq 3, q=2,$ positive solutions of problem (\ref{p-varepsion}) blowing-up at a single point  have been constructed by del Pino et al. \cite{delPino2004},  Pistoia et al. \cite{Pistoia-Serena2024} and  Rey \cite{Rey1990}. On the other hand, when $N\geq 4, q\in (2,2^*)$, for any nondegenerate critical point  $x_{0}$ of Robin function $R(x)$, Molle and Pistoia \cite{Molle2003} proved that there exists a sequence of solutions $u_{\varepsilon}$ which blow-up and concentrate at $x_{0}$ as $\varepsilon\to 0$. For multi-peak solutions:  when $N\geq 5,q=2$, Musso and Pistoia \cite{Msusso2002} proved the existence of solutions which concentrate at more than one point. Moreover,  Musso and Salazar \cite{Musso2018} generalize their results to the case $N=3,q=2$. Particularly, when $N=4, q=2$, the existence of multi-peak solutions to \eqref{p-varepsion} is still open. On the other hand, when $N\geq 4, q\in (2,2^*)$, the existence of positive and sign-changing solutions to \eqref{p-varepsion} with multiple blow-up points have been constructed by Micheletti and  Pistotia \cite{Pistoia2004}. 

In recent years, starting with the work \cite{AtkinsonPeletier,Budd1987}, there has been a
 considerable number of papers on the problem \eqref{p-varepsion} for the asymptotic behavior of positive solutions. First, it is easy to see that the formal limit of (\ref{p-varepsion}) as $\varepsilon\to 0$ is the following critical Emden-Fowler equation 
\begin{equation}\label{p-0}
\begin{cases}
-\Delta u= u^{2^*-1},\quad u>0, &{\text{in}~\Omega},\\
\quad\ \ u=0, &{\text{on}~\partial \Omega},
\end{cases}
\end{equation}
and it is well known \cite{Pohozaev1965} that the problem \eqref{p-0} admits no solutions when $\Omega$ is star-shaped. While, Bahri and  Coron \cite{Bahri1988CPAM} gave an existence result, if $\Omega$ has a non-trivial topology, see also \cite{Corn1984,dancer1988} and the reference therein for other existence results of \eqref{p-0}. In particular, when $\Omega=\mathbb{R}^{N}$, the equation \eqref{p-0} arises as the Euler-Lagrange equation of the following minimizing problem
\begin{equation}
    S:=\inf\left\{\int_{\R^{N}}|\nabla u|^{2}:u\in\mathcal{D}^{1,2}(\R^{N}), \int_{\R^{N}}|u|^{2^*}=1\right\},
\end{equation}
and the solutions of (\ref{p-0}) (see \cite{AUBIN,Talenti}) can be expressed as
\begin{equation}\label{Aubin-Talenti bubble}
    U_{x,\lambda}(y):=\alpha_{N}\left(\frac{\lambda}{1+\lambda^{2}|y-x|^{2}}\right)^{\frac{N-2}{2}},\quad \alpha_{N}=(N(N-2))^{\frac{N-2}{4}},
\end{equation}
which is  called Aubin-Talenti bubble at $x\in\R^{N}$ with height $\lambda\in\mathbb{R}^{+}$. In the following, we also denote the standard bubble by $U$, that is $U:=U_{0,\beta_{N}}$ satisfying $U(0)=1$ and $\beta_{N}=\alpha_{N}^{-2/(N-2)}$. 

Next we recall some results on the asymptotic behavior of positive solutions to \eqref{p-varepsion}. For single-peak solution: when $N=3, q=2$ and $\Omega$ is the unit ball in $\R^{3}$, in an influential paper, Brezis and Peletier \cite{Brezispletier1989}, gave a detailed description on the asymptotic behavior of energy minimizing solutions to \eqref{p-varepsion} as $\varepsilon\searrow\frac{1}{4}\pi^{2}$. In particular, when $N\geq 4,q=2$ and $\Omega$ is a general domain, Brezis and Peletier also proposed a conjecture on the blow-up rate and location of concentration point for the energy minimizing solution of \eqref{p-varepsion} as $\varepsilon\to0$. Afterwards, this conjecture was proved independently by Han \cite{Han1991} and Rey \cite{Rey1989ProofOT}. Very recently, when $N=3,q=2$ and the constant coefficient before the perturbation term is replaced by a regular function $a(x)$ on $\bar{\Omega}$, the corresponding problem was studied by Frank-K\"{o}nig-Kova\v{r}\'{i}k \cite{Frank2024Blowup}. In addition, we also refer to \cite{Frank2019Energyhigher-dimensional,Frank2019Energythree-dimensional,Takahashivariablecoefficients2006} and the reference therein for more results on the single-peak solutions of \eqref{p-varepsion} with $N\geq 3,q=2$. For multi-peak solution: when $N\geq 3, q=2$, the asymptotic behavior of solutions concentrating at several different points has been studied by Cao-Luo-Peng \cite{Cao2021Trans} and K\"{o}nig-Laurain \cite{Knig2022MultibubbleBA,Knig2022FineMA} recently. Hence, when $N\geq 3, q=2$, the blow-up phenomena of the solutions to \eqref{p-varepsion} is fully understand.

Motivate by the results above, it's natural to consider the asymptotic behavior of the multi-peak solutions to \eqref{p-varepsion} with $N\geq3, q\in (2,2^*)$. To state our results precisely, we introduce some more notations. For any $(x,\lambda)\in \Omega\times(0,\infty)$, let $PU_{x,\lambda}$ be the projection of $U_{x,\lambda}$ onto $H^{1}_{0}(\Omega)$, i.e.,
 \begin{equation}
     \begin{cases}
         -\Delta PU_{x,\lambda}=U_{x,\lambda}^{2^*-1},&{\text{in}~\Omega},\\
         \quad\ \ PU_{x,\lambda}=0, &{\text{on}~\partial \Omega},
     \end{cases}
 \end{equation}
and we define
\begin{equation*}
\begin{split}
E_{x,\lambda}=\Big\{v\in H^1_0(\Omega):  \Big\langle \frac{\partial PU_{x,\lambda}}{\partial \lambda},v \Big\rangle=\Big\langle \frac{\partial PU_{x,\lambda}}{\partial x_i},v\Big\rangle=0,~\mbox{for}~i=1,\cdots,N\Big\},
\end{split}
\end{equation*}
where $\langle\cdot,\cdot\rangle$ denotes the inner product in the Sobolev space $H^1_0(\Omega)$. For any fixed number $n\in\mathbb{N}$, let
\begin{equation}
    \Omega_{*}^{n}:=\left\{\vec{x}=(x_{1},\cdots,x_{n})\in\Omega^{n}:x_{i}\neq x_{j} \text{~for all~}i\neq j\right\}.
\end{equation}
For $\vec{x}\in \Omega_{*}^{n}$, we use
${M}(\vec{x})=(m_{i,j})_{i,j=1}^{n}\in\R^{n\times n}$ to denote the matrix with entries
\begin{equation}\label{m ij definition}
    m_{i,j}(\vec{x})=\begin{cases}
        R(x_{i}),&\ \text{~for~} i=j,\\
        -G(x_{i},x_{j}),&\ \text{~for~} i\neq j.
    \end{cases}
\end{equation}
From \cite[Appendix A]{bahri1988}, we know that it's lowest eigenvalue $\rho(\vec x)$ is simple and the corresponding eigenvector can be chosen to have strictly positive components. Here, we use $\Lambda(\vec x)\in\R^{n}$ to denote the unique eigenvector corresponding to $\rho(\vec x)$ such that 
\begin{equation}\label{defination of Lamda x}
 M(\vec x)\cdot \Lambda(\vec x)=\rho(\vec x)\cdot\Lambda(\vec x),\quad (\Lambda(\vec x))_{1}=1.    
\end{equation}
Finally, for given $n\in\N$ and $q\in (\max\{2,\frac{N}{N-2}\},2^*)$, we define the function $ \Phi_{n}$ as
\begin{equation}\label{reduce function}
     \Phi_{n}(\vec{x},\vec{\lambda}):=\frac{A^{2}}{2}\langle \vec{\lambda}^{\frac{N-2}{2}}, {M}(\vec{x})\vec{\lambda}^{\frac{N-2}{2}}\rangle-\frac{B}{q}\sum_{j=1}^{n} \lambda_{j}^{\frac{N-2}{2}(2^*-q)},\quad (\vec{x},\vec{\lambda})\in \Omega_{*}^{n}\times(0,\infty)^{n} ,
 \end{equation}
where
\begin{equation}\label{definition of A and B}
    A=\int_{\R^{N}}U_{0,1}^{2^*-1}dx=\frac{\alpha_{N}^{2^*-1}\omega_{N}}{N},\quad B=\int_{\R^{N}}U_{0,1}^{q}dx=\alpha_{N}^{q}\omega_{N}\frac{\Gamma(\frac{N}{2})\Gamma(\frac{N-2}{2}q-\frac{N}{2})}{2\Gamma(\frac{N-2}{2}q)},
\end{equation}
and $\Gamma(\cdot)$ is the Gamma function. 

In the following, we state the main results in this paper. The first theorem gives a prior estimate of energy and describes the multiple bubbling phenomena of solutions to \eqref{p-varepsion}.

\begin{Thm}\label{thm multibble blowup}
Assume that $N\ge 3, q\in (\max\{2,\frac{4}{N-2}\},2^*)$ and $\Omega$ is a smooth and bounded domain in $\R^{N}$. Let $u_{\varepsilon}$ be a family of solutions of \eqref{p-varepsion}, then there exists two constants $\varepsilon_{0},M_{0}>0$ which depend only on $\Omega,N$ and $q$, such that 
\begin{equation}
    {M_{0}}^{-1}\leq\|u_{\varepsilon}\|_{H^{1}_{0}(\Omega)}\leq M_{0},
\end{equation}
for any $\varepsilon\in (0,\varepsilon_{0})$. Moreover, we have
\begin{enumerate}
    \item If $\Vert u_\varepsilon\Vert_{L^{\infty}(\Omega)} $ is bounded, then, after passing to a subsequence, 
\begin{equation}
    u_{\varepsilon}\to u_{0}\text{~~in~~}C^{2}(\bar{\Omega})\text{~~as~~}\varepsilon\to0,
\end{equation}
where $u_{0}$ is a solution of \eqref{p-0}.
\item If $\Vert u_\varepsilon\Vert_{L^{\infty}(\Omega)}$ is unbounded, then, after passing to a subsequence, 
\begin{equation}
    u_{\varepsilon}\rightharpoonup0\text{~~weakly and not strongly in~~} H^{1}_{0}(\Omega)\text{~~as~~}\varepsilon\to0,
\end{equation}
and there exists $n \in \N$ and $n$ points $x_{1,\varepsilon},..., x_{n,\varepsilon}$ such that the following holds: 
\begin{enumerate}
\item $x_{j,\varepsilon} \to a_{j} \in \Omega$ as $\varepsilon\to 0$, $\lambda_{j,\varepsilon} := u_\varepsilon(x_{j,\varepsilon})^{\frac{2}{N-2}} \to \infty$ as $\varepsilon \to 0$ and $\nabla u_\varepsilon(x_{j,\varepsilon}) = 0$ for every $j$.

\item There exists a constant $\delta_{0}>0$ such that $\min_{j\neq l}\{\text{dist}(a_{j},a_{l}),\text{dist}(a_{j},\partial\Omega)\}\geq \delta_{0}$.

\item $\lambda_{j,\varepsilon}^{-\frac{N-2}{2}} u_\varepsilon(x_{j,\varepsilon} + \lambda_{j,\varepsilon}^{-1} x) \to U$ in $C^2_\text{loc}(\R^N)\cap H^{1}_{\text{loc}}(\R^{N})$ as $\varepsilon\to0$ for every $j$.

\item $|\nabla u_{\varepsilon}|^{2}\weakto S^{\frac{N}{2}}\sum_{j=1}^{n}\delta_{a_{j}}$ in the sense of measure as $\varepsilon\to0$.

\item There exists a constant $C_{0}>0$ such that $u_{\varepsilon} \leq C_{0} \sum_{j=1}^{n} U_{x_{j,\varepsilon},\lambda_{j,\varepsilon}}$ on $\Omega$, Moreover, for any compact set $K\subset\subset\Omega$, there exists a constant $C_{1}(K)>0$ such that $C_{1}\sum_{j=1}^{n} U_{x_{j,\varepsilon},\lambda_{j,\varepsilon}}\leq u_{\varepsilon} $ on $K$.

\item There exists a constant $C_{2}>0$ such that $\varepsilon\leq C_{2} \lambda_{j,\varepsilon}^{2-\frac{N-2}{2}q}$ for every $j$. 

\item $\mu_{j} := \lim\limits_{\varepsilon \to 0} \mu_{j,\varepsilon} := \lim\limits_{\varepsilon \to 0} (\frac{\lambda_{1,\varepsilon}}{\lambda_{j,\varepsilon}})^{\frac{N-2}{2}}$ exists and lies in $(0, \infty)$ for every $j$.

\item There exists some positive constants $\nu_{j}$ such that $\lambda_{1,\varepsilon}^{\frac{N-2}{2}} u_{\varepsilon} \to  \sum_{j=1}^{n} \nu_j G(a_j, \cdot):={G}_{\vec{a},\vec{\nu}}$ in $C^1_{loc}(\Omega\setminus\{a_{1},...,a_{n}\})$ as $\varepsilon\to0$.

\end{enumerate}
\end{enumerate}
\end{Thm}

\begin{Rem}
\begin{enumerate}  
    \item For $N\geq 5, q=2$, the results presented in Theorem \ref{thm multibble blowup} remain valid, as demonstrated by Cerqueti-Grossi \cite{Cerqueti2001Localestimates}. And the main idea in Theorem \ref{thm multibble blowup} is the classical blow-up analysis and local estimates around the blow-up points, which originated from the work of Schoen \cite{Schoenbook,Schoen} and Li \cite{YanYanLi1995,YanYanLi1996,yanyanli1999CCM} for the prescribed scalar curvature problem as well as the Yamabe problem. Futhermore, Druet-Laurain \cite{Druet2010StabilityOT} investigated the global asymptotic behavior of the Brezis-Nirenberg problem with $N=3,q=2$ and K\"onig-Laurain \cite{Knig2022FineMA} extended their results to the case $N\geq 4,q=2$ recently. Their method dates back to the work of Druet-Hebey-Robert \cite{Druetbook}. 
    \item Theorem \ref{thm multibble blowup} shows that the blow-up points of $u_{\varepsilon}$ are isolated, simple and depart each other, thus the number of blow-up points of $u_{\varepsilon}$ is prior finite and clustering configurations can not occur. But, it is not always the case for sign-changing solutions. To be more precise, in Premoselli \cite{Premoselli2022Tower} (see also Iacopetti-Vaira \cite{Iacopetti2016}), it was proved that the sign-changing tower of bubble for the Brezis-Nirenberg problem with $N\geq 7, q=2$ exist. In contrast, in low dimensions $N=4,5,6$ and $q=2$, sign-changing bubble-towers cannot exist as shown in Iacopetti-Pacella \cite{Iacopetti2015}. On the other hand, Musso-Rocci-Vaira \cite{Musso2024} (see also Vaira \cite{Vaira2015}) construct a family of sign-changing solutions clustering around a point on the boundary of $\Omega$ for the Brezis-Nirenberg problem with $N\geq 7,q=2$ recently. But it's still unknown that if this kind of clustering sign-changing solutions exist for the Brezis-Nirenberg problem in low dimension $N=4,5,6$ and $q=2$. 
 \item It is worth noting that the property $(\textrm{f})$ in Theorem \ref{thm multibble blowup} gives a prior estimate on blow-up rate and it implies that the restriction on $N$ and $q$ in Theorem \ref{thm multibble blowup} is optimal. Indeed, when $N=4,q=2$, the property $(\textrm{f})$ is trivial and it should be correct as $\varepsilon\lesssim(\ln\lambda_{j,\varepsilon})^{-1}$, see K\"onig-Laurain \cite{Knig2022FineMA}. Furthermore, we can prove that $\varepsilon\sim\lambda_{j,\varepsilon}^{2-\frac{N-2}{2}q}$ under a nondegenerate assumption of the concentration point $\vec{a}$, see Theorem \ref{thm blow up rate} below for more details.
\end{enumerate}
\end{Rem}

Next, we impose some assumption on domain $\Omega$.
\smallskip

\noindent\textbf{Assumption A:} The limit equation \eqref{p-0} has no solutions.
\begin{Rem}
    \begin{enumerate}
        \item  If $\Omega$ is a star-shaped domain, then from Poho\v{z}aev identity, we know that the assumption A holds. In addition, we can also find some non star-shaped domain such that the assumption A holds, see \cite{Carpio1992}.
        \item From the results in Theorem \ref{thm multibble blowup}, Assumption A implies that the solutions of \eqref{p-varepsion} must blow-up. 
    \end{enumerate}
\end{Rem}

The second theorem gives a decomposition of the solutions to \eqref{p-varepsion}.

\begin{Thm}\label{thm decomposition}
Assume that $N\ge 3, q\in (\max\{2,\frac{4}{N-2}\},2^*)$, $\Omega$ is a smooth bounded domain and satisfies Assumption A. Let $u_{\varepsilon}$ be a family of solutions of \eqref{p-varepsion}, then there exists $n\in\N$ and a vector $\vec{a}=(a_{1},\cdots,a_{n})\in\Omega_{*}^{n}$ such that $u_{\varepsilon}(x)$  has the following decomposition
\begin{equation}\label{eq decomposition}
u_{\varepsilon}=\sum^n_{j=1} PU_{x_{j,\varepsilon}, \lambda_{j,\varepsilon}}+w_{\varepsilon},
\end{equation}
and for any $j=1,\cdots,n$, as $\varepsilon\to0$
\begin{equation*}
~ x_{j,\varepsilon}\rightarrow a_j, ~\lambda_{j,\varepsilon}\rightarrow \infty,~\|w_{\varepsilon}\|_{H^{1}_{0}(\Omega)}\to0,~w_\varepsilon\in \bigcap^n_{j=1}E_{x_{j,\varepsilon},\lambda_{j,\varepsilon}}.
\end{equation*}
Moreover, we have
 \begin{equation}\label{perturbation estimate of w-varepsilon}
  \|w_{\varepsilon}\|_{H^{1}_{0}(\Omega)}=\begin{cases}
     O\left(\frac{1}{\lambda_{\varepsilon}}\right), &\text{~~if~~} N=3\text{~and~} q\in (4,6),\\

         O\left(\frac{1}{\lambda_{\varepsilon}^{2q-3}}\right), &\text{~~if~~} N=4 \text{~and~} q\in (2,\frac{5}{2}),\\
          O\left(\frac{(\ln\lambda_{\varepsilon})^{\frac{3}{4}}}{\lambda_{\varepsilon}^{2}}\right), &\text{~~if~~} N=4 \text{~and~} q=\frac{5}{2},\\
     O\left(\frac{1}{\lambda_{\varepsilon}^{2}}\right), &\text{~~if~~} N=4 \text{~and~} q\in (\frac{5}{2},4),\\
     
      O\left(\frac{1}{\lambda_{\varepsilon}^{\frac{6q-7}{2}}}\right), &\text{~~if~~} N=5 \text{~and~} q\in (2,\frac{13}{6}),\\

O\left(\frac{(\ln\lambda_{\varepsilon})^{\frac{7}{10}}}{\lambda_{\varepsilon}^{3}}\right), &\text{~~if~~} N=5 \text{~and~} q=\frac{13}{6},\\
      
      O\left(\frac{1}{\lambda_{\varepsilon}^{3}}\right), &\text{~~if~~} N=5 \text{~and~} q\in (\frac{13}{6},\frac{10}{3}),\\
O\left(\frac{(\ln\lambda_{\varepsilon})^{\frac{N+2}{2N}}}{\lambda_{\varepsilon}^{\frac{N+2}{2}}}\right), &\text{~~if~~} N\geq 6\text{~and~} q\in (2,2^*),
 \end{cases}
\end{equation}
where ${\lambda}_{\varepsilon}:=\min\{\lambda_{1,\varepsilon},\cdots,\lambda_{n,\varepsilon}\}$.
\end{Thm}

The third theorem gives the exact rate of blow-up and characterize the concentration speed and the location of the concentration point. 
\begin{Thm}\label{thm blow up rate}
Assume that $N\ge 3, q\in (\max\{2,\frac{4}{N-2}\},2^*)$, $\Omega$ is a smooth bounded domain and satisfies Assumption $\textrm{A}$. Using the notations in Theorem $\ref{thm decomposition}$ , we can obtain that $\rho(\vec{a})\geq 0$ and hence $M(\vec{a})$ is a nonnegative matrix. Moreover, let
\begin{equation}\label{definition of F-N-Q-lambda}
    \mathrm{F_{N,q}}(\lambda_{\varepsilon})=\begin{cases}
    \frac{1}{\lambda_{\varepsilon}}, &\text{~~if~~} N=3\text{~and~} q\in (4,6),\\
         \frac{1}{\lambda_{\varepsilon}^{2q-4}}, &\text{~~if~~} N=4 \text{~and~} q\in(2,3),\\
          \frac{\ln\lambda_{\varepsilon}}{\lambda_{\varepsilon}^{2}}, &\text{~~if~~} N=4 \text{~and~} q\in[3,4),\\ 
      \frac{1}{\lambda_{\varepsilon}^{3q-5}}, &\text{~~if~~} N=5\text{~and~} q\in (2,\frac{7}{3}),\\
    \frac{1}{\lambda_{\varepsilon}^{2}}, &\text{~~if~~} N=5\text{~and~} q\in [\frac{7}{3},\frac{10}{3}),\\
    \frac{1}{\lambda_{\varepsilon}^{2}}, &\text{~~if~~} N\geq 6\text{~and~} q\in (2,2^*),
 \end{cases} 
\end{equation}
then we have
\begin{enumerate}
    \item If $\rho(\vec{a})> 0$, then $M(\vec{a})$ is a positive matrix and after passing to a subsequence, 
\begin{equation}
\lim_{\varepsilon\rightarrow 0}\varepsilon\lambda_{j,\varepsilon}^{\frac{N-2}{2}q-2}:=\lambda_{j}^{2-\frac{N-2}{2}q},~\mbox{for}~j=1,\cdots,n,
\end{equation}
exists and lies in $(0,\infty)$ and $(\vec{a},\vec{\lambda})$ is a  critical point of $\Phi_{n}$. Moreover if $q\geq 2^*-1$, then $\vec{\lambda}$ is the unique critical point of $\Phi_{n}(\vec{a},\cdot)$  which is also a minimizer. Furthermore, if $(\vec{a},\vec{\lambda})$ is a nondegenerate critical point of $\Phi_{n}$, then
for $j=1,\cdots,n$, 
\begin{equation}\label{eq concentration speed}
\big|x_{j,\varepsilon}-a_j\big|=O(\mathrm{F_{N,q}}(\lambda_{\varepsilon}))\text{~~and~~}\big|\lambda_j- \big(\varepsilon ^{\frac{2}{(N-2)q-4}}\lambda_{j,\varepsilon}\big)^{-1}\big|=O(\mathrm{F_{N,q}}(\lambda_{\varepsilon})).
 \end{equation}
\item If $\rho(\vec{a})=0$, then $\nabla \rho(\vec{a})=0$. Moreover,  as $\varepsilon\rightarrow 0$
\begin{equation}
\varepsilon\lambda_{j,\varepsilon}^{\frac{N-2}{2}q-2}=O(\mathrm{F_{N,q}}(\lambda_{\varepsilon})),
\end{equation}
and $(\Lambda(\vec{a}))_{j}=\mu_{j}=\lim_{\varepsilon\to0}(\frac{\lambda_{1,\varepsilon}}{\lambda_{j,\varepsilon}})^{\frac{N-2}{2}}$. Furthermore, as $\varepsilon\rightarrow 0$, we have the following quantitative bounds
\begin{equation}
 \rho(\vec{x}_\varepsilon)=o(\mathrm{F_{N,q}(\lambda_\varepsilon)}),
\end{equation}
and
\begin{equation}
   |\nabla \rho(x_\varepsilon)|=O(\mathrm{F_{N,q}}(\lambda_{\varepsilon})).
\end{equation}
\end{enumerate}
\end{Thm}

\begin{Rem}
    \begin{enumerate}
        \item If $u_{\varepsilon}$ is a family of energy minimizing solutions, then $u_{\varepsilon}$ must blow-up and there exists only one concentration point $x_{0}\in\Omega$ satisfying $x_{\varepsilon}\to x_{0}$, where $u_{\varepsilon}(x_{\varepsilon})=\|u_{\varepsilon}\|_{L^{\infty}(\Omega)}$. Furthermore, $\rho({x}_{0})={M} ({x}_{0})=R({x}_{0})>0$ and only the  case 1 in Theorem \ref{thm blow up rate} will occur.
        \item If $\Omega$ is a convex domain, then Poho\v{z}aev identity and Theorem \ref{thm multibble blowup} implies that $u_{\varepsilon}$ must blow-up. Moreover, from Theorem \ref{strictly convex of Robin function} and Theorem \ref{nonexistence of multi-peak solution} in Appendix, we know that there is no multi-peak blow-up solutions, thus $u_{\varepsilon}$ has only one concentration point and the case 1 in Theorem \ref{thm blow up rate} will happen. On the other hand, if $\Omega$ is star-shaped and multiple blow-up occur, then by Poho\v{z}aev identity (see Appendix \eqref{pohozaev identity 3}), we have
        \begin{equation}\label{starshape estimate 1}
            \lambda_{1,\varepsilon}^{N-2}\int_{\partial \Omega}
|\nabla  u_{\varepsilon}|^2
\big\langle x-x_{0},\nu\big\rangle=\left(\frac{2N}{q}-\frac{2N}{2^{*}}\right)\varepsilon\lambda_{1,\varepsilon}^{N-2}\int_{\Omega}u_{\varepsilon}^{q}.
        \end{equation}
        From Theorem \ref{thm multibble blowup}, we know that the left-hand side converges to
        \begin{equation}\label{starshape estimate 2}
            \int_{\partial \Omega}
|\nabla  {G}_{\vec{a},\vec{\nu}}|^2
\big\langle x-x_{0},\nu\big\rangle>0.
        \end{equation}
        On the other hand, from Lemma \ref{lem estimate of q refined} and the proof of Proposition \ref{Prop blow up rate}, the right-hand side is equal to
        \begin{equation}\label{starshape estimate 3}
            \varepsilon\lambda_{1,\varepsilon}^{N-2}\left(\frac{2N}{q}-\frac{2N}{2^{*}}\right)B\sum_{j=1}^{n}\lambda_{j,\varepsilon}^{\frac{N-2}{2}q-N}+o(1)\sim \varepsilon\lambda_{\varepsilon}^{\frac{N-2}{2}q-2}+o(1).
        \end{equation}
        Combining \eqref{starshape estimate 2} and \eqref{starshape estimate 3} together, we deduce that $\lim_{\varepsilon\to 0}\varepsilon\lambda_{\varepsilon}^{\frac{N-2}{2}q-2}>0$. Hence $\rho(\vec{a}_{0})>0$ and the case 1 in Theorem \ref{thm blow up rate} occur. 
        \item If $\rho(\vec{a})>0$ and $q\geq 2^*-1$, then $\vec{\lambda}$ can be determined by concentration point $\vec{a}$, thus the asymptotic behavior of blow-up solutions $u_{\varepsilon}$ is uniquely determined by $N,\Omega,q$ and concentrate point $\vec{a}$.
    \end{enumerate}
\end{Rem}

Another profound problem is to investigate the uniqueness and nondegeneracy of the solutions of \eqref{p-varepsion}, which is more delicate and is known depending on the dimension $N$ and the geometry of domain $\Omega$. In the following, we state some known results established in the literature. 


When $\Omega$ is symmetric with respect to the coordinate hyperplanes ${x_{i}=0}$ and convex in the $x_{i}$-directions for any $i=1,\cdots,N$, Cerqueti-Grossi \cite{Cerqueti1999AUR,Cerqueti2001Localestimates} proved that the energy minimizing solutions of \eqref{p-varepsion} is unique and nondegenerate if $N\geq 5, q=2$ and $\varepsilon>0$ small enough. On the other hand, when $N\geq 5, q=2$ and $\Omega$ is a general domain, Glangetas \cite{Glangetas1993UniquenessOP} proved that if $u_{\varepsilon}$ and $v_{\varepsilon}$ are two solutions of \eqref{p-varepsion} concentrating at the same point $x_{0}$ and $x_{0}$ is a nondegenerate critical point of the Robin function $R(x)$, then $u_{\varepsilon}\equiv v_{\varepsilon}$ for any $\varepsilon>0$ small enough. It is worth noting that the solutions of \eqref{p-varepsion} has at most one blow-up point if $\Omega$ is convex. On the other hand, the method used in \cite{Glangetas1993UniquenessOP} is to reduce into a finite dimension problem and count the local degree, but it's extremely complicated to calculate the degree in the case of multi-peak solutions. Hence the uniqueness and nondegeneracy of the multi-peak solutions is not easy to get. Fortunately, by using various local Poho\v{z}aev identities and blow-up analysis, Cao et al. \cite{Cao2021Trans} and Chen et al. \cite{chen2024} proved the uniqueness and nondegeneracy of multi-peak solutions to \eqref{p-varepsion} with $N\geq 6,q=2$ recently. 

In the following, we generalize the above results to the sub-critical perturbation case $q\in(2,2^*)$ and our main results read as follows.

\begin{Thm}\label{thm uniqueness}
Let $N\ge 5$, $q\in (2,2^*)$, $q\geq 2^*-1$, $\Omega$ is a smooth bounded domain and satisfies Assumption $\textrm{A}$. If $u_{\varepsilon}$ and $v_{\varepsilon}$ are two solutions of \eqref{p-varepsion} which concentrate at the same point $(\vec{a},\vec{\lambda})$, $M(\vec{a})$ is a positive matrix and $(\vec{a},\vec{\lambda})$ is a nondegenerate critical point of $\Phi_{n}$. 
Then, there exists a constant $\varepsilon_{0}>0$ such that for any $\varepsilon<\varepsilon_{0}$
\begin{equation}
    u_{\varepsilon}\equiv v_{\varepsilon}\quad \text{~in~}\Omega.
\end{equation}
\end{Thm}

From Theorem \ref{thm multibble blowup}, we know that the number of blow-up point of solutions \eqref{p-varepsion} is prior finite. Thus we can denote the largest number of blow-up point by $n_0$. In addition, for any $n=1,\cdots,n_{0}$ we also define a set $T_{n}$ as follows
\begin{equation}
    \begin{aligned}
        T_n= \Big\{(\vec{a},\vec{\lambda})\in\Omega^{n}_{*}\times(0,\infty)^{n}: (\vec{a},\vec{\lambda}) \text{~is a critical point of~}\Phi_{n}\Big\},
    \end{aligned}
\end{equation}
and we give another assumption on domain $\Omega$.
\smallskip

\noindent\textbf{Assumption B:} $M(\vec{x})$ is a positive matrix and $(\vec{a},\vec{\lambda})$ is a nondegenerate critical point of $\Phi_{n}$ for any $\vec{x}\in\Omega_{*}^{n}$, $(\vec{a},\vec{\lambda})\in T_{n}$ and $n=1,\cdots,n_{0}$.
\begin{Cor}\label{thm number}
Let $N\ge 5$, $q\in (2,2^*)$, $q\geq 2^*-1$, $\Omega$ is a smooth bounded domain satisfying Assumption $\textrm{A}$ and Assumption $\textrm{B}$. Then for $\varepsilon>0$ sufficiently small,
$$
\mbox{the number of solutions to \eqref{p-varepsion}}~~=\displaystyle\sum^{n_0}_{n=1}\sharp T_n,$$
where $\sharp T_n$ is the number of the elements in the set $T_n$.
\end{Cor}

Finally, we can prove the nondegeneracy by using a similar proof as in Theorem \ref{thm uniqueness} and we omit the details.
\begin{Thm}\label{thm nondegeneracy}
Let $N\ge 5$, $q\in (2,2^*)$, $q\geq 2^*-1$, $\Omega$ is a smooth bounded domain and satisfies Assumption $\textrm{A}$. If $u_{\varepsilon}$ is a family of solutions to \eqref{p-varepsion} which concentrate at the point $(\vec{a},\vec{\lambda})$ with $M(\vec{a})$ is a positive matrix and $(\vec{a},\vec{\lambda})$ is a nondegenerate critical point of $\Phi_{n}$. Then, there exists a constant $\varepsilon_{0}>0$ such that for any $\varepsilon<\varepsilon_{0}$, the solution $u_{\varepsilon}$ is nondegenerate, i.e. the linear problem
    \begin{equation}\label{nondegenerate-1}
        \begin{cases}
            -\Delta v=(2^*-1)u_{\varepsilon}^{2^*-2}v+\varepsilon (q-1)u_{\varepsilon}^{q-2}v, &\text{~~in~~}\Omega,\\
            \quad \  \ v=0,&\text{~~on~~}\partial\Omega,
        \end{cases}
    \end{equation}
    admits only the trivial solution $v\equiv0$.
\end{Thm}
 \begin{Rem}
     \begin{enumerate}
         \item From Theorem \ref{thm blow up rate}, we know that $\Phi_{n}(\vec{a},\cdot)$ has an unique critical point in $(0,\infty)^{n}$ if $q\geq 2^*-1$, thus the uniqueness, nondegeneracy and the number of solutions to \eqref{p-varepsion} only depend on dimension $N$, domain $\Omega$ and concentration point $\vec{a}$.
         \item If $\Omega$ does not satisfy the assumption A, then we can correct the result in Corollary \ref{thm number} into $$\mbox{the number of blow-up solutions to \eqref{p-varepsion}}~~=\displaystyle\sum^{n_0}_{n=1}\sharp T_n.$$
         \item Note that if $\Omega$ is convex, then the multi-bubble blow-up can't occur, thus it's natural to consider whether or not there exists some non convex domain $\Omega$ such that Assumption $\textrm{A}$ holds and $\Phi_{n}$ possesses some critical point which are nondegenerate. This question seems to be difficult, since the properties of the Green's function on a general domain is unclear. We refer to \cite{Msusso2002,Micheletti2014,Bartsch2019} and reference therein for some known results to this question.
          \item Although the existence and asymptotic profile of multi-peak solutions to \eqref{p-varepsion} is clear for $N\geq 3,q\in(\max\{2,\frac{4}{N-2}\},2^*)$.  The uniqueness and nondegeneracy of multi-peak solutions to \eqref{p-varepsion} is still open for $N=3,4,5$ and $q\in(\max\{2,\frac{4}{N-2}\},\frac{N+2}{N-2})$. Indeed, this is due to the estimate of the remainder term $w_{\varepsilon}$ in Theorem \ref{thm decomposition} and concentrate speed \eqref{eq concentration speed} in Theorem \ref{thm blow up rate} is not enough for us to get a contradiction in Proposition \ref{prop c j 0} and Proposition \ref{prop estimate for c ji}, see Section \ref{section local uniqueness} for more details. 
         \item  Since the proof of Theorem \ref{thm uniqueness} and Theorem \ref{thm nondegeneracy} are by contradiction, thus we can not quantify precisely the size of $\varepsilon_{0}$. However, when $\Omega$ is ball, Erbe-Tang \cite{Erbe1997UniquenessTF} proved that \eqref{p-varepsion} has a unique solution for $N\geq 6,q\in[2,2^*)$ and every $\varepsilon>0$. But the uniqueness problem in low dimension $N=3,4,5$ seems unknown and we refer to \cite{Chen2012} for the uniqueness problem of the least energy solution in low dimension. 
\end{enumerate} 
\end{Rem}

\begin{Rem}
Compared with the previous work, there are some features of this paper:
\begin{enumerate}
    \item This paper gives a complete picture of blow-up phenomena in the framework of Brezis-Peletier conjecture in the case of sub-critical perturbation and multi-peak. Moreover, combining the results obtained by Cao-Luo-Peng \cite{Cao2021Trans} and K\"{o}nig-Laurain \cite{Knig2022FineMA,Knig2022MultibubbleBA}, the asymptotic behavior of multi-peak positive solutions to Brezis-Nirenberg problem is almost clear, except the case $N=3$ and $q\in(2,4]$. Furthermore, we also point out that the concave and convex case $1<q<2$ and super-critical case $q>2^*$ has also been studied by Ambrosetti-Brezis-Cerami \cite{Ambrosetti1994} and Merle-Peletier \cite{Merle1990,Merle1992} respectively.
    \item This paper gives a clear expression of the dependence between the subcritical exponent $q$ with the error estimate in decomposition \eqref{eq decomposition}, blow-up rate, concentration speed and location of the concentration point in the general case of multi-peak solutions. Indeed, it can be observed that the error estimate and blow-up rate are almost decreasing, while the concentration speed is almost increasing with respect to $N$ and $q$. On the other hand, we point out that $\frac{4}{N-2}$ and $\frac{N+2}{N-2}$ are two special sub-critical exponents which have significant impact on the existence of positive solutions (see \cite{Brezis1983CPAM,AtkinsonPeletier}) and sign-changing solutions (see \cite{Jones1988,Pistoia2004}) respectively.
    \item It is worth noting that the function $\Phi_{n}$ plays a crucial role in the analysis of asymptotic behavior of multi-peak solutions to \eqref{p-varepsion}. In particular, when $N=4,q=2$, the constant $B$ in the definition of $\Phi_{n}$ is invalid and it should be correct by another constant $\kappa$ which is uniquely determined by the concentration point, see \cite[equation (1.5)]{Knig2022FineMA}. On the other hand, we can also compare the results obtained in this paper with the analysis of the slightly sub-critical Lane-Emden equation
\begin{equation}
\begin{cases}
-\Delta u= u^{2^*-1-\varepsilon}, \quad u>0, &{\text{in}~\Omega},\\
\quad\ \ u=0, &{\text{on}~\partial \Omega},
\end{cases}
\end{equation}
in which the  corresponding function is defined by
 \begin{equation}
     \Phi_{n}(\vec{x},\vec{\lambda}):=\frac{A^{2}}{2}\langle \vec{\lambda}^{\frac{N-2}{2}}, {M}(\vec{x})\vec{\lambda}^{\frac{N-2}{2}}\rangle-\frac{N-2}{2}\sum_{j=1}^{n}\ln\lambda_{j},
 \end{equation}   
 and we refer to \cite{Bahri1995,ReyOlivier1999,LiHouwang2023} for the asymptotic profile, uniqueness and nondegeneracy of the multi-pike solutions to this near critical problem.
\end{enumerate}
\end{Rem}

Our paper is organized as follows. In Section \ref{section Multiple bubbling phenomena}, we describe the multi-bubble phenomenon of solutions $u_{\varepsilon}$ by modifying the argument of Li \cite{YanYanLi1995} and Cerqueti-Grossi \cite{Cerqueti2001Localestimates}. As a consequence, we show Theorem \ref{thm multibble blowup}. In Section \ref{section Decomposition of the solutions}, we prove the Theorem \ref{thm decomposition}, which give a refined decomposition of the solutions $u_{\varepsilon}$. Then combining various local Poho\v{z}aev identities, we establish the Theorem \ref{thm blow up rate} and Theorem \ref{thm uniqueness} in Section \ref{Blow-up rate and location of the concentration point} and Section \ref{section local uniqueness} respectively. Finally, we list some known facts and crucial estimates in the Appendix. 

\smallskip

\noindent\textbf{Notations.}
Throughout this paper, we use the following notations.
\begin{enumerate}
\item For any given $N\in\N$, we use $B(x,r)$ or $B_{r}(x)$ to denote the open ball with center $x\in\R^{N}$ and radius $r>0$. 
    \item We use $\mathcal{D}^{1,2}(\R^{N}):=\left\{u\in L^{2^*}(\R^{N}):~\nabla u\in L^{2}(\R^{N})\right\}$ to denote the homogeneous Sobolev space.
    In addition, we use $\|u\|_{H^{1}_{0}(\Omega)}=\left(\int_{\Omega}|\nabla u|^{2}dx\right)^{1/2}$ to denote the norm in $H^{1}_{0}(\Omega)$ and $\langle\cdot,\cdot\rangle$ means the corresponding inner product.  
    \item We use $C$ to denote various positive constant and use $C_{1}=o(\varepsilon)$ and $C_{2}=O(\varepsilon)$ to denote $C_{1}/\varepsilon\to0$ and $|C_{2}/\varepsilon|\leq C $ as $\varepsilon\to0$ respectively.
     \item  Let $f,g: X\to \R^{+}$ be two nonnegative function defined on some set $X$. we write $f\lesssim g$ or $g\gtrsim f$, if there exists a constant $C>0$ independent on $x$ such that $f(x)\leq C g(x)$ for any $x\in X$ and $f\sim g$  means that $f\lesssim g$ and $g\lesssim f$.
     \item We use $\partial_{i}$ and $\nabla$ to denote the partial derivative for any function $f(x, y): \R^n \times \R^m \to \R$ with respect to $x$ and use $D_{j}$ and $D$ to denote the partial derivative for with respect to $y$.
\end{enumerate}

\section{Multiple bubbling phenomena}\label{section Multiple bubbling phenomena}


\begin{Prop}\label{prop below bound of u}
There exists two small numbers $\varepsilon_{0}>0$ and $M_{0}>0$ independent of $\varepsilon$ such that for any $\varepsilon\in(0,\varepsilon_{0})$
\begin{equation}
    \|u_{\varepsilon}\|_{H^{1}_{0}(\Omega)}\geq M_{0}. 
\end{equation}
\end{Prop}
\begin{proof}
By contradiction, suppose that there exist sequence $\varepsilon_{n}\to0$ such that $\|u_{\varepsilon_{n}}\|_{H^{1}_{0}(\Omega)}\to0$  as $n\to\infty$.  From the Sobolev inequality, H\"older inequality and equation \eqref{p-varepsion}
    \begin{equation}
    \begin{aligned}
        S\leq \frac{\int_{\Omega}|\nabla u_{\varepsilon_n}|^{2}}{(\int_{\Omega}u_{\varepsilon_n}^{2^*})^{2/2^*}}
        &=\frac{\int_{\Omega}u_{\varepsilon_n}^{2^*}+\varepsilon_n u_{\varepsilon_{n}}^{q}}{(\int_{\Omega}u_{\varepsilon_n}^{2^*})^{2/2^*}}\\
        &\leq (\int_{\Omega}u_{\varepsilon_n}^{2^*})^{(2^*-2)/2^*}+\varepsilon_n(\int_{\Omega}u_{\varepsilon_n}^{2^*})^{(q-2)/2^*}|\Omega|^{(2^*-q)/2^*}\\
        &\to0,\text{~~as~~}n\to\infty.
    \end{aligned}
    \end{equation}
    This make a contradiction with $S>0$.
\end{proof}
If $u_{\varepsilon}$ is bounded in $L^{\infty}(\Omega)$, then by standard elliptic regularity theory \cite{GilbargTrudinger}, there exists a  nonnegative function $u$ such that $u_{\varepsilon}\to u$ in $C^{2}(\bar{\Omega})$ and $u$ satisfies 
\begin{equation}
    \begin{cases}
-\Delta u= u^{2^*-1},\quad u\geq0, &{\text{in}~\Omega},\\
\quad\ \ u=0, &{\text{on}~\partial \Omega}.
\end{cases}
\end{equation}
Moreover, from Proposition \ref{prop below bound of u} and equation \eqref{p-varepsion}, we know that $u\not\equiv 0$, then by strong maximum principle $u>0$.

If $u_{\varepsilon}$ is unbounded in $L^{\infty}(\Omega)$, then $u_{\varepsilon}$ must blow-up at several points. The following proposition gives a prior estimate of solutions $u_{\varepsilon}$ around the boundary $\partial\Omega$, which shows that all blow-up points are away from the boundary.

\begin{Prop}\label{prop boundary estimate}
    There exist a small number $\delta_{1}>0$ and a large number $M_{1}>0$ independent of $\varepsilon$ such that for any $\varepsilon>0$ and any $x\in\bar{\Omega}$ with dist$(x,\partial\Omega)\leq \delta_{1}$
    \begin{equation}
        0\leq u_{\varepsilon}(x)\leq M_{1}.
    \end{equation}
\end{Prop}
\begin{proof}
    This proposition can be obtained by using the Kelvin transform, moving plane method \cite{Gidas1979Symmetry} and an interior integral estimate of the solutions $u_{\varepsilon}$, we refer \cite[Theorem 1.1, Theorem 1.2]{Figueiredo-Lions-Nussbaum1982} and \cite[pp.164]{Han1991} for the details of proof.
\end{proof}

Next, we extract sequences of critical points of $u_{\varepsilon}$, whose number may not be uniformly bounded on $\varepsilon$ and their convergence points are candidates to the blow-up points of $u_{\varepsilon}$.
\begin{Prop}\label{prop candidates of blow-up points}
    Let $u_{\varepsilon}$ be a family of solutions of \eqref{p-varepsion}. For any given $\eta>0$ small and $R>0$ large, there exist some positive constants $C_{0}$ and $C_{1}$ depending only on $\Omega,\eta$ and $R$ such that if $u_{\varepsilon}$ satisfies
    \begin{equation}
        \max_{x\in\Omega}u_{\varepsilon}> C_{0},
    \end{equation}
    then up to a subsequence, there exist a constant $n_{\varepsilon}\in\N$ and $n_{\varepsilon}$ points $\{x_{1,\varepsilon},\cdots,x_{n_{\varepsilon},\varepsilon}\}\subset\Omega$ such that
    \begin{enumerate}
        \item The set $\{x_{1,\varepsilon},\cdots,x_{n_{\varepsilon},\varepsilon}\}$ of local maximum of $u_{\varepsilon}$ such that 
        \begin{equation}\label{prop candidates of blow-up points-1}
            \{{B}(x_{i,\varepsilon},Ru_{\varepsilon}^{-\frac{2}{N-2}}(x_{i,\varepsilon}))\}_{i=1}^{n_{\varepsilon}},
        \end{equation} is a disjoint collection of subsets in $\Omega$.
        \item For any $i=1,\cdots,n_{\varepsilon}$, it holds that
        \begin{equation}\label{prop candidates of blow-up points-2}
            \|u_{\varepsilon}^{-1}(x_{i,\varepsilon})u_{\varepsilon}(u_{\varepsilon}^{-\frac{2}{N-2}}(x_{i,\varepsilon})\cdot+x_{i,\varepsilon})-U\|_{C^{2}({B}(0,2R))}<\eta.
        \end{equation}
        \item It holds that
        \begin{equation}\label{prop candidates of blow-up points-3}
            |x_{i,\varepsilon}-x_{j,\varepsilon}|^{\frac{N-2}{2}}u_{\varepsilon}(x_{i,\varepsilon})\geq C_{0}\text{~~for any~~}i> j,
        \end{equation}
        and
        \begin{equation}\label{prop candidates of blow-up points-4}
            \min_{i=1,\cdots,n_{\varepsilon}}|x-x_{i,\varepsilon}|^{\frac{N-2}{2}}u_{\varepsilon}(x)\leq C_{1}\text{~~for any~~}x\in\Omega.
        \end{equation}
    \end{enumerate}
\end{Prop}

\begin{proof}
We can prove this proposition by suitably modifying the argument in the proof of \cite[Proposition 5.1]{yanyanli1999CCM}.
\end{proof}

The following proposition shows that the blow-up points of $u_{\varepsilon}$ are isolated, simple and depart each other, thus the number of blow-up points of $u_{\varepsilon}$ is prior finite and tower of bubbles can not occur.

\begin{Prop}\label{prop ruling out bubble accumulations}
    Under the same notations as in Proposition \ref{prop candidates of blow-up points}, there exists $\varepsilon_{0},\delta_{0}>0$ independent of $\varepsilon$ such that
    \begin{equation}\label{prop ruling out bubble accumulations-1}
        \min_{i\neq j}|x_{i,\varepsilon}-x_{j,\varepsilon}|\geq \delta_{0},
    \end{equation}
    for any $\varepsilon\in (0,\varepsilon_{0})$.
\end{Prop}
\begin{proof}
The proof is similar to that of Proposition \cite[Proposition 4.2]{YanYanLi1995} and \cite[Claim B.6]{Knig2022FineMA}. By contradiction, suppose that there exists a sequence of $\varepsilon_{n}\to 0$ as $n\to\infty$ and a corresponding solutions $u_{\varepsilon_{n}}$ such that 
\begin{equation}\label{prop ruling out bubble accumulations-proof-1}
    \lim_{n\to\infty}\min_{i\neq j}|x_{i,\varepsilon_{n}}-x_{j,\varepsilon_{n}}|=0.
\end{equation}
Without loss of generality, we can assume that
\begin{equation}\label{prop ruling out bubble accumulations-proof-2}
    d_{\varepsilon_{n}}:=|x_{1,\varepsilon_{n}}-x_{2,\varepsilon_{n}}|=\min_{i\neq j}|x_{i,\varepsilon_{n}}-x_{j,\varepsilon_{n}}|\to 0.
\end{equation}
Then from \eqref{prop candidates of blow-up points-1}, we know that 
\begin{equation}\label{prop ruling out bubble accumulations-proof-3}
    d_{\varepsilon_{n}}\geq \max\{Ru_{\varepsilon_{n}}^{-\frac{2}{N-2}}(x_{1,\varepsilon_{n}}),Ru_{\varepsilon_{n}}^{-\frac{2}{N-2}}(x_{2,\varepsilon_{n}})\},
\end{equation}
thus $u_{\varepsilon_{n}}(x_{1,\varepsilon_{n}}),u_{\varepsilon_{n}}(x_{2,\varepsilon_{n}})\to\infty$ as $n\to\infty.$
Now, we define the rescaled function
\begin{equation}\label{prop ruling out bubble accumulations-proof-4}
    \tilde{u}_{\varepsilon_{n}}(x)=d_{\varepsilon_{n}}^{\frac{N-2}{2}}u_{\varepsilon_{n}}(x_{1,\varepsilon_{n}}+d_{\varepsilon_{n}}x),\quad x\in\Omega_{\varepsilon_{n}}:=\{x\in\R^{N}:x_{1,\varepsilon_{n}}+d_{\varepsilon_{n}}x\in\Omega\},
\end{equation}
Then $\tilde{u}_{\varepsilon_{n}}$ satisfies
\begin{equation}\label{prop ruling out bubble accumulations-proof-5}
    -\Delta \tilde{u}_{\varepsilon_{n}}=\tilde{u}_{\varepsilon_{n}}^{2^*-1}+\varepsilon_{n}d_{\varepsilon_{n}}^{\frac{N-2}{2}(2^*-q)}\tilde{u}_{\varepsilon_{n}}^{q-1},\text{~~in~~}\Omega_{\varepsilon_{n}}.
\end{equation}
and from Proposition \ref{prop boundary estimate}, we can deduce that $\Omega_{\varepsilon_{n}}\to\R^{N}$ as $n\to\infty$. 

Now, we define $\tilde{x}_{i,\varepsilon_{n}}=d_{\varepsilon_{n}}^{-1}(x_{i,\varepsilon_{n}}-x_{1,\varepsilon_{n}})$, then from the definition of $d_{\varepsilon_{n}} $, we get that $|\tilde{x}_{i,\varepsilon_{n}}-\tilde{x}_{j,\varepsilon_{n}}|\geq 1$ for any $i\neq j$.
For any $R>0$, let $\tilde{S}_{R,\varepsilon_{n}}=\{\tilde{x}_{i,\varepsilon_{n}}:\tilde{x}_{i,\varepsilon_{n}}\in B(0,R)\}$, after passing to a subsequence, we have $\tilde{S}_{R,\varepsilon_{n}}\to \tilde{S}_{R}$ as $n\to\infty$, where $\tilde{S}_{R}$ is a non-empty finite set. Then up to performing a diagonal extraction, we can define a at most countable set $\tilde{S}=\cup_{R>0}\tilde{S}_{R}$ and we have $|\tilde{x}_{i}-\tilde{x}_{j}|\geq 1$ for any $\tilde{x}_{i},\tilde{x}_{j}\in \tilde{S}$, with $i\neq j$.

Next, we claim that $\text{~for any~} i_{\varepsilon_{n}}=1,\cdots,n_{\varepsilon_{n}}, \text{~such that~} \text{d}(x_{i_{\varepsilon_{n}}},x_{1,\varepsilon_{n}})=O(d_{\varepsilon_{n}})$, it holds
\begin{equation}\label{prop ruling out bubble accumulations-proof-6}
   \tilde{u}_{\varepsilon_{n}}(\tilde{x}_{i_{\varepsilon_{n}}})\to\infty,\text{~as~}n\to\infty.
\end{equation}
Assume that there exist $i_{\varepsilon_{n}}$ such that $\text{d}(x_{i_{\varepsilon_{n}}},x_{1,\varepsilon_{n}})=O(d_{\varepsilon_{n}})$ and $\tilde{u}_{\varepsilon_{n}}(\tilde{x}_{i_{\varepsilon_{n}}})$ bounded. Then for any $j_{\varepsilon_{n}}$ such that $\text{d}(x_{j_{\varepsilon_{n}}},x_{1,\varepsilon_{n}})=O(d_{\varepsilon_{n}})$, we have that $\tilde{u}_{\varepsilon_{n}}(\tilde{x}_{j_{\varepsilon_{n}}})$ bounded. Indeed if there exists $j_{\varepsilon_{n}}$ such that $\text{d}(x_{j_{\varepsilon_{n}}},x_{1,\varepsilon_{n}})=O(d_{\varepsilon_{n}})$ and $\tilde{u}_{\varepsilon_{n}}(\tilde{x}_{j_{\varepsilon_{n}}})\to\infty,\text{~as~}n\to\infty.$ Then from Proposition \ref{prop candidates of blow-up points}, we know that $\tilde{x}_{j_{\varepsilon_{n}}}\to\tilde{x}_{j}$ is an isolate blow-up points of $\tilde{u}_{\varepsilon_{n}}$ and therefore according to the Proposition \ref{prop isolate blow-up point is simple}, $\tilde{x}_{j,\varepsilon_{n}}\to\tilde{x}_{j}$ is also an isolated simple blow-up point of $\tilde{u}_{\varepsilon_{n}}$. Moreover, we can obtain that there exist a small constant $\delta>0$ independent on $n$  such that 
\begin{equation}\label{prop ruling out bubble accumulations-proof-7}
\max_{B(\tilde{x}_{i,\varepsilon_{n}},\delta)}\tilde{u}_{\varepsilon_{n}}=O(1).
\end{equation}
Otherwise, we can choose $y_{\varepsilon_{n}}\in \overline{B(\tilde{x}_{i,\varepsilon_{n}},\delta)}$ such that $\tilde{u}_{\varepsilon_{n}}(y_{\varepsilon_{n}})=\max_{B(\tilde{x}_{i,\varepsilon_{n}},\delta)}\tilde{u}_{\varepsilon_{n}}$ and 
\begin{equation}\label{prop ruling out bubble accumulations-proof-8}
    \tilde{u}_{\varepsilon_{n}}(y_{\varepsilon_{n}})\to\infty\text{~as~}n\to\infty.
\end{equation}
This together with Proposition \ref{prop candidates of blow-up points} implies that $|y_{\varepsilon_{n}}-\tilde{x}_{i,\varepsilon_{n}}|=o(1)$. Next, we define $\mu_{\varepsilon_{n}}= \tilde{u}_{\varepsilon_{n}}^{-\frac{2}{N-2}} (y_{\varepsilon_{n}})$ and set
\begin{equation}\label{prop ruling out bubble accumulations-proof-9}
    \hat{u}_{\varepsilon_{n}}(x)=\mu_{\varepsilon_{n}}^{\frac{N-2}{2}}\tilde{u}_{\varepsilon_{n}}(y_{\varepsilon_{n}}+\mu_{\varepsilon_{n}}x),\text{~for any~} x\in B(0,\frac{\delta}{2\mu_{\varepsilon_{n}}}),
\end{equation}
Then $\hat{u}_{\varepsilon_{n}}(x)$ satisfies that
\begin{equation}\label{prop ruling out bubble accumulations-proof-10}
\begin{cases}
    -\Delta \hat{u}_{\varepsilon_{n}}=\hat{u}_{\varepsilon_{n}}^{2^*-1}+\varepsilon_{n}(\mu_{\varepsilon_{n}}d_{\varepsilon_{n}})^{\frac{N-2}{2}(2^*-q)}\hat{u}_{\varepsilon_{n}}^{q-1},\text{~in~} B(0,\frac{\delta}{2\mu_{\varepsilon_{n}}}),\\
    \hat{u}_{\varepsilon_{n}}(0)=\max_{B(0,\frac{\delta}{2\mu_{\varepsilon_{n}}})}\hat{u}_{\varepsilon_{n}}=1.
\end{cases}
\end{equation}
Note that $B(0,\frac{\delta}{2\mu_{\varepsilon_{n}}})\to\R^{N}$, then by standard elliptic regular theory \cite{GilbargTrudinger} and Caffarelli-Gidas-Spruck \cite{CGS}, $\hat{u}_{\varepsilon_{n}}\to U$ in $C^{2}_{loc}(\R^{N})$, where
\begin{equation}\label{prop ruling out bubble accumulations-proof-11}
    -\Delta U=U^{2^*-1},\text{~and~}0< U\leq 1=U(0).
\end{equation}
Note that from Proposition \ref{prop candidates of blow-up points}, $(\tilde{x}_{i,\varepsilon_{n}}-y_{\varepsilon_{n}})/\mu_{\varepsilon_{n}}$ is bounded, which together with \eqref{prop ruling out bubble accumulations-proof-10}-\eqref{prop ruling out bubble accumulations-proof-11} we have
\begin{equation}
    \liminf_{n\to\infty}\frac{\tilde{u}_{\varepsilon_{n}}(\tilde{x}_{i,\varepsilon_{n}})}{\tilde{u}_{\varepsilon_{n}}(y_{\varepsilon_{n}})}>0,
\end{equation}
this make a contraction with \eqref{prop ruling out bubble accumulations-proof-8} and $\tilde{u}_{\varepsilon_{n}}(\tilde{x}_{i,\varepsilon_{n}})=O(1)$, thus \eqref{prop ruling out bubble accumulations-proof-7} hold. Note that $\tilde{x}_{j,\varepsilon_{n}}\to\tilde{x}_{j}$ is  an isolated simple blow-up point of $\tilde{u}_{\varepsilon_{n}}$, then from Proposition \ref{prop properties of isolated simple blow-up point} and Harnack inequality \cite[Theorem 8.20]{GilbargTrudinger}, we conclude that $\tilde{u}_{\varepsilon_{n}}(\tilde{x}_{i,\varepsilon_{n}})\to 0$. This make a contraction with Proposition \ref{prop candidates of blow-up points}. Thus for any $j_{\varepsilon_{n}}$ such that $\text{d}(x_{j_{\varepsilon_{n}}},x_{1,\varepsilon_{n}})=O(d_{\varepsilon_{n}})$, we have that $\tilde{u}_{\varepsilon_{n}}(\tilde{x}_{j_{\varepsilon_{n}}})$ bounded. Then $\tilde{u}_{\varepsilon_{n}}$ is uniformly bounded in a neighbourhood of any finite subset of $\tilde{S}$. On the other hand, from Proposition \ref{prop candidates of blow-up points}, we know that $\tilde{u}_{\varepsilon_{n}}$ is uniformly bounded on any compact set of $\R^{N}\setminus\tilde{S}$. Thus $\tilde{u}_{\varepsilon_{n}}$ is uniformly bounded on any compact set in $\R^{N}$. Then by standard elliptic regularity theory \cite{GilbargTrudinger} and maximum principle, there exist a positive function $\tilde{U}$ such that $\tilde{u}_{\varepsilon_{n}}\to \tilde{U}$ in $C^{2}_{loc}(\R^{N})$ and
\begin{equation}
\begin{cases}
     -\Delta \tilde{U}=\tilde{U}^{2^*-1},\text{~in~}\R^{N},\\
     \quad~~\tilde{U}(0)\geq R^{\frac{N-2}{2}}>0.
\end{cases}
\end{equation}
Note that $\tilde{U}$ has at least two critical points namely $0,\tilde{x}_{2}$, while this make a contraction with the results of Caffarelli-Gidas-Spruck \cite{CGS}. This end the proof of \eqref{prop ruling out bubble accumulations-proof-6}.

Thus $\tilde{S}$ is the set of blow-up points of $\tilde{u}_{\varepsilon_{n}}$, Clearly, $0,\tilde{x}_{2}\in \tilde{S}$, where $\tilde{x}_{2}\in\mathbb{S}^{N-1}$ and $\tilde{x}_{2,\varepsilon_{n}}\to\tilde{x}_{2}$ as $n\to\infty$. Moreover, from Proposition \ref{prop candidates of blow-up points} the blow-up points are isolate and therefore according to Proposition \ref{prop isolate blow-up point is simple} the blow-up points are isolated and simple. Note that $\tilde{u}_{\varepsilon_{n}}(0) \tilde{u}_{\varepsilon_{n}}(x)$ is locally uniform bounded in $\R^{N}\setminus \tilde{S}$, multiplying \eqref{prop ruling out bubble accumulations-proof-5} by $\tilde{u}_{\varepsilon_{n}}(0)$, then from standard elliptic regularity theory, we have (up to a subsequence)
\begin{equation}\label{prop ruling out bubble accumulations-proof-14}
   \tilde{w}_{{\varepsilon_{n}}}=\tilde{u}_{\varepsilon_{n}}(0) \tilde{u}_{\varepsilon_{n}}(x)\to \tilde{h},\text{~~in~~}C^{2}_{loc}(\R^{N}\setminus\{\tilde{S}\}),\text{~~as~~}n\to\infty.
\end{equation}
Moreover, from Proposition \ref{prop properties of isolated simple blow-up point} and B$\hat{\textrm{o}}$cher's theorem \cite{Axler1992}
\begin{equation}
    \tilde{h}(x)=\frac{a_{1}}{|x|^{N-2}}+\frac{a_{2}}{|x-\tilde{x}_{2}|^{N-2}}+\tilde{b}(x),\text{~~in~~}C^{2}_{loc}(\R^{N}\setminus\{\tilde{S}\}),
\end{equation}
where $a_{1},a_{2}$ is a positive constants, $\tilde{b}$ is a regular harmonic function in $\R^{N}\setminus\{\tilde{S}\setminus\{0,\tilde{x}_{2}\}\}$. In addition, by maximum principle, we can conclude that $\tilde{b}(x)\geq 0$ for any $x\in \R^{N}\setminus\{\tilde{S}\setminus\{0,\tilde{x}_{2}\}\}$. Hence there exist constant $A>0$ such that
\begin{equation}\label{prop ruling out bubble accumulations-proof-16}
    \tilde{h}(x)=\frac{a_{1}}{|x|^{N-2}}+A+o(1),\text{~~for any~}|x| \text{~close to~}0.
\end{equation}

Finally, from \eqref{pohozaev identity 1} in Appendix, we can obtain the following Poho\v{z}aev identity of $\tilde{u}_{\varepsilon_{n}}$ in the ball $B(0,\rho)$ with $0<\rho\leq 1$ small enough
\begin{equation}\label{prop ruling out bubble accumulations-proof-17}
\begin{split}
&\varepsilon_{n}d_{\varepsilon_{n}}^{\frac{N-2}{2}(2^*-q)}\frac{2N-(N-2)q}{2q}\int_{ B(0,\rho)} \tilde{u}_{\varepsilon_{n}}^q\\
&=\rho \int_{\partial B(0,\rho)}\left(\frac{\partial\tilde{u}_{\varepsilon_{n}}}{\partial\nu}\right)^{2}
-\frac{\rho}{2}\int_{\partial B(0,\rho)}
|\nabla  \tilde{u}_{\varepsilon_{n}}|^2
+\frac{N-2}{2}\int_{\partial B(0,\rho)}\frac{\partial \tilde{u}_{\varepsilon_{n}}}{\partial\nu}
  \tilde{u}_{\varepsilon_{n}} \\
  &\quad+
  \frac{\rho}{2^*}\int_{\partial B(0,\rho)} \tilde{u}_{\varepsilon_{n}}^{2^*}+\varepsilon_{n}d_{\varepsilon_{n}}^{\frac{N-2}{2}(2^*-q)}\frac{\rho}{q
  }\int_{\partial B(0,\rho)}\tilde{u}_{\varepsilon_{n}}^{q}.
\end{split}
\end{equation}
 Multiplying the both side of \eqref{prop ruling out bubble accumulations-proof-17} by $\tilde{u}_{\varepsilon_{n}}^{2}(0)$, then left-hand side of \eqref{prop ruling out bubble accumulations-proof-17} satisfies
 \begin{equation}
    \text{LHS}=\varepsilon_{n}d_{\varepsilon_{n}}^{\frac{N-2}{2}(2^*-q)}\frac{2N-(N-2)q}{2q}\int_{ B(0,\rho)} \tilde{u}_{\varepsilon_{n}}^{2}(0)\tilde{u}_{\varepsilon_{n}}^q\geq 0,\text{~for any~} n\text{~and~}\rho.
 \end{equation}
 Next, from \eqref{prop ruling out bubble accumulations-proof-14}, \eqref{prop ruling out bubble accumulations-proof-16} and Lemma \ref{lem properties of B}, the first term in the right-hand side of \eqref{prop ruling out bubble accumulations-proof-17} satisfies
\begin{equation}
    \begin{split}
&\rho \int_{\partial B(0,\rho)}\left(\frac{\partial\tilde{w}_{\varepsilon_{n}}}{\partial\nu}\right)^{2}
-\frac{\rho}{2}\int_{\partial B(0,\rho)}
|\nabla \tilde{w}_{\varepsilon_{n}}|^2+\frac{N-2}{2}\int_{\partial B(0,\rho)}\frac{\partial \tilde{w}_{\varepsilon_{n}}}{\partial\nu}
  \tilde{w}_{\varepsilon_{n}} \\
  &\to -\frac{(N-2)^{2}}{2}a_{1}A\omega_{N}<0,\text{~~as~~}n\to\infty \text{~and~}\rho\to 0.
\end{split}
\end{equation}
In addition, from Proposition \ref{prop properties of isolated simple blow-up point}, there exist a constant $C>0$ independent of $n$ such that
\begin{equation}
    \tilde{u}_{\varepsilon_{n}}(0) \tilde{u}_{\varepsilon_{n}}(x)\leq \frac{C}{|x|^{N-2}},\text{~~in~~}B(0,\rho),
\end{equation}
thus the second term in the right-hand side of \eqref{prop ruling out bubble accumulations-proof-17} satisfies that for any $\rho>0$
\begin{equation}
\begin{split}
\frac{\rho}{2^*}\int_{\partial B(0,\rho)} \tilde{u}_{\varepsilon_{n}}^{2}(0)\tilde{u}_{\varepsilon_{n}}^{2^*}+\varepsilon_{n}d_{\varepsilon_{n}}^{\frac{N-2}{2}(2^*-q)}\frac{\rho}{q
  }\int_{\partial B(0,\rho)}\tilde{u}_{\varepsilon_{n}}^{2}(0)\tilde{u}_{\varepsilon_{n}}^{q}\to
   0, \text{~~as~~}n\to \infty.   
\end{split}
\end{equation}
Hence, when $\rho>0$ small enough and $n$ large enough, the right-hand side of \eqref{prop ruling out bubble accumulations-proof-17} satisfy that $\textrm{RHS}<0$, then we obtain a contraction.
\end{proof}

\begin{Cor}
    Let $u_{\varepsilon}$ be a family of solutions of \eqref{p-varepsion}, then there exists a small number $\varepsilon_{2}>0$ and a large number $M_{2}>0$ independent of $\varepsilon$ such that for any $\varepsilon\in(0,\varepsilon_{2})$ we have
    \begin{equation}
       \|u_{\varepsilon}\|_{H^{1}_{0}(\Omega)}\leq M_{2}.
    \end{equation}  
\end{Cor}
  
\begin{proof}
    By contradiction, suppose that there exist $\varepsilon_{n}\to 0$ as $n\to\infty$ such that 
    \begin{equation}\label{proof energy upperbound-1}
         \|u_{\varepsilon_{n}}\|_{H^{1}_{0}(\Omega)}\to\infty,\text{~as~}n\to\infty.
    \end{equation}
    Then by equation \eqref{p-varepsion}, we can obtain that $\|u_{\varepsilon_{n}}\|_{L^{\infty}(\Omega)}\to\infty$ as $n\to\infty$. Next, from Proposition \ref{prop ruling out bubble accumulations}, we know that the blow-up points of $u_{\varepsilon_{n}}$ are isolated, simple and depart each other, thus the number of blow-up points of $u_{\varepsilon_{n}}$ is prior finite. Moreover, from Proposition \ref{lem isolated blow up point 2}, Proposition \ref{prop properties of isolated simple blow-up point} and Harnack inequality, we get that $ \|u_{\varepsilon_{n}}\|_{H^{1}_{0}(\Omega)}$ tend to $S^{\frac{N}{2}}$ around each blow-up point, thus $\|u_{\varepsilon_{n}}\|_{H^{1}_{0}(\Omega)}$ is uniformly bounded and we obtain a contradiction with \eqref{proof energy upperbound-1}.
\end{proof}

\begin{proof}[Proof of Theorem \ref{thm multibble blowup}]
From the results in appendix \ref{Local estimates for blow-up points and proof of Theorem} and the argument above, we only need to consider the case that $\|u_{\varepsilon}\|_{L^{\infty}(\Omega)}$ is unbounded and it remains to prove $(\textrm{g})$ and $(\text{h})$ in Theorem \ref{thm multibble blowup}.

First, we prove $(\textrm{g})$. From Proposition \ref{prop properties of isolated simple blow-up point}, we know that for any $j\neq l\in\{1,\cdots,n\}$ and $\rho>0$ small enough
   \begin{equation}\label{proof of them multibubble blowup -1}
    u_{\varepsilon}(x)\leq C\frac{1}{u_{\varepsilon}(x_{j,\varepsilon})},\text{~for any~}x\in B(x_{j,\varepsilon},\rho/3)\setminus B(x_{j,\varepsilon},\rho/6).
\end{equation} 
On the other hand, using the Green's representation formula, Proposition \ref{prop properties of isolated simple blow-up point}, Proposition \ref{lem isolated blow up point 2} and dominated convergence theorem, we have
  \begin{equation}\label{proof of them multibubble blowup -2}
\begin{aligned}
     u_{\varepsilon}(x)&=\int_{\Omega}G(x,y)(u_{\varepsilon}^{2^*-1}(y)+\varepsilon u_{\varepsilon}^{q-1}(y))dy\\
     &\geq C\int_{B(x_{l,\varepsilon},\rho/3)}u_{\varepsilon}^{2^*-1}(y)dy\\
     &=\frac{C}{u_{\varepsilon}(x_{l,\varepsilon})}(\int_{\R^{N}}U^{2^*-1}+o(1))\geq C\frac{1}{u_{\varepsilon}(x_{l,\varepsilon})},
\end{aligned}
\end{equation}  
for any $x\in B(x_{j,\varepsilon},\rho/3)$. Combining \eqref{proof of them multibubble blowup -1} and \eqref{proof of them multibubble blowup -2} together, we can obtain $(\textrm{g})$.

In the following, we prove $(\textrm{h})$. For any $x\in \Omega\setminus\{a_{1},\cdots,a_{n}\}$, we choose $d>0$ small enough and use the Green's representation formula, 
 \begin{equation}\label{proof of them multibubble blowup -3}
\begin{aligned}
    \lambda_{1,\varepsilon}^{\frac{N-2}{2}}u_{\varepsilon}(x)&=\lambda_{1,\varepsilon}^{\frac{N-2}{2}}\int_{\Omega}G(x,y)(u_{\varepsilon}^{2^*-1}(y)+\varepsilon u_{\varepsilon}^{q-1}(y))dy\\    &=\sum_{j=1}^{n}G(x,x_{j,\varepsilon})\int_{B(x_{j,\varepsilon},d)}\lambda_{1,\varepsilon}^{\frac{N-2}{2}}(u_{\varepsilon}^{2^*-1}(y)+\varepsilon u_{\varepsilon}^{q-1}(y))dy\\
    &\quad+\sum_{j=1}^{n}\int_{B(x_{j,\varepsilon},d)}(G(x,y)-G(x,x_{j,\varepsilon}))\lambda_{1,\varepsilon}^{\frac{N-2}{2}}(u_{\varepsilon}^{2^*-1}(y)+\varepsilon u_{\varepsilon}^{q-1}(y))dy\\    &\quad+\int_{\Omega\setminus\cup_{j=1}^{n}B(x_{j,\varepsilon},d)}\lambda_{1,\varepsilon}^{\frac{N-2}{2}}G(x,y)(u_{\varepsilon}^{2^*-1}(y)+\varepsilon u_{\varepsilon}^{q-1}(y))dy.
\end{aligned}
\end{equation}   
Now, we estimate the terms on the right-hand side of \eqref{proof of them multibubble blowup -3} respectively.
First, from $(\textrm{c})$, $(\textrm{e})$ and $(\textrm{g})$ in Theorem \ref{thm multibble blowup}, we have
\begin{align}\label{proof of them multibubble blowup -4}
    &\lambda_{1,\varepsilon}^{\frac{N-2}{2}}\int_{B(x_{j,\varepsilon},d)}(u_{\varepsilon}^{2^*-1}(y)+\varepsilon u_{\varepsilon}^{q-1}(y))dy\nonumber\\
    &=(\frac{\lambda_{1,\varepsilon}}{\lambda_{j,\varepsilon}})^{\frac{N-2}{2}}\int_{B(0,\lambda_{j,\varepsilon}d)}(\lambda_{j,\varepsilon}^{-\frac{N-2}{2}}u_{\varepsilon}(x_{j,\varepsilon}+\lambda_{j,\varepsilon}^{-1}x))^{2^*-1}dx\nonumber\\
    &\quad+\varepsilon(\frac{\lambda_{1,\varepsilon}}{\lambda_{j,\varepsilon}})^{\frac{N-2}{2}}\lambda_{j,\varepsilon}^{\frac{N-2}{2}(q-2^*)}\int_{B(0,\lambda_{j,\varepsilon}d)}(\lambda_{j,\varepsilon}^{-\frac{N-2}{2}}u_{\varepsilon}(x_{j,\varepsilon}+\lambda_{j,\varepsilon}^{-1}x))^{q-1}dx\nonumber\\
    &=(\frac{\lambda_{1,\varepsilon}}{\lambda_{j,\varepsilon}})^{\frac{N-2}{2}}\int_{B(0,\lambda_{j,\varepsilon}d)}(\lambda_{j,\varepsilon}^{-\frac{N-2}{2}}u_{\varepsilon}(x_{j,\varepsilon}+\lambda_{j,\varepsilon}^{-1}x))^{2^*-1}dx\nonumber\\
    &\quad+\begin{cases}
    O(\varepsilon\lambda_{j,\varepsilon}^{\frac{N-2}{2}(q-2^*)}),&\text{~~if~~}q>\frac{2N-2}{N-2},\\ O(\varepsilon\ln(\lambda_{j,\varepsilon})\lambda_{j,\varepsilon}^{-1}),&\text{~~if~~}q=\frac{2N-2}{N-2},\\
    O(\varepsilon \lambda_{j,\varepsilon}^{-\frac{N-2}{2}(q-2)}),&\text{~~if~~}q<\frac{2N-2}{N-2},
    \end{cases}\nonumber\\
    &\to \mu_{j}\|U\|^{2^*-1}_{L^{2^*-1}(\R^{N})},\text{~~as~~}\varepsilon\to 0.
\end{align}   
Similarly, for any $x\in\Omega\setminus B(x_{j,\varepsilon},d)$, we get
\begin{small}
\begin{equation}\label{proof of them multibubble blowup -5}
    \begin{aligned}
        &\lambda_{1,\varepsilon}^{\frac{N-2}{2}}\int_{B(x_{j,\varepsilon},d)}(G(x,y)-G(x,x_{j,\varepsilon}))(u_{\varepsilon}^{2^*-1}(y)+\varepsilon u_{\varepsilon}^{q-1}(y))dy\\
        &\lesssim\lambda_{1,\varepsilon}^{\frac{N-2}{2}}\int_{B(x_{j,\varepsilon},d)}|y-x_{j,\varepsilon}|(u_{\varepsilon}^{2^*-1}(y)+\varepsilon u_{\varepsilon}^{q-1}(y))dy\\
        &=(\frac{\lambda_{1,\varepsilon}}{\lambda_{j,\varepsilon}})^{\frac{N-2}{2}}\lambda_{j,\varepsilon}^{-1}\int_{B(0,\lambda_{j,\varepsilon}d)}|x|(\lambda_{j,\varepsilon}^{-\frac{N-2}{2}}u_{\varepsilon}(x_{j,\varepsilon}+\lambda_{j,\varepsilon}^{-1}x))^{2^*-1}dx\\
        &\quad+(\frac{\lambda_{1,\varepsilon}}{\lambda_{j,\varepsilon}})^{\frac{N-2}{2}}\varepsilon\lambda_{j,\varepsilon}^{\frac{N-2}{2}q-(N+1)}\int_{B(0,\lambda_{j,\varepsilon}d)}|x|(\lambda_{j,\varepsilon}^{-\frac{N-2}{2}}u_{\varepsilon}(x_{j,\varepsilon}+\lambda_{j,\varepsilon}^{-1}x))^{q-1}dx\\
        &\to 0,\text{~~as~~}\varepsilon\to0,
    \end{aligned}
\end{equation}
\end{small}
and
\begin{equation}\label{proof of them multibubble blowup -6}
    \begin{aligned}
\int_{\Omega\setminus\cup_{j=1}^{n}B(x_{j,\varepsilon},d)}&\lambda_{1,\varepsilon}^{\frac{N-2}{2}}G(x,y)(u_{\varepsilon}^{2^*-1}(y)+\varepsilon u_{\varepsilon}^{q-1}(y))dy\\
&\leq C(\lambda_{\varepsilon}^{-2}+\varepsilon\lambda_{\varepsilon}^{-\frac{N-2}{2}(q-2)})\to 0,\text{~~as~~}\varepsilon\to0.
    \end{aligned}
\end{equation}
Thus combining \eqref{proof of them multibubble blowup -3}-\eqref{proof of them multibubble blowup -6} together, we can obtain that
\begin{equation}
\begin{aligned}
  \lambda_{1,\varepsilon}^{\frac{N-2}{2}}u_{\varepsilon}(x)&=\lambda_{1,\varepsilon}^{\frac{N-2}{2}}\int_{\Omega}G(x,y)(u_{\varepsilon}^{2^*-1}(y)+\varepsilon u_{\varepsilon}^{q-1}(y))dy\\
  &\to\|U\|_{L^{2^*-1}(\R^{N})}^{2^*-1}\sum_{j=1}^{n}\mu_{j}G(x,a_{j}):=\sum_{j=1}^{n}\nu_{j}G(x,a_{j})
\end{aligned}     
\end{equation}
in $C_{loc}(\Omega\setminus\{a_{1},\cdots,a_{n}\})$.
Similarly , we can prove that
\begin{equation}
     \lambda_{1,\varepsilon}^{\frac{N-2}{2}}u_{\varepsilon}(x)\to\sum_{j=1}^{n}\nu_{j}G(x,a_{j}), \text{~in~}C^{1}_{loc}(\Omega\setminus\{a_{1},\cdots,a_{n}\}).
\end{equation} 
Thus we complete the proof.
\end{proof}

\section{Asymptotic decomposition of the solutions}\label{section Decomposition of the solutions}
\setcounter{equation}{0}
Throughout this section, we always assume that $N\ge 3, q\in (\max\{2,\frac{4}{N-2}\},2^*)$, $\Omega$ is a smooth bounded domain and satisfies Assumption $\textrm{A}$.

\begin{Prop}\label{prop decomposition of the solutions-1}
Let $u_{\varepsilon}$ be a family of solutions of \eqref{p-varepsion}, then there exists $n\in\N$ and $\vec{a}=(a_{1},\cdots,a_{n})\in\Omega_{*}^{n}$ such that $u_{\varepsilon}(x)$  has the following decomposition
\begin{equation}\label{decomposition of u}
u_{\varepsilon}=\sum^n_{j=1} PU_{x_{j,\varepsilon}, \lambda_{j,\varepsilon}}+w_{\varepsilon},
\end{equation}
and for any $j=1,\cdots,n$, as $\varepsilon\to0$
\begin{equation}\label{equ v 1}
~ x_{j,\varepsilon}\rightarrow a_j, ~\lambda_{j,\varepsilon}\rightarrow \infty,~\|w_{\varepsilon}\|_{H^{1}_{0}(\Omega)}\to0,~w_\varepsilon\in \bigcap^n_{j=1}E_{x_{j,\varepsilon},\lambda_{j,\varepsilon}}.
\end{equation}   
\end{Prop}
\begin{proof}

First, from Theorem \ref{thm multibble blowup}, assumption A and Lemma \ref{estimate of U-lambda-a and psi-lambda-a 3}, we can obtain the following decomposition for $u_{\varepsilon}$
\begin{equation}
    u_{\varepsilon}=\sum^n_{j=1} PU_{x_{j,\varepsilon}, \lambda_{j,\varepsilon}}+w_{\varepsilon},
\end{equation}
with 
\begin{equation}
    ~ x_{j,\varepsilon}\rightarrow a_j, ~\vec{a}=(a_{1},\cdots,a_{n})\in\Omega_{*}^{n},~\lambda_{j,\varepsilon}=u_{\varepsilon}(x_{j,\varepsilon})^{\frac{2}{N-2}}\rightarrow \infty,~\|w_{\varepsilon}\|_{H^{1}_{0}(\Omega)}\to0.
\end{equation}
Next, we need to move the location $x_{j,\varepsilon}$ and $\lambda_{j,\varepsilon}$ obtained above a little bit such that 
\begin{equation}
    w_\varepsilon\in \bigcap^n_{j=1}E_{x_{j,\varepsilon},\lambda_{j,\varepsilon}}.
\end{equation}
For any fixed $\varepsilon$ small enough, we consider the problem
\begin{equation}\label{prop 2.1 proof 10}
    \inf\left\{\|u_{\varepsilon}-\sum_{j=1}^{n}PU_{x_{j},\lambda_{j}}\|_{H^{1}_{0}(\Omega)}^{2}:|x_{j}-x_{j,\varepsilon}|\leq c_{j,1}^{-1}\lambda_{j,\varepsilon}^{-1}\text{~and~}\lambda_{j}\in[c_{j,1}\lambda_{j,\varepsilon},c_{j,2}\lambda_{j,\varepsilon}]\right\},
\end{equation}
where $c_{j,1},c_{j,2}>0$ are some constants independent of $\varepsilon$. Then \eqref{prop 2.1 proof 10} is achieved by some $(\bar{x}_{j,\varepsilon},\bar{\lambda}_{j,\varepsilon})$. Since
\begin{equation}
    \|u_{\varepsilon}-\sum_{j=1}^{n}PU_{\bar{x}_{j,\varepsilon},\bar{\lambda}_{j,\varepsilon}}\|_{H^{1}_{0}(\Omega)}\leq \|u_{\varepsilon}-\sum_{j=1}^{n}PU_{x_{j,\varepsilon},\lambda_{j,\varepsilon}}\|_{H^{1}_{0}(\Omega)}=\|w_{\varepsilon}\|_{H^{1}_{0}(\Omega)}\to0,
\end{equation}
then by Lemma \ref{estimate of two different bubbles}
\begin{equation}
\begin{aligned}
    &\|u_{\varepsilon}-\sum_{j=1}^{n}PU_{\bar{x}_{j,\varepsilon},\bar{\lambda}_{j,\varepsilon}}\|_{H^{1}_{0}(\Omega)}^{2}\\
    &=\|\sum_{j=1}^{n}PU_{\bar{x}_{j,\varepsilon},\bar{\lambda}_{j,\varepsilon}}-\sum_{j=1}^{n}PU_{x_{j,\varepsilon},\lambda_{j,\varepsilon}}\|_{H^{1}_{0}(\Omega)}^{2}+O( \|u_{\varepsilon}-\sum_{j=1}^{n}PU_{x_{j,\varepsilon},\lambda_{j,\varepsilon}}\|_{H^{1}_{0}(\Omega)})\\
    &=\|\sum_{j=1}^{n}PU_{\bar{x}_{j,\varepsilon},\bar{\lambda}_{j,\varepsilon}}\|_{H^{1}_{0}(\Omega)}^{2}+\|\sum_{j=1}^{n}PU_{{x}_{j,\varepsilon},{\lambda}_{j,\varepsilon}}\|_{H^{1}_{0}(\Omega)}^{2}+O(\sum_{j,l}\int_{\Omega}U^{2^*-1}_{\bar{x}_{j,\varepsilon},\bar{\lambda}_{j,\varepsilon}}U_{{x}_{l,\varepsilon},{\lambda}_{l,\varepsilon}})+o(1)\\
    &=2nS^{\frac{N}{2}}+O((\sum_{j}\frac{\bar{\lambda}_{j,\varepsilon}}{\lambda_{j,\varepsilon}}+\frac{\lambda_{j,\varepsilon}}{\bar{\lambda}_{j,\varepsilon}}+\lambda_{j,\varepsilon}\bar{\lambda}_{j,\varepsilon}|\bar{x}_{j,\varepsilon}-x_{j,\varepsilon}|^{2})^{-\frac{N-2}{2}})+o(1),
\end{aligned}
\end{equation}
it follows  that for $c_{j,1}>0$ small enough and $c_{j,2}>0$ large enough, it holds
\begin{equation}
    |\bar{x}_{j,\varepsilon}-x_{j,\varepsilon}|<c_{j,1}^{-1}\lambda_{j,\varepsilon}^{-1},\text{~~and~~}\bar{\lambda}_{j,\varepsilon}\in (c_{j,1}\lambda_{j,\varepsilon},c_{j,2}\lambda_{j,\varepsilon}).
\end{equation}
Hence, we let $\bar{w}_{\varepsilon}:=u_{\varepsilon}-\sum_{j=1}^{n}PU_{\bar{x}_{j,\varepsilon},\bar{\lambda}_{j,\varepsilon}}\in \bigcap^n_{j=1}E_{\bar{x}_{j,\varepsilon},\bar{\lambda}_{j,\varepsilon}}$ and this completes the proof.
\end{proof}

Next, from Theorem \ref{thm multibble blowup} with the proof of Proposition \ref{prop decomposition of the solutions-1}, we have the following corollary.
\begin{Cor}\label{cor of the decomposition}
Using the notations in Proposition \ref{prop decomposition of the solutions-1}, it holds 
\begin{enumerate}
    \item  $\mu_{j}=\lim_{\varepsilon\to0}(\frac{\lambda_{1,\varepsilon}}{\lambda_{j,\varepsilon}})^{\frac{N-2}{2}}$ exists and lies in $(0,\infty)$ for any $j=1,\cdots,n$.
    \item There exist a constant $C>0$ such that $u_{\varepsilon}\leq C\sum_{j=1}^{n}U_{x_{j,\varepsilon},\lambda_{j,\varepsilon}}$ in $\Omega$.
    \item There exist a constant $C_{1}>0$ such that $\varepsilon\leq C_{1}\lambda_{j,\varepsilon}^{-\frac{N-2}{2}q+2}$ for any $j=1,\cdots,n$.  
\end{enumerate}
\end{Cor}

Since the $\lambda_{j,\varepsilon}$ are all of comparable size by Corollary \ref{cor of the decomposition}, it will be convenient to set  ${\lambda}_{\varepsilon}:=\min\{\lambda_{1,\varepsilon},\cdots,\lambda_{n,\varepsilon}\}$.

In the following, we need to estimate the error term $w_{\varepsilon}$ in the decomposition \eqref{decomposition of u}. First, using \eqref{decomposition of u} and \eqref{p-varepsion}, it is easy to see that $w_{\varepsilon}$ satisfies
\begin{equation}\label{equ of w}
    Q_{\varepsilon}w_{\varepsilon}=L_{\varepsilon}+R_{\varepsilon}(w_{\varepsilon}),
\end{equation}
where
\begin{equation}
    Q_{\varepsilon}w_{\varepsilon}:=-\Delta w_{\varepsilon}-(2^*-1)(\sum_{j=1}^{n}PU_{x_{j,\varepsilon},\lambda_{j,\varepsilon}})^{2^*-2}w_{\varepsilon},
\end{equation}
\begin{equation}
    L_{\varepsilon}:=(\sum_{j=1}^{n}PU_{x_{j,\varepsilon},\lambda_{j,\varepsilon}})^{2^*-1}-\sum_{j=1}^{n}U_{x_{j,\varepsilon},\lambda_{j,\varepsilon}}^{2^*-1}+\varepsilon (\sum_{j=1}^{n}PU_{x_{j,\varepsilon},\lambda_{j,\varepsilon}})^{q-1},
\end{equation}
and
\begin{equation}
\begin{aligned}
R_{\varepsilon}(w_{\varepsilon})&=(\sum_{j=1}^{n}PU_{x_{j,\varepsilon},\lambda_{j,\varepsilon}}+w_{\varepsilon})^{2^*-1}-(\sum_{j=1}^{n}PU_{x_{j,\varepsilon},\lambda_{j,\varepsilon}})^{2^*-1}-(2^*-1)(\sum_{j=1}^{n}PU_{x_{j,\varepsilon},\lambda_{j,\varepsilon}})^{2^*-2}w_{\varepsilon}\\
 &\quad+\varepsilon\left((\sum_{j=1}^{n}PU_{x_{j,\varepsilon},\lambda_{j,\varepsilon}}+w_{\varepsilon})^{q-1}-(\sum_{j=1}^{n}PU_{x_{j,\varepsilon},\lambda_{j,\varepsilon}})^{q-1}\right).
\end{aligned}   
\end{equation}

 We first recall the following coercivity of operator $Q_{\varepsilon}$, which is a crucial tool to estimate the error term $w_{\varepsilon}$. 
\begin{Lem}\label{q var}
There exists two constants $\varepsilon_{0}>0$ and $\rho_{0}>0$ such that for any $0<\varepsilon<\varepsilon_{0}$, it holds
 \begin{equation}
     \langle Q_{\varepsilon}u,u\rangle\geq \rho_{0} \|u\|_{H^{1}_{0}(\Omega)}^{2},\quad\forall u\in \bigcap^n_{j=1}E_{x_{j,\varepsilon},\lambda_{j,\varepsilon}},
 \end{equation}
 \begin{proof}
     We refer Proposition 3.1 in \cite{Cao2021Trans} and Proposition 2.4.3 in \cite{Cao_Peng_Yan_2021} for the details of the proof. 
 \end{proof}
\end{Lem}

\begin{Lem}\label{L var}
For any $\eta\in H^{1}_{0}(\Omega)$, it holds
\begin{equation}\label{L var-1}
 \int_{\Omega}L_{\varepsilon}\eta=\begin{cases}
     O\left(\frac{1}{\lambda_{\varepsilon}}+\frac{\varepsilon}{\lambda_{\varepsilon}^{3-\frac{q}{2}}}\right), &\text{~~if~~} N=3\text{~and~} q\in (4,6),\\

         O\left(\frac{1}{\lambda_{\varepsilon}^{2}}+\frac{\varepsilon}{\lambda_{\varepsilon}^{q-1}}\right), &\text{~~if~~} N=4 \text{~and~} q\in (2,\frac{5}{2}),\\
          O\left(\frac{1}{\lambda_{\varepsilon}^{2}}+\frac{\varepsilon(\ln\lambda_{\varepsilon})^{\frac{3}{4}}}{\lambda_{\varepsilon}^{3/2}}\right), &\text{~~if~~} N=4 \text{~and~} q=\frac{5}{2},\\
     O\left(\frac{1}{\lambda_{\varepsilon}^{2}}+\frac{\varepsilon}{\lambda_{\varepsilon}^{4-q}}\right), &\text{~~if~~} N=4 \text{~and~} q\in (\frac{5}{2},4),\\
     
      O\left(\frac{1}{\lambda_{\varepsilon}^{3}}+\frac{\varepsilon}{\lambda_{\varepsilon}^{\frac{3(q-1)}{2}}}\right), &\text{~~if~~} N=5 \text{~and~} q\in (2,\frac{13}{6}),\\

O\left(\frac{1}{\lambda_{\varepsilon}^{3}}+\frac{\varepsilon(\ln\lambda_{\varepsilon})^{\frac{7}{10}}}{\lambda_{\varepsilon}^{5-\frac{3}{2}q}}\right), &\text{~~if~~} N=5 \text{~and~} q=\frac{13}{6},\\
      
      O\left(\frac{1}{\lambda_{\varepsilon}^{3}}+\frac{\varepsilon}{\lambda_{\varepsilon}^{5-\frac{3}{2}q}}\right), &\text{~~if~~} N=5 \text{~and~} q\in (\frac{13}{6},\frac{10}{3}),\\
O\left(\frac{(\ln\lambda_{\varepsilon})^{\frac{N+2}{2N}}}{\lambda_{\varepsilon}^{\frac{N+2}{2}}}+\frac{\varepsilon}{\lambda_{\varepsilon}^{N-\frac{N-2}{2}q}}\right), &\text{~~if~~} N\geq 6
 \end{cases}   \times \|\eta\|_{H^{1}_{0}(\Omega)}.
\end{equation}
\end{Lem}
\begin{proof}
First, using the elementary inequality in Lemma \ref{Elementary estimate 2}, we have
    \begin{equation}\label{estimate of L var proof 1}
        \begin{aligned}
           (\sum_{i=1}^{n}PU_{x_{i,\varepsilon},\lambda_{i,\varepsilon}})^{2^*-1}-\sum_{i=1}^{n}U_{x_{i,\varepsilon},\lambda_{i,\varepsilon}}^{2^*-1}&=\sum_{i=1}^{n}PU_{x_{i,\varepsilon},\lambda_{i,\varepsilon}}^{2^*-1}-\sum_{i=1}^{n}U_{x_{i,\varepsilon},\lambda_{i,\varepsilon}}^{2^*-1}\\
            &+\begin{cases}
            O(\sum_{i\neq j}PU_{x_{i,\varepsilon},\lambda_{i,\varepsilon}}^{\frac{2^*-1}{2}}PU_{x_{j,\varepsilon},\lambda_{j,\varepsilon}}^{\frac{2^*-1}{2}}),\quad~\text{if~}N\geq 6,\\
            O(\sum_{i\neq j}PU_{x_{i,\varepsilon},\lambda_{i,\varepsilon}}^{2^*-2}PU_{x_{j,\varepsilon},\lambda_{j,\varepsilon}}),\quad~\text{if~}N=3,4,5.\\
        \end{cases}
        \end{aligned}
    \end{equation}
For any $i=1,\cdots,n$, let $\psi_{x_{i,\varepsilon},\lambda_{i,\varepsilon}}:=U_{x_{i,\varepsilon},\lambda_{i,\varepsilon}}-PU_{x_{i,\varepsilon},\lambda_{i,\varepsilon}}$, then by strong maximum principle, we get that $0\leq \psi_{x_{i,\varepsilon},\lambda_{i,\varepsilon}}\leq U_{x_{i,\varepsilon},\lambda_{i,\varepsilon}}$, thus by Lemma \ref{Elementary estimate} and Lemma \ref{estimate of U-lambda-a and psi-lambda-a 3}, we obtain that
\begin{equation}\label{estimate of L var proof 2}
    |PU_{x_{i,\varepsilon},\lambda_{i,\varepsilon}}^{2^*-1}-U_{x_{i,\varepsilon},\lambda_{i,\varepsilon}}^{2^*-1}|\lesssim U_{x_{i,\varepsilon},\lambda_{i,\varepsilon}}^{2^*-2}\psi_{x_{i,\varepsilon},\lambda_{i,\varepsilon}}\lesssim U_{x_{i,\varepsilon},\lambda_{i,\varepsilon}}^{2^*-2}\lambda_{i,\varepsilon}^{-\frac{N-2}{2}}.
\end{equation}
Hence, for any $\eta\in H^{1}_{0}(\Omega)$, by using \eqref{estimate of L var proof 2}, H\"{o}lder inequality and Sobolev inequality, the first term in the right hand side of \eqref{estimate of L var proof 1} can be bounded by
\begin{equation}\label{estimate of L var proof 3}
\begin{aligned}
    |\int_{\Omega}(PU_{x_{i,\varepsilon},\lambda_{i,\varepsilon}}^{2^*-1}-U_{x_{i,\varepsilon},\lambda_{i,\varepsilon}}^{2^*-1})\eta|&\lesssim\lambda_{i,\varepsilon}^{-\frac{N-2}{2}}\int_{\Omega}U_{x_{i,\varepsilon},\lambda_{i,\varepsilon}}^{2^*-2}|\eta|\\
    &\lesssim\begin{cases}
        \lambda_{i,\varepsilon}^{-\frac{N+2}{2}}\|\eta\|_{H^{1}_{0}(\Omega)},&\ \text{~if~}N\geq 7,\\
        (\ln\lambda_{i,\varepsilon})^{\frac{2}{3}}\lambda_{i,\varepsilon}^{-4}\|\eta\|_{H^{1}_{0}(\Omega)},&\ \text{~if~}N=6,\\
        \lambda_{i,\varepsilon}^{2-N}\|\eta\|_{H^{1}_{0}(\Omega)},&\ \text{~if~}N\leq 5.
    \end{cases}
\end{aligned} 
\end{equation}

Next, we consider the second term in the right hand side of \eqref{estimate of L var proof 1}. When $N\geq 6$, for any $i\neq j$, $\eta \in H^{1}_{0}(\Omega)$ and $d>0$ small enough, we have
\begin{equation}
    \begin{aligned}
        \int_{\Omega}PU_{x_{i,\varepsilon},\lambda_{i,\varepsilon}}^{\frac{2^*-1}{2}}PU_{x_{j,\varepsilon},\lambda_{j,\varepsilon}}^{\frac{2^*-1}{2}}\eta
       &=\int_{B(a_{i},d)}PU_{x_{i,\varepsilon},\lambda_{i,\varepsilon}}^{\frac{2^*-1}{2}}PU_{x_{j,\varepsilon},\lambda_{j,\varepsilon}}^{\frac{2^*-1}{2}}\eta+\int_{B(a_{j},d)}PU_{x_{i,\varepsilon},\lambda_{i,\varepsilon}}^{\frac{2^*-1}{2}}PU_{x_{j,\varepsilon},\lambda_{j,\varepsilon}}^{\frac{2^*-1}{2}}\eta\\
        &+\int_{\Omega\setminus (B(a_{i},d)\cup B(a_{j},d))}PU_{x_{i,\varepsilon},\lambda_{i,\varepsilon}}^{\frac{2^*-1}{2}}PU_{x_{j,\varepsilon},\lambda_{j,\varepsilon}}^{\frac{2^*-1}{2}}\eta,
    \end{aligned}
\end{equation}
where
\begin{equation}
    \begin{aligned}
        \int_{B(a_{i},d)}PU_{x_{i,\varepsilon},\lambda_{i,\varepsilon}}^{\frac{2^*-1}{2}}PU_{x_{j,\varepsilon},\lambda_{j,\varepsilon}}^{\frac{2^*-1}{2}}\eta&\lesssim\lambda_{j,\varepsilon}^{-\frac{N+2}{4}}\int_{B(a_{i},d)}PU_{x_{i,\varepsilon},\lambda_{i,\varepsilon}}^{\frac{2^*-1}{2}}\eta\\
        &\lesssim\lambda_{j,\varepsilon}^{-\frac{N+2}{4}}(\int_{B(a_{i},d)}PU_{x_{i,\varepsilon},\lambda_{i,\varepsilon}}^{\frac{N}{N-2}})^{\frac{N+2}{2N}}\|\eta\|_{H^{1}_{0}}\\
        &\lesssim \lambda_{j,\varepsilon}^{-\frac{N+2}{4}}\lambda_{i,\varepsilon}^{-\frac{N+2}{4}}(\ln\lambda_{i,\varepsilon})^{\frac{N+2}{2N}}\|\eta\|_{H^{1}_{0}}.
    \end{aligned}
\end{equation}
Similarly, we have
\begin{equation}
    \int_{B(a_{j},d)}PU_{x_{j,\varepsilon},\lambda_{j,\varepsilon}}^{\frac{2^*-1}{2}}PU_{x_{i,\varepsilon},\lambda_{i,\varepsilon}}^{\frac{2^*-1}{2}}\eta\lesssim \lambda_{i,\varepsilon}^{-\frac{N+2}{4}}\lambda_{j,\varepsilon}^{-\frac{N+2}{4}}(\ln\lambda_{j,\varepsilon})^{\frac{N+2}{2N}}\|\eta\|_{H^{1}_{0}}.
\end{equation}
On the other hand
\begin{equation}
    \int_{\Omega\setminus (B(a_{i},d)\cup B(a_{j},d))}PU_{x_{i,\varepsilon},\lambda_{i,\varepsilon}}^{\frac{2^*-1}{2}}PU_{x_{j,\varepsilon},\lambda_{j,\varepsilon}}^{\frac{2^*-1}{2}}\eta\lesssim \lambda_{i,\varepsilon}^{-\frac{N+2}{4}}\lambda_{j,\varepsilon}^{-\frac{N+2}{4}}\|\eta\|_{H^{1}_{0}}.
\end{equation}
Hence, when $N\geq 6$ we get that
\begin{equation}\label{estimate of L var proof 8}
    \int_{\Omega}PU_{x_{i,\varepsilon},\lambda_{i,\varepsilon}}^{\frac{2^*-1}{2}}PU_{x_{j,\varepsilon},\lambda_{j,\varepsilon}}^{\frac{2^*-1}{2}}\eta\lesssim\lambda_{\varepsilon}^{-\frac{N+2}{2}}(\ln\lambda_{\varepsilon})^{\frac{N+2}{2N}}\|\eta\|_{H^{1}_{0}}.
\end{equation}
On the other hand, when $N<6$, for any $i\neq j$, $\eta \in H^{1}_{0}(\Omega)$ and $d>0$ small enough, we have
\begin{equation}
    \begin{aligned}
        \int_{\Omega}PU_{x_{i,\varepsilon},\lambda_{i,\varepsilon}}^{2^*-2}PU_{x_{j,\varepsilon},\lambda_{j,\varepsilon}}\eta&=\int_{B(a_{i},d)}PU_{x_{i,\varepsilon},\lambda_{i,\varepsilon}}^{2^*-2}PU_{x_{j,\varepsilon},\lambda_{j,\varepsilon}}\eta+\int_{B(a_{j},d)}PU_{x_{i,\varepsilon},\lambda_{i,\varepsilon}}^{2^*-2}PU_{x_{j,\varepsilon},\lambda_{j,\varepsilon}}\eta\\
        &+\int_{\Omega\setminus(B(a_{i},d)\cup B(a_{j},d)}PU_{x_{i,\varepsilon},\lambda_{i,\varepsilon}}^{2^*-2}PU_{x_{j,\varepsilon},\lambda_{j,\varepsilon}}\eta,
    \end{aligned}
\end{equation}
where
\begin{equation}
    \begin{aligned}
        \int_{B(a_{i},d)}PU_{x_{i,\varepsilon},\lambda_{i,\varepsilon}}^{2^*-2}PU_{x_{j,\varepsilon},\lambda_{j,\varepsilon}}\eta&\lesssim\lambda_{j,\varepsilon}^{-\frac{N-2}{2}}\int_{B(a_{i},d)}PU_{x_{i,\varepsilon},\lambda_{i,\varepsilon}}^{2^*-2}\eta\\
    &\lesssim\lambda_{j,\varepsilon}^{-\frac{N-2}{2}}(\int_{B(a_{i},d)}PU_{x_{i,\varepsilon},\lambda_{i,\varepsilon}}^{\frac{8N}{(N+2)(N-2)}})^{\frac{N+2}{2N}}\|\eta\|_{H^{1}_{0}}\\
    &\lesssim\lambda_{j,\varepsilon}^{-\frac{N-2}{2}}\lambda_{i,\varepsilon}^{-\frac{N-2}{2}}\|\eta\|_{H^{1}_{0}},
    \end{aligned}
\end{equation}
\begin{equation}
   \begin{aligned}\int_{B(a_{j},d)}PU_{x_{i,\varepsilon},\lambda_{i,\varepsilon}}^{2^*-2}PU_{x_{j,\varepsilon},\lambda_{j,\varepsilon}}\eta&\lesssim\lambda_{i,\varepsilon}^{-2}\int_{B(a_{j},d)}PU_{x_{j,\varepsilon},\lambda_{i,\varepsilon}}\eta\\
    &\lesssim\lambda_{i,\varepsilon}^{-2}(\int_{B(a_{j},d)}PU_{x_{j,\varepsilon},\lambda_{i,\varepsilon}}^{\frac{2N}{N+2}})^{\frac{N+2}{2N}}\|\eta\|_{H^{1}_{0}}\\
    &\lesssim\lambda_{i,\varepsilon}^{-2}\lambda_{j,\varepsilon}^{-\frac{N-2}{2}}\|\eta\|_{H^{1}_{0}},
    \end{aligned}
\end{equation}
and
\begin{equation}
    \int_{\Omega\setminus (B(a_{i},d)\cup B(a_{j},d))}PU_{x_{i,\varepsilon},\lambda_{i,\varepsilon}}^{2^*-2}PU_{x_{j,\varepsilon},\lambda_{j,\varepsilon}}\eta\lesssim \lambda_{i,\varepsilon}^{-2}\lambda_{j,\varepsilon}^{-\frac{N-2}{2}}\|\eta\|_{H^{1}_{0}}.
\end{equation}
Hence, when $N<6$ we conclude that
\begin{equation}\label{estimate of L var proof 13}
    \int_{\Omega}PU_{x_{i,\varepsilon},\lambda_{i,\varepsilon}}^{2^*-2}PU_{x_{j,\varepsilon},\lambda_{j,\varepsilon}}\eta\lesssim \lambda_{\varepsilon}^{2-N}\|\eta\|_{H^{1}_{0}}.
\end{equation}

Finally, we consider the sub-critical term in $L_{\varepsilon}$. For any $i=1,\cdots,n$ and $\eta\in H^{1}_{0}(\Omega)$, by H\"older inequality and Sobolev inequality
\begin{equation}\label{estimate of L var proof 14}
\begin{aligned}
    \int_{\Omega}PU_{x_{i,\varepsilon},\lambda_{i,\varepsilon}}^{q-1}\eta&\lesssim(\int_{\Omega}PU_{x_{i,\varepsilon},\lambda_{i,\varepsilon}}^{\frac{2N(q-1)}{N+2}})^{\frac{N+2}{2N}}\|\eta\|_{H^{1}_{0}},\\
    &\lesssim\begin{cases}
     O\left(\lambda_{i,\varepsilon}^{\frac{q}{2}-3}\right), &\text{~~if~~} N=3\text{~and~} q\in (4,6),\\

         O\left(\lambda_{i,\varepsilon}^{-(q-1)}\right), &\text{~~if~~} N=4 \text{~and~} q\in (2,\frac{5}{2}),\\
          O\left(\lambda_{i,\varepsilon}^{q-4}(\ln\lambda_{i,\varepsilon})^{\frac{3}{4}}\right), &\text{~~if~~} N=4 \text{~and~} q=\frac{5}{2},\\
     O\left(\lambda_{i,\varepsilon}^{q-4}\right), &\text{~~if~~} N=4 \text{~and~} q\in (\frac{5}{2},4),\\
     
      O\left(\lambda_{i,\varepsilon}^{-\frac{3(q-1)}{2}}\right), &\text{~~if~~} N=5 \text{~and~} q\in (2,\frac{13}{6}),\\

O\left(\lambda_{i,\varepsilon}^{\frac{3}{2}q-5}(\ln\lambda_{i,\varepsilon})^{\frac{7}{10}}\right), &\text{~~if~~} N=5 \text{~and~} q=\frac{13}{6},\\
      
      O\left(\lambda_{i,\varepsilon}^{\frac{3}{2}q-5}\right), &\text{~~if~~} N=5 \text{~and~} q\in (\frac{13}{6},\frac{10}{3}),\\
O\left(\lambda_{i,\varepsilon}^{\frac{N-2}{2}q-N}\right), &\text{~~if~~} N\geq 6
 \end{cases}   
\end{aligned}\times \|\eta\|_{H^{1}_{0}(\Omega)}
\end{equation}
Combining \eqref{estimate of L var proof 1}, \eqref{estimate of L var proof 3}, \eqref{estimate of L var proof 8}, \eqref{estimate of L var proof 13} and \eqref{estimate of L var proof 14} together, we can obtain \eqref{L var-1}.
\end{proof}

\begin{Lem}\label{r var}
    For any $\eta\in H^{1}_{0}(\Omega)$, it holds
    \begin{equation}
        \int_{\Omega}R_{\varepsilon}(w_{\varepsilon})\eta=o(\|w_{\varepsilon}\|_{H^{1}_{0}(\Omega)})\|\eta\|_{H^{1}_{0}(\Omega)}.
    \end{equation}
\end{Lem}
\begin{proof}
    We can prove this Lemma by using H\"older inequality, Sobolev inequality and Lemma \ref{Elementary estimate}.
\end{proof}

\begin{proof}[Proof of Theorem \ref{thm decomposition}]
   By using \eqref{equ of w}, Lemma \ref{q var}-Lemma \ref{r var} and Corollary \ref{cor of the decomposition}, we can obtain the estimate \eqref{perturbation estimate of w-varepsilon} and thus Theorem \ref{thm decomposition} follows.
\end{proof}

\section{Blow-up rate and location of the concentration point}\label{Blow-up rate and location of the concentration point}
Throughout this section, we always assume that $N\ge 3, q\in (\max\{2,\frac{4}{N-2}\},2^*)$, $\Omega$ is a smooth bounded domain and satisfies Assumption $\textrm{A}$.
\subsection{Some integral estimates}

\begin{Lem}\label{lem of estimate of u 1}
    For any $j=1,\cdots,n$ and fixed $d>0$ small enough, it holds
    \begin{equation}\label{lem of estimate of u 1-1}
         \begin{aligned}
\int_{B(x_{j,\varepsilon},d)}u_{\varepsilon}^{2^*-1}(y)dy=\frac{A}{\lambda_{j,\varepsilon}^{\frac{N-2}{2}}}+\begin{cases}
     O\left(\frac{1}{\lambda_{\varepsilon}^{\frac{3}{2}}}\right), &\text{~~if~~} N=3\text{~and~} q\in (4,6),\\

         O\left(\frac{1}{\lambda_{\varepsilon}^{2q-2}}\right), &\text{~~if~~} N=4 \text{~and~} q\in (2,\frac{5}{2}),\\
          O\left(\frac{\ln\lambda_{\varepsilon}}{\lambda_{\varepsilon}^{3}}\right), &\text{~~if~~} N=4 \text{~and~} q\in[\frac{5}{2},4),\\

     O\left(\frac{1}{\lambda_{\varepsilon}^{\frac{N+2}{2}}}\right), &\text{~~if~~} N\geq 5\text{~and~} q\in (2,2^*),
 \end{cases}
    \end{aligned}
    \end{equation}
    where $A$ is the  constant defined in \eqref{definition of A and B}.
\end{Lem}

\begin{proof}
Using decomposition \eqref{decomposition of u} and $PU_{x_{j,\varepsilon},\lambda_{j,\varepsilon}}=U_{x_{j,\varepsilon},\lambda_{j,\varepsilon}}-\psi_{x_{j,\varepsilon},\lambda_{j,\varepsilon}}$ together with Lemma \ref{Elementary estimate}, we have
\begin{equation}\label{lem of estimate of u 1 proof 1}
\int_{B(x_{j,\varepsilon},d)}u_{\varepsilon}^{2^*-1}(y)dy=\int_{B(x_{j,\varepsilon},d)}U_{x_{j,\varepsilon},\lambda_{j,\varepsilon}}^{2^*-1}(y)dy+\mathrm{I},
\end{equation}
where the term $\textrm{I}$ can be controlled by
\begin{equation}\label{lem of estimate of u 1 proof 2}
    \begin{aligned}       \textrm{I}\lesssim&\int_{B(x_{j,\varepsilon},d)}U_{x_{j,\varepsilon},\lambda_{j,\varepsilon}}^{2^*-2}(y)\left(\sum_{i\neq j}U_{x_{i,\varepsilon},\lambda_{i,\varepsilon}}(y)+\psi_{x_{j,\varepsilon},\lambda_{j,\varepsilon}}(y)+|w_{\varepsilon}|(y)\right)dy\\
        &+\int_{B(x_{j,\varepsilon},d)}\sum_{i\neq j}U_{x_{i,\varepsilon},\lambda_{i,\varepsilon}}^{2^*-1}(y)+\psi_{x_{j,\varepsilon},\lambda_{j,\varepsilon}}^{2^*-1}(y)+|w_{\varepsilon}|^{2^*-1}(y)dy.
    \end{aligned}
\end{equation}
First, we can estimate the first term on the right hand side of \eqref{lem of estimate of u 1 proof 1} as follows
\begin{equation}\label{lem of estimate of u 1 proof 4}
    \int_{B(x_{j,\varepsilon},d)}U_{x_{j,\varepsilon},\lambda_{j,\varepsilon}}^{2^*-1}(y)dy=\lambda_{j,\varepsilon}^{\frac{2-N}{2}}\int_{B(0,\lambda_{j,\varepsilon}d)}U_{0,1}^{2^*-1}(y)dy=A\lambda_{j,\varepsilon}^{\frac{2-N}{2}}+O(\lambda_{\varepsilon}^{-\frac{N+2}{2}}),
\end{equation}
where $A$ is defined in \eqref{definition of A and B}. Next, we estimate the first term on the right hand side of \eqref{lem of estimate of u 1 proof 2}. Using H\"older inequality and Sobolev inequality 

\begin{equation}\label{lem of estimate of u 1 proof 5}
  \begin{aligned}
\int_{B(x_{j,\varepsilon},d)}U_{x_{j,\varepsilon},\lambda_{j,\varepsilon}}^{2^*-2}(y)|w_{\varepsilon}|(y)dy & \leq \left(\int_{B(x_{j,\varepsilon},d)}U_{x_{j,\varepsilon},\lambda_{j,\varepsilon}}^{\frac{8N}{(N-2)(N+2)}}(y)dy\right)^{\frac{N+2}{2N}}\|w_{\varepsilon}\|_{L^{2^*}(\Omega)}\\
&=\begin{cases}
    O(\lambda_{j,\varepsilon}^{\frac{2-N}{2}}),&\text{~~if~~}N=3,4,5\\
    O(\lambda_{j,\varepsilon}^{-2}(\ln\lambda_{j,\varepsilon})^{2/3}),&\text{~~if~~}N=6\\
    O(\lambda_{j,\varepsilon}^{-2}),&\text{~~if~~}N\geq 7\\
\end{cases}\times \|w_{\varepsilon}\|_{H^{1}_{0}(\Omega)}.\\
  \end{aligned}
\end{equation}       
On the other hand, for any $i\neq j$, by Lemma \ref{estimate of U-lambda-a and psi-lambda-a 3}, we know that 
\begin{equation}\label{lem of estimate of u 1 proof 6}
    \begin{aligned}
\int_{B(x_{j,\varepsilon},d)}&U_{x_{j,\varepsilon},\lambda_{j,\varepsilon}}^{2^*-2}(y)\left(U_{x_{i,\varepsilon},\lambda_{i,\varepsilon}}(y)+\psi_{x_{j,\varepsilon},x_{j,\varepsilon}}(y)\right)dy\\
&\lesssim \lambda_{\varepsilon}^{-\frac{N-2}{2}}\int_{B(x_{j,\varepsilon},d)}U_{x_{j,\varepsilon},\lambda_{j,\varepsilon}}^{2^*-2}(y)dy\\
&=\begin{cases}
     O\left(\frac{1}{\lambda_{\varepsilon}^{\frac{3}{2}}}\right), &\text{~~if~~} N=3\text{~and~} q\in (4,6),\\

          O\left(\frac{\ln\lambda_{\varepsilon}}{\lambda_{\varepsilon}^{3}}\right), &\text{~~if~~} N=4 \text{~and~} q\in(2,2^*),\\

     O\left(\frac{1}{\lambda_{\varepsilon}^{\frac{N+2}{2}}}\right), &\text{~~if~~} N\geq 5\text{~and~} q\in (2,2^*).
\end{cases}
    \end{aligned}
\end{equation}
Hence, combining \eqref{lem of estimate of u 1 proof 5}, \eqref{lem of estimate of u 1 proof 6} and \eqref{perturbation estimate of w-varepsilon}, we have
\begin{equation}
    \begin{aligned}
        \int_{B(x_{j,\varepsilon},d)}&U_{x_{j,\varepsilon},\lambda_{j,\varepsilon}}^{2^*-2}(y)\left(\sum_{i\neq j}U_{x_{i,\varepsilon},\lambda_{i,\varepsilon}}(y)+\psi_{x_{j,\varepsilon},\lambda_{j,\varepsilon}}(y)+|w_{\varepsilon}|(y)\right)dy\\
   &=\begin{cases}
     O\left(\frac{1}{\lambda_{\varepsilon}^{\frac{3}{2}}}\right), &\text{~~if~~} N=3\text{~and~} q\in (4,6),\\

         O\left(\frac{1}{\lambda_{\varepsilon}^{2q-2}}\right), &\text{~~if~~} N=4 \text{~and~} q\in (2,\frac{5}{2}),\\
          O\left(\frac{\ln\lambda_{\varepsilon}}{\lambda_{\varepsilon}^{3}}\right), &\text{~~if~~} N=4 \text{~and~} q\in[\frac{5}{2},4),\\

     O\left(\frac{1}{\lambda_{\varepsilon}^{\frac{N+2}{2}}}\right), &\text{~~if~~} N\geq 5\text{~and~} q\in (2,2^*).
 \end{cases}
    \end{aligned}
\end{equation}
Finally, we need to estimate the second term in \eqref{lem of estimate of u 1 proof 2}. It's easy to see that
\begin{equation}\label{lem of estimate of u 1 proof 8}
    \int_{B(x_{j,\varepsilon},d)}\sum_{i\neq j}U_{x_{i,\varepsilon},\lambda_{i,\varepsilon}}^{2^*-1}(y)+\psi_{x_{j,\varepsilon},\lambda_{j,\varepsilon}}^{2^*-1}(y)dy=O\left( \frac{1}{\lambda_{\varepsilon}^{\frac{N+2}{2}}}\right).
\end{equation}
On the other hand, by \eqref{perturbation estimate of w-varepsilon} and H\"older inequality
\begin{equation}\label{lem of estimate of u 1 proof 9}
     \int_{B(x_{j,\varepsilon},d)}|w_{\varepsilon}|^{2^*-1}(y)dy=O\left( \frac{1}{\lambda_{\varepsilon}^{\frac{N+2}{2}}}\right).
\end{equation}
Hence combining \eqref{lem of estimate of u 1 proof 8} and \eqref{lem of estimate of u 1 proof 9} together, we have
\begin{equation}
    \int_{B(x_{j,\varepsilon},d)}\sum_{i\neq j}U_{x_{i,\varepsilon},\lambda_{i,\varepsilon}}^{2^*-1}(y)+\psi_{x_{j,\varepsilon},\lambda_{j,\varepsilon}}^{2^*-1}(y)+|w_{\varepsilon}|^{2^*-1}(y)dy=O\left( \frac{1}{\lambda_{\varepsilon}^{\frac{N+2}{2}}}\right).
\end{equation}
and then \eqref{lem of estimate of u 1-1} follows.

\end{proof}

\begin{Lem}\label{lem of estimate of u 2} For any $j=1,\cdots,n$ and $i=1,\cdots,N$, it holds
   \begin{equation}\label{eq of estimate of u 2}
    \int_{B(x_{j,\varepsilon},d)}\big(y_i-(x_{j,\varepsilon})_{i}\big)  u_{\varepsilon}^{2^*-1}(y) dy=\begin{cases}
     O\left(\frac{\ln\lambda_{\varepsilon}}{\lambda_{\varepsilon}^{\frac{5}{2}}}\right), &\text{~~if~~} N=3\text{~and~} q\in (4,6),\\
     
     O\left(\frac{1}{\lambda_{\varepsilon}^{\frac{N+2}{2}}}\right), &\text{~~if~~} N\geq 4\text{~and~} q\in (2,2^*).
\end{cases}
\end{equation} 
\end{Lem}
\begin{proof}
Using Lemma \ref{Elementary estimate}, decomposition \eqref{decomposition of u} and $PU_{x_{j,\varepsilon},\lambda_{j,\varepsilon}}=U_{x_{j,\varepsilon},\lambda_{j,\varepsilon}}-\psi_{x_{j,\varepsilon},\lambda_{j,\varepsilon}}$, we have
\begin{equation}
\int_{B(x_{j,\varepsilon},d)}\big(y_i-(x_{j,\varepsilon})_{i}\big)  u_{\varepsilon}^{2^*-1}(y) dy=\int_{B(x_{j,\varepsilon},d)}\big(y_i-(x_{j,\varepsilon})_{i}\big)U_{x_{j,\varepsilon},\lambda_{j,\varepsilon}}^{2^*-1}(y)dy+\mathrm{I}_{1},
\end{equation}
where the term $\textrm{I}_{1}$ can be controlled by
\begin{equation}
    \begin{aligned}       \textrm{I}_{1}\lesssim&\int_{B(x_{j,\varepsilon},d)}|y-x_{j,\varepsilon}|U_{x_{j,\varepsilon},\lambda_{j,\varepsilon}}^{2^*-2}(y)\left(\sum_{i\neq j}U_{x_{i,\varepsilon},\lambda_{i,\varepsilon}}(y)+\psi_{x_{j,\varepsilon},\lambda_{j,\varepsilon}}(y)+|w_{\varepsilon}|(y)\right)dy\\
        &+\int_{B(x_{j,\varepsilon},d)}|y-x_{j,\varepsilon}|\left(\sum_{i\neq j}U_{x_{i,\varepsilon},\lambda_{i,\varepsilon}}^{2^*-1}(y)+\psi_{x_{j,\varepsilon},\lambda_{j,\varepsilon}}^{2^*-1}(y)+|w_{\varepsilon}|^{2^*-1}(y)\right)dy.
    \end{aligned}
\end{equation}
Note that from the symmetry, we have
\begin{equation}
    \int_{B(x_{j,\varepsilon},d)}\big(y_i-(x_{j,\varepsilon})_{i}\big)U_{x_{j,\varepsilon},\lambda_{j,\varepsilon}}^{2^*-1}(y)dy=0,
\end{equation}
thus we only need to estimate the term $\textrm{I}_{1}$. First, we estimate the first term on the right hand side of $\textrm{I}_{1}$ . Using H\"older inequality and Sobolev inequality 

\begin{equation}\label{lem of estimate of u 2 proof 5}
  \begin{aligned}
\int_{B(x_{j,\varepsilon},d)}&|y-x_{j,\varepsilon}|U_{x_{j,\varepsilon},\lambda_{j,\varepsilon}}^{2^*-2}(y)|w_{\varepsilon}|(y)dy \\
& \leq \left(\int_{B(x_{j,\varepsilon},d)}|y-x_{j,\varepsilon}|^{\frac{2N}{N+2}}U_{x_{j,\varepsilon},\lambda_{j,\varepsilon}}^{\frac{8N}{(N-2)(N+2)}}(y)dy\right)^{\frac{N+2}{2N}}\|w_{\varepsilon}\|_{L^{2^*}(\Omega)}\\
&=\begin{cases}
    O(\lambda_{j,\varepsilon}^{-\frac{3}{2}}),&\text{~~if~~}N=3\\
    O(\lambda_{j,\varepsilon}^{-2}(\ln\lambda_{j,\varepsilon})^{3/4}),&\text{~~if~~}N=4\\
    O(\lambda_{j,\varepsilon}^{-2}),&\text{~~if~~}N\geq 5\\
\end{cases}\times \|w_{\varepsilon}\|_{H^{1}_{0}(\Omega)}=O\left(\frac{1}{\lambda_{\varepsilon}^{\frac{N+2}{2}}}\right).
  \end{aligned}
\end{equation}       
On the other hand, for any $i\neq j$, by Lemma \ref{estimate of U-lambda-a and psi-lambda-a 3}, we know that 
\begin{equation}\label{lem of estimate of u 2 proof 6}
    \begin{aligned}
\int_{B(x_{j,\varepsilon},d)}&|y-x_{j,\varepsilon}|U_{x_{j,\varepsilon},\lambda_{j,\varepsilon}}^{2^*-2}(y)\left(U_{x_{i,\varepsilon},\lambda_{i,\varepsilon}}(y)+\psi_{x_{j,\varepsilon},x_{j,\varepsilon}}(y)\right)dy\\
&\lesssim \lambda_{\varepsilon}^{-\frac{N-2}{2}}\int_{B(x_{j,\varepsilon},d)}|y-x_{j,\varepsilon}|U_{x_{j,\varepsilon},\lambda_{j,\varepsilon}}^{2^*-2}(y)dy\\
&=\begin{cases}
     O\left(\frac{\ln\lambda_{\varepsilon}}{\lambda_{\varepsilon}^{\frac{5}{2}}}\right), &\text{~~if~~} N=3\text{~and~} q\in (4,6),\\
     
     O\left(\frac{1}{\lambda_{\varepsilon}^{\frac{N+2}{2}}}\right), &\text{~~if~~} N\geq 4\text{~and~} q\in (2,2^*).
\end{cases}
    \end{aligned}
\end{equation}
Moreover, from the proof of Lemma \ref{lem of estimate of u 1}, we get that 
\begin{equation}
  \int_{B(x_{j,\varepsilon},d)}|y-x_{j,\varepsilon}|\left(\sum_{i\neq j}U_{x_{i,\varepsilon},\lambda_{i,\varepsilon}}^{2^*-1}(y)+\psi_{x_{j,\varepsilon},\lambda_{j,\varepsilon}}^{2^*-1}(y)+|w_{\varepsilon}|^{2^*-1}(y)\right)dy=O\left(\frac{1}{\lambda_{\varepsilon}^{\frac{N+2}{2}}}\right).  
\end{equation}
Now, combining all the estimates above, we get \eqref{eq of estimate of u 2}.
    \end{proof}

\begin{Lem}\label{lem of estimate of u 3} For any $j=1,\cdots,n$, $i,l=1,\cdots,N$ and $i\neq l$, it holds
   \begin{equation}\label{eq of estimate of u 3}
       \int_{B(x_{j,\varepsilon},d)}\big(y_i-(x_{j,\varepsilon})_{i}\big)
\big(y_l-(x_{j,\varepsilon})_{l}\big)  u_{\varepsilon}^{2^*-1}(y) dy=O\left(\frac{1}{\lambda_{\varepsilon}^{\frac{N+2}{2}}}\right).
   \end{equation} 
\end{Lem}
    \begin{proof}
Using Lemma \ref{Elementary estimate}, decomposition \eqref{decomposition of u} and $PU_{x_{j,\varepsilon},\lambda_{j,\varepsilon}}=U_{x_{j,\varepsilon},\lambda_{j,\varepsilon}}-\psi_{x_{j,\varepsilon},\lambda_{j,\varepsilon}}$, we have
\begin{equation}
\begin{split}
    \int_{B(x_{j,\varepsilon},d)}&\big(y_i-(x_{j,\varepsilon})_{i}\big)\big(y_l-(x_{j,\varepsilon})_{l}\big)  u_{\varepsilon}^{2^*-1}(y) dy\\
    &=\int_{B(x_{j,\varepsilon},d)}\big(y_i-(x_{j,\varepsilon})_{i}\big)\big(y_l-(x_{j,\varepsilon})_{l}\big)U_{x_{j,\varepsilon},\lambda_{j,\varepsilon}}^{2^*-1}(y)dy+\mathrm{I}_{2},
\end{split}
\end{equation}
where the term $\textrm{I}_{2}$ can be controlled by
\begin{equation}
    \begin{aligned}       \textrm{I}_{2}\lesssim&\int_{B(x_{j,\varepsilon},d)}|y-x_{j,\varepsilon}|^{2}U_{x_{j,\varepsilon},\lambda_{j,\varepsilon}}^{2^*-2}(y)\left(\sum_{i\neq j}U_{x_{i,\varepsilon},\lambda_{i,\varepsilon}}(y)+\psi_{x_{j,\varepsilon},\lambda_{j,\varepsilon}}(y)+|w_{\varepsilon}|(y)\right)dy\\
        &+\int_{B(x_{j,\varepsilon},d)}|y-x_{j,\varepsilon}|^{2}\left(\sum_{i\neq j}U_{x_{i,\varepsilon},\lambda_{i,\varepsilon}}^{2^*-1}(y)+\psi_{x_{j,\varepsilon},\lambda_{j,\varepsilon}}^{2^*-1}(y)+|w_{\varepsilon}|^{2^*-1}(y)\right)dy.
    \end{aligned}
\end{equation}
Note that from the symmetry, we have
\begin{equation}
    \int_{B(x_{j,\varepsilon},d)}\big(y_i-(x_{j,\varepsilon})_{i}\big)\big(y_l-(x_{j,\varepsilon})_{l}\big)U_{x_{j,\varepsilon},\lambda_{j,\varepsilon}}^{2^*-1}(y)dy=0,
\end{equation}
thus we only need to estimate the term $\textrm{I}_{2}$. First, from the proof of Lemma \ref{lem of estimate of u 1} and Lemma \ref{lem of estimate of u 2}, we get that 
\begin{equation}
\begin{split}
    &\int_{B(x_{j,\varepsilon},d)}|y-x_{j,\varepsilon}|^{2}U_{x_{j,\varepsilon},\lambda_{j,\varepsilon}}^{2^*-2}(y)|w_{\varepsilon}|(y)dy\\
    &+\int_{B(x_{j,\varepsilon},d)}|y-x_{j,\varepsilon}|^{2}\left(\sum_{i\neq j}U_{x_{i,\varepsilon},\lambda_{i,\varepsilon}}^{2^*-1}(y)+\psi_{x_{j,\varepsilon},\lambda_{j,\varepsilon}}^{2^*-1}(y)+|w_{\varepsilon}|^{2^*-1}(y)\right)dy=O\left(\frac{1}{\lambda_{\varepsilon}^{\frac{N+2}{2}}}\right).  
\end{split}
\end{equation}
Moreover, for any $i\neq j$, by Lemma \ref{estimate of U-lambda-a and psi-lambda-a 3}, we know that 
\begin{equation}\label{lem of estimate of u 3 proof 6}
    \begin{aligned}
\int_{B(x_{j,\varepsilon},d)}&|y-x_{j,\varepsilon}|^{2}U_{x_{j,\varepsilon},\lambda_{j,\varepsilon}}^{2^*-2}(y)\left(U_{x_{i,\varepsilon},\lambda_{i,\varepsilon}}(y)+\psi_{x_{j,\varepsilon},x_{j,\varepsilon}}(y)\right)dy\\
&\lesssim \lambda_{\varepsilon}^{-\frac{N-2}{2}}\int_{B(x_{j,\varepsilon},d)}|y-x_{j,\varepsilon}|^{2}U_{x_{j,\varepsilon},\lambda_{j,\varepsilon}}^{2^*-2}(y)dy=O\left(\frac{1}{\lambda_{\varepsilon}^{\frac{N+2}{2}}}\right).
    \end{aligned}
\end{equation}
Hence, combining all the estimates above, we obtain \eqref{eq of estimate of u 3}.
    \end{proof}
\begin{Lem}\label{lem third taylor}
    For any $j=1,\cdots,n$ and fixed $d>0$ small enough, it holds
    \begin{equation}
        \int_{B(x_{j,\varepsilon},d)}|y-x_{j,\varepsilon}|^{3}u_{\varepsilon}^{2^*-1}(y)dy=O\left(\frac{1}{\lambda_{\varepsilon}^{\frac{N+2}{2}}}\right).
    \end{equation}
\end{Lem}
\begin{proof}
  Note that $u_{\varepsilon}\lesssim U_{x_{j,\varepsilon},\lambda_{j,\varepsilon}}$ in $B(x_{j,\varepsilon},d)$, then by direct computation, we get that
    \begin{equation}
    \begin{aligned}
        \int_{B(x_{j,\varepsilon},d)}|y-x_{j,\varepsilon}|^{3}u_{\varepsilon}^{2^*-1}(y)dy&\lesssim\lambda_{\varepsilon}^{\frac{N+2}{2}-N-3}\int_{B(0,\lambda_{j,\varepsilon}d)}|y|^{3}U_{0,1}^{2^*-1}(y)dy\\&=O\left(\frac{1}{\lambda_{\varepsilon}^{\frac{N+2}{2}}}\right).
    \end{aligned}
    \end{equation}
Thus we complete the proof.
\end{proof}

\begin{Lem}\label{lem estimate of u q-1}
    For any $j=1,\cdots,n$ and fixed $d>0$ small enough, it holds
     \begin{equation}
    \begin{aligned}
        \varepsilon\int_{B(x_{j,\varepsilon},d)}u_{\varepsilon}^{q-1}(y)dy=\begin{cases}
            O\left(\frac{1}{\lambda_{\varepsilon}^{\frac{2(N-2)q-(N+2)}{2}}}\right),&\text{~~if~~}q\in\left(2,\frac{2N-2}{N-2}\right),\\
            O\left(\frac{\ln\lambda_{\varepsilon}}{\lambda_{\varepsilon}^{\frac{3N-6}{2}}}\right),&\text{~~if~~}q=\frac{2N-2}{N-2},\\
            O\left(\frac{1}{\lambda_{\varepsilon}^{\frac{3N-6}{2}}}\right),&\text{~~if~~}q\in\left(\frac{2N-2}{N-2},2^*\right).\\
        \end{cases}
    \end{aligned}
    \end{equation}
\end{Lem}
\begin{proof}
  Note that $u_{\varepsilon}\lesssim U_{x_{j,\varepsilon},\lambda_{j,\varepsilon}}$ in $B(x_{j,\varepsilon},d)$, then by direct computation, we get that
    \begin{equation}
    \begin{aligned}
        \int_{B(x_{j,\varepsilon},d)}u_{\varepsilon}^{q-1}(y)dy&\lesssim\int_{B(x_{j,\varepsilon},d)}U_{x_{j,\varepsilon},\lambda_{j,\varepsilon}}^{q-1}(y)dy\\&
        =\begin{cases}
            O\left(\frac{1}{\lambda_{\varepsilon}^{\frac{(N-2)(q-1)}{2}}}\right),&\text{~~if~~}q\in\left(2,\frac{2N-2}{N-2}\right),\\
            O\left(\frac{\ln\lambda_{\varepsilon}}{\lambda_{\varepsilon}^{N-\frac{(N-2)(q-1)}{2}}}\right),&\text{~~if~~}q=\frac{2N-2}{N-2},\\
            O\left(\frac{1}{\lambda_{\varepsilon}^{N-\frac{(N-2)(q-1)}{2}}}\right),&\text{~~if~~}q\in\left(\frac{2N-2}{N-2},2^*\right).\\
        \end{cases}
    \end{aligned}
    \end{equation}
    Therefore, by Corollary \ref{cor of the decomposition}, we get 
    \begin{equation}
    \begin{aligned}
        \varepsilon\int_{B(x_{j,\varepsilon},d)}u_{\varepsilon}^{q-1}(y)dy=\begin{cases}
            O\left(\frac{1}{\lambda_{\varepsilon}^{\frac{2(N-2)q-(N+2)}{2}}}\right),&\text{~~if~~}q\in\left(2,\frac{2N-2}{N-2}\right),\\
            O\left(\frac{\ln\lambda_{\varepsilon}}{\lambda_{\varepsilon}^{\frac{3N-6}{2}}}\right),&\text{~~if~~}q=\frac{2N-2}{N-2},\\
            O\left(\frac{1}{\lambda_{\varepsilon}^{\frac{3N-6}{2}}}\right),&\text{~~if~~}q\in\left(\frac{2N-2}{N-2},2^*\right).\\
        \end{cases}
    \end{aligned}
    \end{equation}
Thus we complete the proof.
\end{proof}

\begin{Lem}\label{lem estimate of q refined}
   For any $j=1,\cdots,n$ and fixed $d>0$ small enough, it holds
    \begin{equation}\label{lem estimate of q refined-1}
    \begin{aligned}
\int_{B(x_{j,\varepsilon},d)}u_{\varepsilon}^{q}
    &=\frac{B}{\lambda_{j,\varepsilon}^{N-\frac{N-2}{2}q}}+\begin{cases}
      O\left(\frac{1}{\lambda_{j,\varepsilon}^{4-\frac{q}{2}}}\right),&\text{~~if~~}N=3\text{~and~}q\in(4,6),\\
    O\left(\frac{1}{\lambda_{\varepsilon}^{\frac{(N-2)q}{2}}}\right),&\text{~~if~~}4\leq N\leq 6\text{~and~}q\in\left(2,\frac{2N-2}{N-2}\right),\\

            O\left(\frac{\ln\lambda_{\varepsilon}}{\lambda_{\varepsilon}^{2N-2-\frac{(N-2)q}{2}}}\right),&\text{~~if~~}4\leq N\leq 6\text{~and~}q\in\left[\frac{2N-2}{N-2},2^*\right),\\
             O\left(\frac{1}{\lambda_{\varepsilon}^{\frac{(N-2)q}{2}}}\right), &\text{~~if~~} 7\leq N\leq 9\text{~and~} q\in (2,\frac{3N+2}{2N-4}),\\
      O\left(\frac{\ln\lambda_{\varepsilon}}{\lambda_{\varepsilon}^{\frac{3N+2-(N-2)q}{2}}}\right), &\text{~~if~~} 7\leq N\leq 9\text{~and~} q\in [\frac{3N+2}{2N-4},2^*),\\
O\left(\frac{\ln\lambda_{\varepsilon}}{\lambda_{\varepsilon}^{\frac{3N+2-(N-2)q}{2}}}\right), &\text{~~if~~} N\geq 10\text{~and~} q\in (2,2^*),
\end{cases} 
    \end{aligned}
\end{equation}
where $B$ is a constant defined in \eqref{definition of A and B}.
\end{Lem}
\begin{proof}
Using decomposition \eqref{decomposition of u}, $PU_{x_{j,\varepsilon},\lambda_{j,\varepsilon}}=U_{x_{j,\varepsilon},\lambda_{j,\varepsilon}}-\psi_{x_{j,\varepsilon},\lambda_{j,\varepsilon}}$ and Lemma \ref{Elementary estimate}, we have
\begin{equation}\label{lem of estimate of u q proof 1}
\int_{B(x_{j,\varepsilon},d)}u_{\varepsilon}^{q}(y)dy=\int_{B(x_{j,\varepsilon},d)}U_{x_{j,\varepsilon},\lambda_{j,\varepsilon}}^{q}(y)dy+\mathrm{II},
\end{equation}
where the term $\mathrm{II}$ can be controlled by
\begin{equation}\label{lem of estimate of u q proof 3}
    \begin{aligned}
\mathrm{II}\lesssim&\int_{B(x_{j,\varepsilon},d)}U_{x_{j,\varepsilon},\lambda_{j,\varepsilon}}^{q-1}(y)\left(\sum_{i\neq j}U_{x_{i,\varepsilon},\lambda_{i,\varepsilon}}(y)+\psi_{x_{j,\varepsilon},\lambda_{j,\varepsilon}}(y)+|w_{\varepsilon}|(y)\right)dy\\
        &+\int_{B(x_{j,\varepsilon},d)}\sum_{i\neq j}U_{x_{i,\varepsilon},\lambda_{i,\varepsilon}}^{q}(y)+\psi_{x_{j,\varepsilon},\lambda_{j,\varepsilon}}^{q}(y)+|w_{\varepsilon}|^{q}(y)dy.
    \end{aligned}
\end{equation}
First, we estimate the first term on the right hand side of \eqref{lem of estimate of u q proof 1} 
\begin{equation}\label{lem of estimate of u q proof 4}
    \begin{aligned}
        \int_{B(x_{j,\varepsilon},d)}U_{x_{j,\varepsilon},\lambda_{j,\varepsilon}}^{q}(y)dy&=\lambda_{j,\varepsilon}^{\frac{N-2}{2}q-N}\int_{B(0,\lambda_{j,\varepsilon}d)}U_{0,1}^{q}(y)dy\\
        &=B\lambda_{j,\varepsilon}^{\frac{N-2}{2}q-N}+O(\lambda_{j,\varepsilon}^{-\frac{N-2}{2}q}),
    \end{aligned}
\end{equation}
where $B$ is  defined in \eqref{definition of A and B}. Next, we estimate the first term on the right hand side of \eqref{lem of estimate of u q proof 3}. Using H\"older inequality and Sobolev inequality, we get 

\begin{equation}\label{lem of estimate of u q proof 5}
  \begin{aligned}
\int_{B(x_{j,\varepsilon},d)}U_{x_{j,\varepsilon},\lambda_{j,\varepsilon}}^{q-1}(y)w_{\varepsilon}(y)dy & \leq \left(\int_{B(x_{j,\varepsilon},d)}U_{x_{j,\varepsilon},\lambda_{j,\varepsilon}}^{\frac{2(q-1)N}{N+2}}(y)dy\right)^{\frac{N+2}{2N}}\|w_{\varepsilon}\|_{L^{2^*}(\Omega)}\\
&=\begin{cases}
    O\left(\lambda_{j,\varepsilon}^{-\frac{N-2}{2}(q-1)}\right), &\text{~if~} q<\frac{3N-2}{2N-4},\\
    O\left(\lambda_{j,\varepsilon}^{\frac{N-2}{2}q-N}(\ln\lambda_{j,\varepsilon})^{\frac{N+2}{2N}}\right), &\text{~if~} q=\frac{3N-2}{2N-4},\\
    O\left(\lambda_{j,\varepsilon}^{\frac{N-2}{2}q-N}\right), &\text{~if~} q\in (\frac{3N-2}{2N-4},2^*)\\
\end{cases}\times\|w_{\varepsilon}\|_{H^{1}_{0}(\Omega)}.\\
  \end{aligned}
\end{equation}       
On the other hand, for any $i\neq j$, by Lemma \ref{estimate of U-lambda-a and psi-lambda-a 3}, we know that 
\begin{equation}\label{lem of estimate of u q proof 6}
    \begin{aligned}
\int_{B(x_{j,\varepsilon},d)}&U_{x_{j,\varepsilon},\lambda_{j,\varepsilon}}^{q-1}(y)\left(U_{x_{i,\varepsilon},\lambda_{i,\varepsilon}}(y)+\psi_{x_{j,\varepsilon},x_{j,\varepsilon}}(y)\right)dy\\
&\lesssim \lambda_{\varepsilon}^{-\frac{N-2}{2}}\int_{B(x_{j,\varepsilon},d)}U_{x_{j,\varepsilon},\lambda_{j,\varepsilon}}^{q-1}(y)dy\\
&=\begin{cases}
            O\left(\frac{1}{\lambda_{\varepsilon}^{\frac{(N-2)q}{2}}}\right),&\text{~~if~~}q\in\left(2,\frac{2N-2}{N-2}\right),\\
            O\left(\frac{\ln\lambda_{\varepsilon}}{\lambda_{\varepsilon}^{2N-2-\frac{(N-2)q}{2}}}\right),&\text{~~if~~}q=\frac{2N-2}{N-2},\\
            O\left(\frac{1}{\lambda_{\varepsilon}^{2N-2-\frac{(N-2)q}{2}}}\right),&\text{~~if~~}q\in\left(\frac{2N-2}{N-2},2^*\right).\\
        \end{cases}
    \end{aligned}
\end{equation}
Finally, we estimate the second term in \eqref{lem of estimate of u q proof 3}. It's easy to see that
\begin{equation}\label{lem of estimate of u q proof 8}
    \int_{B(x_{j,\varepsilon},d)}\sum_{i\neq j}U_{x_{i,\varepsilon},\lambda_{i,\varepsilon}}^{q}(y)+\psi_{x_{j,\varepsilon},\lambda_{j,\varepsilon}}^{q}(y)dy=O\left( \frac{1}{\lambda_{\varepsilon}^{\frac{N-2}{2}q}}\right).
\end{equation}
On the other hand, by \eqref{perturbation estimate of w-varepsilon}, we get
\begin{equation}\label{lem of estimate of u q proof 9}
     \int_{B(x_{j,\varepsilon},d)}|w_{\varepsilon}|^{q}(y)dy=O\left( \frac{1}{\lambda_{\varepsilon}^{\frac{N-2}{2}q}}\right).
\end{equation}
Hence by \eqref{lem of estimate of u q proof 8} and \eqref{lem of estimate of u q proof 9}, we get
\begin{equation}
    \int_{B(x_{j,\varepsilon},d)}\sum_{i\neq j}U_{x_{i,\varepsilon},\lambda_{i,\varepsilon}}^{q}(y)+\psi_{x_{j,\varepsilon},\lambda_{j,\varepsilon}}^{2^*-1}(y)+|w_{\varepsilon}|^{q}(y)dy=O\left( \frac{1}{\lambda_{\varepsilon}^{\frac{N-2}{2}q}}\right).
\end{equation}
Combining all the estimates above and using \eqref{perturbation estimate of w-varepsilon}, then we can obtained \eqref{lem estimate of q refined-1}.
\end{proof}

\subsection{Proof of Theorem \ref{thm blow up rate}}
\begin{Prop}\label{prop estimate for u}
Let $u_{\varepsilon}$ be a solution of \eqref{p-varepsion}, then for any fixed $d>0$ small enough, it holds
\begin{equation}
    u_{\varepsilon}=A\left(\sum_{j=1}^{n}\frac{G(x_{j,\varepsilon},x)}{\lambda_{j,\varepsilon}^{\frac{N-2}{2}}}\right)+\begin{cases}
     O\left(\frac{1}{\lambda_{\varepsilon}^{\frac{3}{2}}}\right), &\text{~~if~~} N=3\text{~and~} q\in (4,6),\\

         O\left(\frac{1}{\lambda_{\varepsilon}^{2q-3}}\right), &\text{~~if~~} N=4 \text{~and~} q\in(2,3),\\
          O\left(\frac{\ln\lambda_{\varepsilon}}{\lambda_{\varepsilon}^{3}}\right), &\text{~~if~~} N=4 \text{~and~} q\in[3,4),\\
     
      O\left(\frac{1}{\lambda_{\varepsilon}^{\frac{6q-7}{2}}}\right), &\text{~~if~~} N=5\text{~and~} q\in (2,\frac{7}{3}),\\
     O\left(\frac{1}{\lambda_{\varepsilon}^{\frac{7}{2}}}\right), &\text{~~if~~} N=5\text{~and~} q\in [\frac{7}{3},\frac{10}{3}),\\
     O\left(\frac{1}{\lambda_{\varepsilon}^{\frac{N+2}{2}}}\right), &\text{~~if~~} N\geq 6\text{~and~} q\in (2,2^*),
 \end{cases}
\end{equation}
$\text{in~}C^{1}(\Omega\setminus\cup_{j=1}^{n}B(x_{j,\varepsilon},2d))$.
\end{Prop}
\begin{proof}
First  for any $x\in \Omega\backslash\displaystyle\bigcup^n_{j=1}B(x_{j,\varepsilon},2d)$, we have
\begin{equation}\label{prop estimate for u proof-1}
\begin{split}
u_\varepsilon(x)=& \int_{\Omega}G(y,x)
\big(u_{\varepsilon}^{2^*-1}(y)+\varepsilon u_{\varepsilon}^{q-1} (y)\big) dy\\=&
\sum^n_{j=1}\int_{B(x_{j,\varepsilon},d)}G(y,x)
u_{\varepsilon}^{2^*-1}(y)dy+\int_{\Omega\backslash\bigcup^n_{j=1}B(x_{j,\varepsilon},d)}G(y,x)
u_{\varepsilon}^{2^*-1}(y)dy\\&
+ \varepsilon \sum^n_{j=1}\int_{B(x_{j,\varepsilon},d)}G(y,x)u_{\varepsilon}^{q-1} (y)dy+\varepsilon \int_{\Omega\backslash\bigcup^n_{j=1}B(x_{j,\varepsilon},d)}G(y,x)u_{\varepsilon}^{q-1} (y)dy.
\end{split}
\end{equation}
By Taylor's expansion, the first term in \eqref{prop estimate for u proof-1} can be rewrite as
\begin{equation}\label{prop estimate for u proof-2}
\begin{split}
\int_{B(x_{j,\varepsilon},d)} &G(y,x) u_{\varepsilon}^{2^*-1}(y) dy\\=&
  G(x_{j,\varepsilon},x)\int_{B(x_{j,\varepsilon},d)}
  u_{\varepsilon}^{2^*-1}+\sum^N_{i=1}\partial_{y_i} G(x_{j,\varepsilon},x)
\int_{B(x_{j,\varepsilon},d)}\big(y_i-(x_{j,\varepsilon})_{i}\big)  u_{\varepsilon}^{2^*-1}(y) dy\\&
+\sum^N_{i=1}\sum^N_{l=1}\partial^2_{y_i,y_l} G(x_{j,\varepsilon},x) \int_{B(x_{j,\varepsilon},d)}\big(y_i-(x_{j,\varepsilon})_{i}\big)
\big(y_l-(x_{j,\varepsilon})_{l}\big)  u_{\varepsilon}^{2^*-1}(y) dy\\&
+O\Big(\int_{B(x_{j,\varepsilon},d)} |y-x_{j,\varepsilon}|^3 u_{\varepsilon}^{2^*-1}(y) dy\Big).
\end{split}
\end{equation}
Hence, from Lemma \ref{lem of estimate of u 1}-Lemma \ref{lem third taylor} and the fact that $\displaystyle\sum^N_{i=1}\partial^2_{y_i,y_i} G(x_{j,\varepsilon},x) =0$ for $x\in \Omega\backslash B(x_{j,\varepsilon},d)$, we have
\begin{equation}
    \int_{B(x_{j,\varepsilon},d)}G(y,x) u_{\varepsilon}^{2^*-1}(y) dy=A\frac{G(x_{j,\varepsilon},x)}{\lambda_{j,\varepsilon}^{\frac{N-2}{2}}}+\begin{cases}
     O\left(\frac{1}{\lambda_{\varepsilon}^{\frac{3}{2}}}\right), &\text{~~if~~} N=3\text{~and~} q\in (4,6),\\

         O\left(\frac{1}{\lambda_{\varepsilon}^{2q-2}}\right), &\text{~~if~~} N=4 \text{~and~} q\in (2,\frac{5}{2}),\\
          O\left(\frac{\ln\lambda_{\varepsilon}}{\lambda_{\varepsilon}^{3}}\right), &\text{~~if~~} N=4 \text{~and~} q\in[\frac{5}{2},4),\\

     O\left(\frac{1}{\lambda_{\varepsilon}^{\frac{N+2}{2}}}\right), &\text{~~if~~} N\geq 5\text{~and~} q\in (2,2^*),
 \end{cases}
\end{equation}
where $A$ is the  constant defined in \eqref{definition of A and B}.
Next for $x\in \Omega\backslash\displaystyle\bigcup^n_{j=1}B_{2d}(x_{j,\varepsilon})$, from Lemma \ref{lem estimate of u q-1}, we have
\begin{equation}\label{prop estimate for u proof-5}
\begin{split}
 \varepsilon\int_{B(x_{j,\varepsilon},d)}G(y,x) u_{\varepsilon}^{q-1} (y) dy &=O\Big(\varepsilon
 \int_{B(x_{j,\varepsilon},d)} u_{\varepsilon}^{q-1} (y) dy\Big)
 \\&=\begin{cases}
            O\left(\frac{1}{\lambda_{\varepsilon}^{\frac{2(N-2)q-(N+2)}{2}}}\right),&\text{~~if~~}q\in\left(2,\frac{2N-2}{N-2}\right),\\
            O\left(\frac{\ln\lambda_{\varepsilon}}{\lambda_{\varepsilon}^{\frac{3N-6}{2}}}\right),&\text{~~if~~}q=\frac{2N-2}{N-2},\\
            O\left(\frac{1}{\lambda_{\varepsilon}^{\frac{3N-6}{2}}}\right),&\text{~~if~~}q\in\left(\frac{2N-2}{N-2},2^*\right).\\
        \end{cases}
\end{split}
\end{equation}
Finally, it holds that
\begin{equation}
\begin{aligned}
     \int_{\Omega\backslash\bigcup^n_{j=1}B(x_{j,\varepsilon},d)}&G(y,x)
u_{\varepsilon}^{2^*-1}(y)dy+\varepsilon\int_{\Omega\backslash\bigcup^n_{j=1}B(x_{j,\varepsilon},d)}G(y,x)
u_{\varepsilon}^{q-1}(y)dy\\
&\lesssim \frac{1}{\lambda_{\varepsilon}^{\frac{N+2}{2}}}+\varepsilon\frac{1}{\lambda_{\varepsilon}^{\frac{N-2}{2}(q-1)}}\lesssim\frac{1}{\lambda_{\varepsilon}^{\frac{N+2}{2}}}+\frac{1}{\lambda_{\varepsilon}^{(N-2)q-\frac{N+2}{2}}}.
\end{aligned}  
\end{equation}
Combining all the estimates above
\begin{equation}\label{prop estimate for u proof-8}
    u_{\varepsilon}=A\left(\sum_{j=1}^{n}\frac{G(x_{j,\varepsilon},x)}{\lambda_{j,\varepsilon}^{\frac{N-2}{2}}}\right)+\begin{cases}
     O\left(\frac{1}{\lambda_{\varepsilon}^{\frac{3}{2}}}\right), &\text{~~if~~} N=3\text{~and~} q\in (4,6),\\

         O\left(\frac{1}{\lambda_{\varepsilon}^{2q-3}}\right), &\text{~~if~~} N=4 \text{~and~} q\in(2,3),\\
          O\left(\frac{\ln\lambda_{\varepsilon}}{\lambda_{\varepsilon}^{3}}\right), &\text{~~if~~} N=4 \text{~and~} q\in[3,4),\\
     
      O\left(\frac{1}{\lambda_{\varepsilon}^{\frac{6q-7}{2}}}\right), &\text{~~if~~} N=5\text{~and~} q\in (2,\frac{7}{3}),\\
     O\left(\frac{1}{\lambda_{\varepsilon}^{\frac{7}{2}}}\right), &\text{~~if~~} N=5\text{~and~} q\in [\frac{7}{3},\frac{10}{3}),\\
     O\left(\frac{1}{\lambda_{\varepsilon}^{\frac{N+2}{2}}}\right), &\text{~~if~~} N\geq 6\text{~and~} q\in (2,2^*),
 \end{cases}
\end{equation}
$\text{in~}C(\Omega\setminus\cup_{j=1}^{n}B(x_{j,\varepsilon},2d))$. Notice that
\begin{equation}
    \begin{split}
\nabla u_\varepsilon(x)= \int_{\Omega}\nabla_{x}G(y,x)
\big(u_{\varepsilon}^{2^*-1}(y)+\varepsilon u_{\varepsilon}^{q-1} (y)\big) dy,
\end{split}
\end{equation}
thus we can prove that \eqref{prop estimate for u proof-8} is still hold $\text{in~}C^{1}(\Omega\setminus\cup_{j=1}^{n}B(x_{j,\varepsilon},2d))$ similarly.
\end{proof}


\begin{Prop}\label{Prop blow up rate} Let $u_{\varepsilon}$ be a family of solutions of \eqref{p-varepsion} with decomposition \eqref{decomposition of u},  then  ${M}(\vec{a})$  is a non-negative matrix. Moreover if  ${M}(\vec{a})$ is a positive matrix, then (up to a subsequence)
\begin{equation}\label{Prop blow up rate 1}
\lambda_j:=\lim_{\varepsilon\rightarrow 0}\big(\varepsilon ^{\frac{2}{(N-2)q-4}}\lambda_{j,\varepsilon}\big)^{-1},\text{~~for~~}j=1,\cdots,n,
\end{equation}
exists and lies in $(0,\infty)$ and
\begin{equation}\label{Prop blow up rate 2}
\nabla_{\vec{\lambda}} \Phi_n(\vec{a},\vec{\lambda})=\vec{0}~\mbox{with}~\vec{a}=(a_1,\cdots,a_{n})~\mbox{and}~ \vec{\lambda}=(\lambda_{1},\cdots,\lambda_{n}).
\end{equation}
Furthermore, if $q\geq 2^*-1$, then $\vec{\lambda}$ is the unique critical point of $\Phi_{n}(\vec{a},\cdot)$, which is also a minimizer.
\end{Prop}
\begin{proof}
First, from Corollary \ref{cor of the decomposition}, we can obtain that for any $j=1,\cdots,n$
\begin{equation}\label{Prop blow up rate proof 1}
   \Lambda_{j,\varepsilon}:=\big(\varepsilon ^{\frac{2}{(N-2)q-4}}\lambda_{j,\varepsilon}\big)^{-1}\geq C>0.
\end{equation}
Thus we only need to prove that $\Lambda_{j,\varepsilon}$ have a upper bound. Next, we use the local Poho\v{z}aev identity \eqref{pohozaev identity 1} for $u_{\varepsilon}$ with $\Omega=B(x_{j,\varepsilon},\theta)$ and $\theta>0$ small enough. From Proposition \ref{prop estimate for u}, the left-hand side of local Poho\v{z}aev identity \eqref{pohozaev identity 1} is
\begin{equation}\label{Prop blow up rate proof 2}
\begin{split}
 A^{2}\sum^n_{l=1}\sum^n_{m=1}\frac{P\big(G(x_{m,\varepsilon},x),G(x_{l,\varepsilon},x)\big)}
{\lambda^{(N-2)/2}_{m,\varepsilon}\lambda^{(N-2)/2}_{l,\varepsilon}} +\begin{cases}
     O\left(\frac{1}{\lambda_{\varepsilon}^{2}}\right), &\text{~~if~~} N=3\text{~and~} q\in (4,6),\\
         O\left(\frac{1}{\lambda_{\varepsilon}^{2q-2}}\right), &\text{~~if~~} N=4 \text{~and~} q\in(2,3),\\
          O\left(\frac{\ln\lambda_{\varepsilon}}{\lambda_{\varepsilon}^{4}}\right), &\text{~~if~~} N=4 \text{~and~} q\in[3,4),\\ 
      O\left(\frac{1}{\lambda_{\varepsilon}^{3q-2}}\right), &\text{~~if~~} N=5\text{~and~} q\in (2,\frac{7}{3}),\\
     O\left(\frac{1}{\lambda_{\varepsilon}^{5}}\right), &\text{~~if~~} N=5\text{~and~} q\in [\frac{7}{3},\frac{10}{3}),\\
     O\left(\frac{1}{\lambda_{\varepsilon}^{N}}\right), &\text{~~if~~} N\geq 6\text{~and~} q\in (2,2^*),
 \end{cases}
\end{split}
\end{equation}
where $A$ is the constant defined in \eqref{definition of A and B} and $P(\cdot,\cdot)$ is the quadratic form defined in \eqref{defin for quadratic P}. Moreover, using Lemma \ref{lem for estimate of quadractic form}, then \eqref{Prop blow up rate proof 2} now becomes

\begin{equation}\label{Prop blow up rate proof 3}
 -\frac{A^{2}(N-2)}{2}\left(\frac{ R(x_{j,\varepsilon})}{\lambda^{N-2}_{j,\varepsilon}}
-\sum^n_{l\neq j}\frac{G(x_{j,\varepsilon},x_{l,\varepsilon})}{\lambda^{(N-2)/2}_{j,\varepsilon}\lambda^{(N-2)/2}_{l,\varepsilon}}\right)+\begin{cases}
     O\left(\frac{1}{\lambda_{\varepsilon}^{2}}\right), &\text{~~if~~} N=3\text{~and~} q\in (4,6),\\
         O\left(\frac{1}{\lambda_{\varepsilon}^{2q-2}}\right), &\text{~~if~~} N=4 \text{~and~} q\in(2,3),\\
          O\left(\frac{\ln\lambda_{\varepsilon}}{\lambda_{\varepsilon}^{4}}\right), &\text{~~if~~} N=4 \text{~and~} q\in[3,4),\\ 
      O\left(\frac{1}{\lambda_{\varepsilon}^{3q-2}}\right), &\text{~~if~~} N=5\text{~and~} q\in (2,\frac{7}{3}),\\
     O\left(\frac{1}{\lambda_{\varepsilon}^{5}}\right), &\text{~~if~~} N=5\text{~and~} q\in [\frac{7}{3},\frac{10}{3}),\\
     O\left(\frac{1}{\lambda_{\varepsilon}^{N}}\right), &\text{~~if~~} N\geq 6\text{~and~} q\in (2,2^*).
 \end{cases}
\end{equation}
On the other hand, by Proposition \ref{prop estimate for u} and Lemma \ref{lem estimate of q refined}, the right hand side of local Poho\v{z}aev identity \eqref{pohozaev identity 1} becomes
\begin{equation}\label{Prop blow up rate proof 4}
\begin{aligned}
     -\varepsilon\frac{2N-(N-2)q}{2q}\frac{B}{\lambda_{j,\varepsilon}^{N-\frac{N-2}{2}q}}&+ O\left(\frac{1}{\lambda_{\varepsilon}^{N}}+\frac{1}{\lambda_{\varepsilon}^{(N-2)q-2}}\right)\\
     &+\begin{cases}
      O\left(\frac{1}{\lambda_{j,\varepsilon}^{2}}\right),&\text{~~if~~}N=3\text{~and~}q\in(4,6),\\
            O\left(\frac{\ln\lambda_{\varepsilon}}{\lambda_{\varepsilon}^{2N-4}}\right),&\text{~~if~~}4\leq N\leq 6\text{~and~}q\in\left[\frac{2N-2}{N-2},2^*\right),\\
            
      O\left(\frac{\ln\lambda_{\varepsilon}}{\lambda_{\varepsilon}^{\frac{3N-2}{2}}}\right), &\text{~~if~~} 7\leq N\leq 9\text{~and~} q\in [\frac{3N+2}{2N-4},2^*),\\
O\left(\frac{\ln\lambda_{\varepsilon}}{\lambda_{\varepsilon}^{\frac{3N-2}{2}}}\right), &\text{~~if~~} N\geq 10\text{~and~} q\in (2,2^*),
\end{cases} 
\end{aligned}
\end{equation}
where $B$ is the constant defined in \eqref{definition of A and B}. Thus by \eqref{Prop blow up rate proof 3} and \eqref{Prop blow up rate proof 4}, we have
\begin{equation}\label{Prop blow up rate proof 5}
    \frac{ R(x_{j,\varepsilon})}{\lambda^{N-2}_{j,\varepsilon}}
-\sum^n_{l\neq j}\frac{G(x_{j,\varepsilon},x_{l,\varepsilon})}{\lambda^{(N-2)/2}_{j,\varepsilon}\lambda^{(N-2)/2}_{l,\varepsilon}}=\frac{2^*-q}{q}\frac{B}{A^{2}}\frac{\varepsilon}{\lambda_{j,\varepsilon}^{N-\frac{N-2}{2}q}}+o\left(\frac{1}{\lambda_{\varepsilon}^{N-2}}\right).
\end{equation}
Now we define
$\Lambda_{\varepsilon}=\max\big\{\Lambda_{j,\varepsilon}, j=1,\cdots,n\big\}$. Then by \eqref{Prop blow up rate proof 5}, we can obtain that
 \begin{equation}\label{Prop blow up rate proof 6}
 \begin{split}
 \Lambda^{N-2}_{j,\varepsilon} R(x_{j,\varepsilon})
-\sum^n_{l\neq j}\Lambda^{(N-2)/2}_{j,\varepsilon}\Lambda^{(N-2)/2}_{l,\varepsilon}G(x_{j,\varepsilon},x_{l,\varepsilon})
=\frac{(2^*-q)B}{qA^{2}}\Lambda^{N-\frac{N-2}{2}q}_{j,\varepsilon}+o\Big((\Lambda_{\varepsilon})^{N-2}\Big).
\end{split}
\end{equation}
Hence
\begin{equation}\label{Prop blow up rate proof 7}
 \Big({M}(\vec{x}_\varepsilon)+o(1)\Big)
\vec{\mu}_{n,\varepsilon}^T=\frac{(2^*-q)B}{qA^2} \big(\Lambda^{\frac{(N+2)-(N-2)q}{2}}_{1,\varepsilon},\cdots,
\Lambda^{\frac{(N+2)-(N-2)q}{2}}_{k,\varepsilon}\big)^T,
\end{equation}
where $\vec{\mu}_{n,\varepsilon}=\big(\Lambda_{1,\varepsilon}^{(N-2)/2},\cdots,\Lambda_{n,\varepsilon}^{(N-2)/2}
\big)$ and
$\vec{x}_{\varepsilon}=(x_{1,\varepsilon},\cdots,x_{n,\varepsilon})$.
Recall that the first eigenvector of a symmetric matrix may be chosen with all its components strictly positive (see Appendix A in \cite{Bahri1995}). So if $\rho(\vec{x}_{\varepsilon})$ is the first eigenvalue of ${M}(\vec{x}_{\varepsilon})$, then there exists a first eigenvector $\Lambda(\vec{x}_{\varepsilon})$ of ${M}(\vec{x}_{\varepsilon})$ such that all its components are strictly positive.
   Then $\eqref{Prop blow up rate proof 7}$ implies that
  \begin{equation}\label{Prop blow up rate proof 8}
 \Lambda(\vec{x}_{\varepsilon})\Big({M}(\vec{x}_\varepsilon)+o(1)\Big)
\vec{\mu}_{n,\varepsilon}^T=\frac{(2^*-q)B}{qA^2} \Lambda(\vec{x}_{\varepsilon}) \big(\Lambda^{\frac{(N+2)-(N-2)q}{2}}_{1,\varepsilon},\cdots,
\Lambda^{\frac{(N+2)-(N-2)q}{2}}_{n,\varepsilon}\big)^T> 0.
\end{equation}
Note that
  \begin{equation}\label{Prop blow up rate proof 9}
 \rho(\vec{x}_{\varepsilon}) \langle\Lambda(\vec{x}_{\varepsilon}),
\vec{\mu}_{n,\varepsilon}^T\rangle= \langle\Lambda(\vec{x}_{\varepsilon}), {M}(\vec{x}_{\varepsilon})\cdot
\vec{\mu}_{n,\varepsilon}^T\rangle~\mbox{and}~\langle\Lambda(\vec{x}_{\varepsilon})
,\vec{\mu}_{n,\varepsilon}^T\rangle>0.
\end{equation}
which together with \eqref{Prop blow up rate proof 8}, we deduce that  $ \rho(\vec{a})\geq 0$ and ${M}(\vec{a})$ is a non-negative matrix.
Moreover, if ${M}(\vec{a})$ is a positive matrix, it implies that  $\Lambda_{j,\varepsilon}$ is bounded for $j=1,\cdots,n$ and then \eqref{Prop blow up rate 1} follows.
Moreover, from \eqref{Prop blow up rate proof 3}, \eqref{Prop blow up rate proof 4} and \eqref{Prop blow up rate 1}, we can obtain that
\begin{equation}\label{Prop blow up rate proof 10}
     \Lambda^{N-2}_{j,\varepsilon} R(x_{j,\varepsilon})
-\sum^n_{l\neq j}\Lambda^{(N-2)/2}_{j,\varepsilon}\Lambda^{(N-2)/2}_{l,\varepsilon}G(x_{j,\varepsilon},x_{l,\varepsilon})-\frac{(2^*-q)B}{qA^{2}}\Lambda^{N-\frac{N-2}{2}q}_{j,\varepsilon}=O(\mathrm{F_{N,q}}(\lambda_{\varepsilon})).
\end{equation}
where $\mathrm{F_{N,q}}(\lambda_{\varepsilon})$ defined in \eqref{definition of F-N-Q-lambda}. Letting $\varepsilon \rightarrow 0$ in \eqref{Prop blow up rate proof 10}, then by Lemma \ref{Lem derivate for reduce function}, we  can obtain \eqref{Prop blow up rate 2}. Finally, if $q\geq 2^*-1$, it's easy to check that $\partial^{2}_{\vec{\lambda}}\Psi_{n}(\vec{a},\vec{\lambda})$ is strictly positive definite for any $\vec{\lambda}$, where $\Psi_{n}(\vec{a},\vec{\lambda})$ is  defined by
\begin{equation}
     \Psi_{n}(\vec{x},\vec{\lambda}):=\frac{A^{2}}{2}\langle \vec{\lambda}, {M}(\vec{x})\vec{\lambda}\rangle-\frac{B}{q}\sum_{j=1}^{n} \lambda_{j}^{2^*-q},\quad (\vec{x},\vec{\lambda})\in \Omega_{*}^{n}\times(0,\infty)^{n} ,
 \end{equation}
thus $\Psi_{n}(\vec{a},\vec{\lambda})$ is convex in the variable $\vec{\lambda}$ in $(0,\infty)^{n}$, hence it has a unique critical point. Then by changing of variable, we deduce that $\Phi_{n}(\vec{a},\cdot)$ has a unique critical point in $(0,\infty)^{n}$. This completes the proof.
\end{proof}

\begin{Prop}\label{prop blow up rate 2}
Let $u_{\varepsilon}$ be a family of solutions of \eqref{p-varepsion} with decomposition \eqref{decomposition of u} and  ${M}(\vec{a})$ be a positive matrix. Then
\begin{equation}\label{prop blow up rate 2-1}
\nabla_x \Phi_n(\vec{a},\vec{\lambda})=0,~\mbox{with}~
 \vec{a}=(a_1,\cdots,a_{n})~ \mbox{and}~\vec{\lambda}=(\lambda_{1},\cdots,\lambda_{n}),
 \end{equation}
 where $\lambda_j:=\displaystyle\lim_{\varepsilon\rightarrow 0}\big(\varepsilon ^{\frac{2}{(N-2)q-4}}\lambda_{j,\varepsilon}\big)^{-1}$ for $j=1,\cdots,n$.
Moreover if $(\vec{a},\vec{\lambda})$ is a nondegenerate critical point of $\Phi_{n}$, then
for $j=1,\cdots,n$, 
\begin{equation}\label{prop blow up rate 2-2}
\big|x_{j,\varepsilon}-a_j\big|=O(\mathrm{F_{N,q}}(\lambda_{\varepsilon})),
 \end{equation}
 and
 \begin{equation}\label{prop blow up rate 2-3}
\big|\lambda_j- \big(\varepsilon ^{\frac{2}{(N-2)q-4}}\lambda_{j,\varepsilon}\big)^{-1}\big|=O(\mathrm{F_{N,q}}(\lambda_{\varepsilon})),
\end{equation}
where $\mathrm{F_{N,q}}(\lambda_{\varepsilon})$ defined in \eqref{definition of F-N-Q-lambda}.
\end{Prop}
\begin{proof}
We will use the local Poho\v{z}aev identity \eqref{pohozaev identity 2} for $u_{\varepsilon}$ with $\Omega=B(x_{j,\varepsilon},\theta)$ and $\theta>0$ small enough. From Proposition \ref{prop estimate for u} and Proposition \ref{Prop blow up rate}, the left-hand side of local Poho\v{z}aev identity \eqref{pohozaev identity 2} becomes
\begin{equation}\label{prop blow up rate 2 proof 1}
\begin{split}
\frac{ A^{2}}{2}\sum^n_{l=1}\sum^n_{m=1}\frac{Q\big(G(x_{m,\varepsilon},x),G(x_{l,\varepsilon},x)\big)}
{\lambda^{(N-2)/2}_{m,\varepsilon}\lambda^{(N-2)/2}_{l,\varepsilon}} +\begin{cases}
     O\left(\frac{1}{\lambda_{\varepsilon}^{2}}\right), &\text{~~if~~} N=3\text{~and~} q\in (4,6),\\
         O\left(\frac{1}{\lambda_{\varepsilon}^{2q-2}}\right), &\text{~~if~~} N=4 \text{~and~} q\in(2,3),\\
          O\left(\frac{\ln\lambda_{\varepsilon}}{\lambda_{\varepsilon}^{4}}\right), &\text{~~if~~} N=4 \text{~and~} q\in[3,4),\\ 
      O\left(\frac{1}{\lambda_{\varepsilon}^{3q-2}}\right), &\text{~~if~~} N=5\text{~and~} q\in (2,\frac{7}{3}),\\
     O\left(\frac{1}{\lambda_{\varepsilon}^{5}}\right), &\text{~~if~~} N=5\text{~and~} q\in [\frac{7}{3},\frac{10}{3}),\\
     O\left(\frac{1}{\lambda_{\varepsilon}^{N}}\right), &\text{~~if~~} N\geq 6\text{~and~} q\in (2,2^*),
 \end{cases}
\end{split}
\end{equation}
where $A$ is the constant in \eqref{definition of A and B} and $Q(\cdot,\cdot)$ is the quadratic form defined in \eqref{defin for quadratic Q}. Moreover, from Lemma \ref{lem for estimate of quadractic form},  \eqref{prop blow up rate 2 proof 1} reduces to
\begin{equation}\label{prop blow up rate 2 proof 2}
\begin{aligned}
   -\frac{A^{2}}{2\lambda_{j,\varepsilon}^{N-2}}\frac{\partial R(x_{j,\varepsilon})}{
\partial{x_i}}&+\displaystyle\sum^n_{l=1,l\neq j}
\frac{A^{2}}{\lambda^{(N-2)/2}_{j,\varepsilon}\lambda^{(N-2)/2}_{l,\varepsilon}}
\frac{\partial G(x_{j,\varepsilon},x_{l,\varepsilon})}{\partial x_i}\\
&+\begin{cases}
     O\left(\frac{1}{\lambda_{\varepsilon}^{2}}\right), &\text{~~if~~} N=3\text{~and~} q\in (4,6),\\
         O\left(\frac{1}{\lambda_{\varepsilon}^{2q-2}}\right), &\text{~~if~~} N=4 \text{~and~} q\in(2,3),\\
          O\left(\frac{\ln\lambda_{\varepsilon}}{\lambda_{\varepsilon}^{4}}\right), &\text{~~if~~} N=4 \text{~and~} q\in[3,4),\\ 
      O\left(\frac{1}{\lambda_{\varepsilon}^{3q-2}}\right), &\text{~~if~~} N=5\text{~and~} q\in (2,\frac{7}{3}),\\
     O\left(\frac{1}{\lambda_{\varepsilon}^{5}}\right), &\text{~~if~~} N=5\text{~and~} q\in [\frac{7}{3},\frac{10}{3}),\\
     O\left(\frac{1}{\lambda_{\varepsilon}^{N}}\right), &\text{~~if~~} N\geq 6\text{~and~} q\in (2,2^*).
 \end{cases}
\end{aligned}
\end{equation}
where $i=1,\cdots,N$. On the other hand, from Proposition \ref{prop estimate for u} and Proposition \ref{Prop blow up rate}, the right-hand side of local Poho\v{z}aev identity \eqref{pohozaev identity 2} can be estimated as
\begin{equation}\label{prop blow up rate 2 proof 3}
    O\Big(\frac{1}{\lambda^N_{\varepsilon}}+\frac{1}{\lambda_{\varepsilon}^{(N-2)q-2}}\Big).
\end{equation}
Then \eqref{prop blow up rate 2 proof 2} and \eqref{prop blow up rate 2 proof 3} imply that
\begin{equation}\label{prop blow up rate 2 proof 4}
\begin{aligned}
    \frac{1}{2\lambda_{j,\varepsilon}^{N-2}}\frac{\partial R(x_{j,\varepsilon})}{
\partial{x_i}}&-\displaystyle\sum^n_{l=1,l\neq j}
\frac{1}{\lambda^{(N-2)/2}_{j,\varepsilon}\lambda^{(N-2)/2}_{l,\varepsilon}}
\frac{\partial G(x_{j,\varepsilon},x_{l,\varepsilon})}{\partial x_i}\\
&=
 \begin{cases}
     O\left(\frac{1}{\lambda_{\varepsilon}^{2}}\right), &\text{~~if~~} N=3\text{~and~} q\in (4,6),\\
         O\left(\frac{1}{\lambda_{\varepsilon}^{2q-2}}\right), &\text{~~if~~} N=4 \text{~and~} q\in(2,3),\\
          O\left(\frac{\ln\lambda_{\varepsilon}}{\lambda_{\varepsilon}^{4}}\right), &\text{~~if~~} N=4 \text{~and~} q\in[3,4),\\ 
      O\left(\frac{1}{\lambda_{\varepsilon}^{3q-2}}\right), &\text{~~if~~} N=5\text{~and~} q\in (2,\frac{7}{3}),\\
     O\left(\frac{1}{\lambda_{\varepsilon}^{5}}\right), &\text{~~if~~} N=5\text{~and~} q\in [\frac{7}{3},\frac{10}{3}),\\
     O\left(\frac{1}{\lambda_{\varepsilon}^{N}}\right), &\text{~~if~~} N\geq 6\text{~and~} q\in (2,2^*).
 \end{cases} 
\end{aligned}
\end{equation}
Let $\Lambda_{j,\varepsilon}:=\Big(\varepsilon ^{\frac{2}{(N-2)q-4}}\lambda_{j,\varepsilon}\Big)^{-1}$,
 we can  rewrite \eqref{prop blow up rate 2 proof 4}  as follows
\begin{equation}\label{prop blow up rate 2 proof 5}
     \frac{\Lambda_{j,\varepsilon}^{N-2}}{2}\frac{\partial R(x_{j,\varepsilon})}{
\partial{x_i}}-\displaystyle\sum^n_{l=1,l\neq j}
\Lambda_{j,\varepsilon}^{(N-2)/2}\Lambda_{l,\varepsilon}^{(N-2)/2}
\frac{\partial G(x_{j,\varepsilon},x_{l,\varepsilon})}{\partial x_i}=O(\mathrm{F_{N,q}}(\lambda_{\varepsilon})).
\end{equation}
Then taking  $\varepsilon\rightarrow 0$ in \eqref{prop blow up rate 2 proof 5}, then from Lemma \ref{Lem derivate for reduce function}, we know that \eqref{prop blow up rate 2-1} hold.
Moreover by the assumption that $(\vec{a},\vec{\lambda})$ is a nondegenerate critical point of $\Phi_{n}$, we can obtain \eqref{prop blow up rate 2-2} and \eqref{prop blow up rate 2-3} from \eqref{prop blow up rate 2 proof 5} and \eqref{Prop blow up rate proof 10}.
\end{proof}

In the following, we consider the case $\rho(\vec{a})=0$. First, we introduce the vector $\vec{\mu}_{\varepsilon}\in \R^{n}$ with components
\begin{equation}\label{prop blow-up rate 3 proof-3}
    (\vec{\mu}_{\varepsilon})_{i}:=\mu_{i,\varepsilon}:=\left(\frac{\lambda_{1,\varepsilon}}{\lambda_{i,\varepsilon}}\right)^{\frac{N-2}{2}}\text{~~with}~~i=1,\cdots,n,
\end{equation}
then by Corollary \ref{cor of the decomposition}, we know that $\mu_{i,\varepsilon}\in(0,\infty)$. In addition, we also need the matrix 
$\tilde{M}^{l}(\vec{x})\in\R^{n\times n}=(\tilde{m}^{l}_{i,j})_{i,j=1}^{n}$ with $l=1,\cdots,N$ and entries
\begin{equation}\label{derive of m ij definition}
    \tilde{m}^{l}_{i,j}(\vec{x})=\begin{cases}
        \partial_{l}R(x_{i}),&\quad\text{~for~} i=j,\\
        -2\partial_{l}^{x}G(x_{i},x_{j}),&\quad\text{~for~} i\neq j.
    \end{cases}
\end{equation}
Recall the definition of matrix $M(\vec{x})$ is given in \eqref{m ij definition}. 
\begin{Prop}\label{prop blow-up rate 3}
 Let $u_{\varepsilon}$ be a family of solutions of \eqref{p-varepsion} with decomposition \eqref{decomposition of u}. If $\rho(\vec{a})=0$, then $\vec{\mu}_0 = \Lambda(\vec{a})$ and
\begin{equation}
\lim_{\varepsilon\to0}\varepsilon\lambda_{j,\varepsilon}^{\frac{N-2}{2}q-2}=0,\text{~~for~~}j=1,\cdots,n.
\end{equation}    
\end{Prop}
\begin{proof}
From Proposition \ref{Prop blow up rate} and Proposition \ref{prop blow up rate 2}, for any $j=1,\cdots,n$ and $l=1,\cdots,N$
\begin{equation}\label{prop blow-up rate 3 proof-1}
         \sum_{l=1}^{n}m_{j,l}(\vec{x}_{\varepsilon})\lambda_{l,\varepsilon}^{-\frac{N-2}{2}}=\frac{2^*-q}{q}\frac{B}{A^{2}}\varepsilon\lambda_{j,\varepsilon}^{-\frac{N-2}{2}(2^*-1-q)}+O(\lambda_{\varepsilon}^{-\frac{N-2}{2}}\mathrm{F_{N,q}}(\lambda_{\varepsilon})),
\end{equation}
and
\begin{equation}\label{prop blow-up rate 3 proof-2}
    \begin{aligned}
         \sum_{l=1}^{n}\tilde{m}_{j,l}^{l}(\vec{x}_{\varepsilon})\lambda_{l,\varepsilon}^{-\frac{N-2}{2}}=O(\lambda_{\varepsilon}^{-\frac{N-2}{2}}\mathrm{F_{N,q}}(\lambda_{\varepsilon})),
    \end{aligned}
\end{equation}
where $\mathrm{F_{N,q}}(\lambda_{\varepsilon})$ defined in \eqref{definition of F-N-Q-lambda}.
Next, we can rewrite \eqref{prop blow-up rate 3 proof-1} as
\begin{equation}\label{prop blow-up rate 3 proof-4}
    ({M}(\vec{x}_{\varepsilon})\cdot\vec{\mu}_{\varepsilon})_{j}=\frac{2^*-q}{q}\frac{B}{A^{2}}\varepsilon\lambda_{j,\varepsilon}^{-\frac{N-2}{2}(2^*-2-q)}\mu_{j,\varepsilon}+O(\mathrm{F_{N,q}}(\lambda_{\varepsilon})).
\end{equation}
From Appendix A in \cite{Bahri1995}, the first eigenvalue $\rho(\vec{x}_{\varepsilon})$ of ${M}(\vec{x}_{\varepsilon})$ is simple and the corresponding eigenvector $\Lambda(\vec{x}_{\varepsilon})$ can be choose normalized such that $(\Lambda(\vec{x}_{\varepsilon}))_{1}=1$ and all its components are strictly positive. Then taking the scalar product of \eqref{prop blow-up rate 3 proof-4} with $\Lambda(\vec{x}_{\varepsilon})$, we have
\begin{equation}\label{prop blow-up rate 3 proof-5}
\begin{aligned}
    &\rho(\vec{x}_{\varepsilon})\langle\Lambda(\vec{x}_{\varepsilon}),\vec{\mu}_{\varepsilon}\rangle=\langle\Lambda(\vec{x}_{\varepsilon}),{M}(\vec{x}_{\varepsilon})\cdot\vec{\mu}_{\varepsilon}\rangle\\
    &=\frac{2^*-q}{q}\frac{B}{A^{2}}\varepsilon\sum_{j=1}^{n}\lambda_{j,\varepsilon}^{-\frac{N-2}{2}(2^*-2-q)}\mu_{j,\varepsilon}(\Lambda(\vec{x}_{\varepsilon}))_{j}+O(\mathrm{F_{N,q}}(\lambda_{\varepsilon})).
\end{aligned}
\end{equation}
If $\rho(\vec{a})=0$, then \eqref{prop blow-up rate 3 proof-5} shows that $\lim_{\varepsilon\to0}\varepsilon\lambda_{j,\varepsilon}^{-\frac{N-2}{2}(2^*-2-q)}=0$ and that $\vec{\mu}_0$ is an eigenvector with eigenvalue 0. Since $(\vec{ \mu}_0)_1 = 1 = (\Lambda(\vec{a}))_1$ and $\rho(\vec{a}_0)$ is simple, thus $\vec{\mu}_0 = \Lambda(\vec{a})$, that is $\vec{\mu}_0$ is precisely the lowest eigenvector of ${M}(\vec{a})$, with eigenvalue $\rho(\vec{a}) =0$.
\end{proof}

For the following analysis, we decompose $\vec{\mu}_\varepsilon=\alpha_\varepsilon \Lambda(\vec{x}_{\varepsilon})+ \vec{\delta} (\vec{x}_\varepsilon)$, where $\alpha_\varepsilon \in \R$, $\Lambda(\vec{x}_{\varepsilon})$ is the eigenvector of ${M}(\vec{x}_\varepsilon)$ corresponding to the lowest eigenvalue $\rho(\vec{x}_{\varepsilon})$ and $\vec{\delta}(\vec{x}_\varepsilon) \bot \Lambda(\vec{x}_{\varepsilon})$. Notice that $\alpha_\varepsilon \to 1$ as a consequence of $\vec{\mu}_\varepsilon \to \Lambda(\vec{a})$.

Before going on, we need the following Proposition and the proof is similar as  in \cite[Proposition 4.1]{Knig2022FineMA}. We give it for completeness.

\begin{Prop}
\label{proposition bound on rho}
As  $\varepsilon \to 0$, 
\begin{equation}\label{delta bound}
\begin{aligned}
    |\vec{\delta} (\vec{x}_\varepsilon)|&=O( \varepsilon \lambda_\varepsilon^{\frac{N-2}{2}q-2}+ |\rho(\vec{x}_\varepsilon)|+\mathrm{F_{N,q}}(\lambda_{\varepsilon})).
\end{aligned}
\end{equation} 
Moreover, if $\rho(\vec{a})=0$,  then as $\varepsilon \to 0$,  
\begin{equation}
\label{rho bound}
\begin{aligned}
    \rho(\vec{x}_\varepsilon) &=o(\mathrm{F_{N,q}}(\lambda_{\varepsilon})).
\end{aligned}
\end{equation}
\end{Prop}
\begin{proof}
First of all, arguing as in \cite[Lemma 4.5]{Knig2022MultibubbleBA}, we have that 
\begin{align*}
\partial_l^{x_i} \rho(\vec{x}_\varepsilon) &=  \partial_l^{x_i} \langle \vec{\mu}_\varepsilon, {M}(\vec{x}) \cdot \vec{\mu}_\varepsilon \rangle |_{\vec{x} = \vec{x}_\varepsilon}  + O(|\rho(\vec{x}_\varepsilon)|  + | \vec{\delta}(\vec{x}_\varepsilon)|) \\
&= \mu_{i,\varepsilon} (\tilde{M}^l(\vec{x}_\varepsilon) \cdot \vec{\mu}_\varepsilon)_i +  O(|\rho(\vec{x}_\varepsilon)|  + | \vec{\delta}(\vec{x}_\varepsilon)|). 
\end{align*}
Then by \eqref{prop blow-up rate 3 proof-2} , we have that,
\begin{equation}\label{nabla rho  a priori bound}
\begin{aligned}
|\nabla \rho(\vec{x}_\varepsilon)| &=
  O \left(|\rho(\vec{x}_\varepsilon)|  + | \vec{\delta}(\vec{x}_\varepsilon)|+\mathrm{F_{N,q}}(\lambda_{\varepsilon})\right).
\end{aligned} 
\end{equation} 
On the other hand, since $M(\vec x_\varepsilon) \cdot \vec \mu_\varepsilon = \alpha_\varepsilon \rho(\vec x_\varepsilon) \Lambda (\vec x_\varepsilon) + {M}(\vec x_\varepsilon) \cdot \vec \delta (\vec x_\varepsilon)$, then \eqref{prop blow-up rate 3 proof-4} implies 
\begin{equation}
    \begin{aligned}
        {M}(\vec{x}_\varepsilon) \cdot \vec \delta (\vec{x}_\varepsilon)&=
 O( \varepsilon \lambda_\varepsilon^{\frac{N-2}{2}q-2}+ |\rho(\vec x_\varepsilon)|+\mathrm{F_{N,q}}(\lambda_{\varepsilon})).
    \end{aligned}
\end{equation}
Since $\rho(\vec x_\varepsilon)$ is simple, $M(\vec x_\varepsilon)$ is uniformly coercive on the subspace orthogonal to $\Lambda(\vec x_\varepsilon)$, which contains $\vec \delta(\vec x_\varepsilon)$. Hence \eqref{delta bound} follows. Now, using \eqref{delta bound}, we can simplify \eqref{nabla rho  a priori bound}  to 
\begin{equation}\label{nabla rho bound without delta}
\begin{aligned}
|\nabla \rho(\vec{x}_\varepsilon)|&=O( \varepsilon \lambda_\varepsilon^{\frac{N-2}{2}q-2}+ |\rho(\vec{x}_\varepsilon)|+\mathrm{F_{N,q}}(\lambda_{\varepsilon}))
\end{aligned}
\end{equation}
Moreover, it's easy to see that $\rho(\vec x)$ is an analytic function of $\vec x$, thus arguing as in \cite[proof of Theorem 2.1]{Knig2022MultibubbleBA}, we know that there exists a constant  $\sigma > 1$ such that 
\begin{equation}
\label{rho to nabla rho estimate}
\rho(\vec x_\varepsilon) \lesssim |\nabla \rho(\vec x_\varepsilon)|^\sigma.
\end{equation}
Note that $\varepsilon \lambda_\varepsilon^{\frac{N-2}{2}q-2} = o(1)$ and $\rho(\vec x_\varepsilon) = o(1)$, then by \eqref{nabla rho bound without delta}, \eqref{rho to nabla rho estimate} and \eqref{prop blow-up rate 3 proof-5}, we can obtain \eqref{rho bound}. 
\end{proof}

\begin{proof}[Proof of Theorem \ref{thm blow up rate}]
From the propositions obtained above, we only need to consider the case that $\rho(\vec{a})=0$. Taking the scalar product of identity \eqref{prop blow-up rate 3 proof-4} with $\mu_{j,\varepsilon}$ and using the properties of $\Lambda(\vec x_\varepsilon)$ and $\vec \delta_\varepsilon$, we obtain 
\begin{equation}
    \begin{aligned}
    \rho(\vec x_\varepsilon) |\Lambda (\vec x_\varepsilon)|^2 \alpha_\varepsilon^2  + \langle \vec \delta(\vec x_\varepsilon), {M}(\vec x_\varepsilon) \cdot \vec \delta(\vec x_\varepsilon) \rangle 
        =\frac{2^*-q}{q}\frac{B}{A^{2}}\varepsilon\lambda_{j,\varepsilon}^{\frac{N-2}{2}q-2}\mu_{j,\varepsilon}^{2}+O(\mathrm{F_{N,q}}(\lambda_{\varepsilon})).
    \end{aligned}
\end{equation}
From the results given by Proposition \ref{proposition bound on rho}, it's easy to see that the terms in $\rho(\vec x_\varepsilon)$ and in $\vec \delta(\vec x_\varepsilon)$ on the left-hand side are negligible. Since $\mu_{j,\varepsilon} \sim 1$, the above identity then implies 
\begin{equation}
    \varepsilon \lambda_{j,\varepsilon}^{\frac{N-2}{2}q-2}=O(\mathrm{F_{N,q}}(\lambda_{\varepsilon})),
\end{equation}
and with this information, we can return to \eqref{nabla rho bound without delta} to deduce the bound that 
\begin{equation}
    |\nabla \rho(\vec x_\varepsilon)|=O(\mathrm{F_{N,q}}(\lambda_{\varepsilon})).
\end{equation}
This completes the proof of Theorem \ref{thm blow up rate}. 
\end{proof}

\section{Local uniqueness}\label{section local uniqueness}

\setcounter{equation}{0}
Throughout this section, we always assume that $N\ge 5$, $q\in(2,2^*)$, $q\geq 2^*-1$, $\Omega$ is a smooth bounded domain and satisfies Assumption $\textrm{A}$.

\subsection{Difference between two different solutions}\label{Regularization and blow-up analysis}

Let $u^{(1)}_{\varepsilon}(x)$, $u^{(2)}_{\varepsilon}(x)$ be two different solutions of $\eqref{p-varepsion}$ concentrating at the same point $\vec{a}\in\Omega^{n}_{*}$ with $M(\vec{a})$ is a positive matrix. Then from Theorem \ref{thm decomposition} and Theorem \ref{thm blow up rate}, the solution $u^{(l)}_{\varepsilon}(x)$ with $l=1,2$ have the following decomposition
\begin{equation}\label{eq decomposition for u l l=1,2}
u^{(l)}_{\varepsilon}=\sum^n_{j=1} PU_{x^{(l)}_{j,\varepsilon}, \lambda^{(l)}_{j,\varepsilon}}+w^{(l)}_{\varepsilon},
\end{equation}
satisfying, for $j=1,\cdots,n$, $l=1,2$,
\begin{equation}
 x^{(l)}_{j,\varepsilon}\rightarrow a_j, ~\Big(\varepsilon ^{\frac{2}{(N-2)q-4}} \lambda^{(l)}_{j,\varepsilon}\Big)^{-1}\rightarrow \lambda_j,~ \|w^{(l)}_{\varepsilon}\|_{H^{1}_{0}(\Omega)}=o(1),~~w^{(l)}_\varepsilon\in \bigcap^n_{j=1}E_{x^{(l)}_{j,\varepsilon},\lambda^{(l)}_{j,\varepsilon}}.
\end{equation}
Moreover, we assume that concentration point $(\vec{a},\vec{\lambda})$ is a nondegenerate critical point of $\Phi_{n}$.

In the following, we need to estimate the difference between the above two different solutions. Before going on, we set $\bar{\lambda}_\varepsilon:=
\min\limits_{j=1,\cdots,n;l=1,2}
\{\lambda^{(l)}_{j,\varepsilon}\}$ and use $\|\cdot\|$ to denote the norm $\|\cdot\|_{H^{1}_{0}(\Omega)}$.

\begin{Lem}\label{lem defference of w}
It holds
\begin{equation}\label{lem defference of w-1}
\|w^{(1)}_{\varepsilon}-w^{(2)}_{\varepsilon}\|=\begin{cases}
      O\left(\frac{1}{\bar{\lambda}_{\varepsilon}^{4}}\right), &\text{~~if~~} N=5 \text{~and~} q\in [\frac{7}{3},\frac{10}{3}),\\
O\left(\frac{(\ln\bar{\lambda}_{\varepsilon})^{\frac{N+2}{2N}}}{\bar{\lambda}_{\varepsilon}^{\frac{N+4}{2}}}\right), &\text{~~if~~} N\geq 6\text{~and~} q\in (2,2^*),
 \end{cases}=o\left(\frac{1}{\bar{\lambda}_{\varepsilon}^{\frac{N+2}{2}}}\right).
\end{equation}
\end{Lem}

     
      



\begin{proof}
First we define $\overline{w}_{\varepsilon}:=w^{(1)}_{\varepsilon}-w^{(2)}_{\varepsilon}$, then
 \begin{equation}
Q_\varepsilon \overline{w}_{\varepsilon}
    =R_\varepsilon^{(1)}\big(w^{(1)}_{\varepsilon}\big)-R_\varepsilon^{(2)}\big(w^{(2)}_{\varepsilon}\big)+ L_\varepsilon^{(1)}-L_\varepsilon^{(2)}+l_{\varepsilon},
 \end{equation}
where
\begin{equation}
    Q_{\varepsilon}\overline{w}_{\varepsilon}:=-\Delta \overline{w}_{\varepsilon}-(2^*-1)(\sum_{j=1}^{n}PU_{x_{j,\varepsilon}^{(1)},\lambda_{j,\varepsilon}^{(1)}})^{2^*-2}\overline{w}_{\varepsilon},
\end{equation}
\begin{equation}
\begin{split}
&R^{(l)}_\varepsilon(w^{(l)}_\varepsilon)\\
&=
\Big(\sum^n_{j=1}PU_{x^{(l)}_{j,\varepsilon},\lambda^{(l)}_{j,\varepsilon}}
+w^{(l)}_\varepsilon\Big)^{2^*-1}
-
\Big(\sum^n_{j=1}PU_{x^{(l)}_{j,\varepsilon},\lambda^{(l)}_{j,\varepsilon}}\Big)^{2^*-1}
-(2^*-1)
\Big(\sum^n_{j=1}PU_{x^{(l)}_{j,\varepsilon},\lambda^{(l)}_{j,\varepsilon}}\Big)^{2^*-2}
w^{(l)}_\varepsilon\\
&\quad +\varepsilon\left[\Big(\sum^n_{j=1}PU_{x^{(l)}_{j,\varepsilon},\lambda^{(l)}_{j,\varepsilon}}
+w^{(l)}_\varepsilon\Big)^{q-1}
-
\Big(\sum^n_{j=1}PU_{x^{(l)}_{j,\varepsilon},\lambda^{(l)}_{j,\varepsilon}}\Big)^{q-1}\right]\\
&:=R_{\varepsilon,2^*}^{(l)}(w_{\varepsilon}^{(l)})+R_{\varepsilon,q}^{(l)}(w_{\varepsilon}^{(l)}),
\end{split}\end{equation}
 \begin{equation}
\begin{aligned}
L_{\varepsilon}^{(l)}=\Big(\sum^n_{j=1}PU_{x^{(l)}_{j,\varepsilon},\lambda^{(l)}_{j,\varepsilon}}\Big)^{2^*-1}-\sum^n_{j=1}U_{x^{(l)}_{j,\varepsilon},\lambda^{(1)}_{j,\varepsilon}}^{2^*-1}+\varepsilon\Big(\sum^n_{j=1}PU_{x^{(1)}_{j,\varepsilon},\lambda^{(1)}_{j,\varepsilon}}\Big)^{q-1},
\end{aligned}
\end{equation}
and
\begin{equation}
    l_{\varepsilon}=(2^*-1)\left[\Big(\sum^n_{j=1}PU_{x^{(1)}_{j,\varepsilon},\lambda^{(1)}_{j,\varepsilon}}\Big)^{2^*-2}-\Big(\sum^n_{j=1}PU_{x^{(2)}_{j,\varepsilon},\lambda^{(2)}_{j,\varepsilon}}\Big)^{2^*-2}\right]w_{\varepsilon}^{(2)}.
\end{equation}
Now we write $\overline{w}_{\varepsilon}=\overline{w}_{\varepsilon,1}+\overline{w}_{\varepsilon,2}$ with
$\overline{w}_{\varepsilon,1}\in \displaystyle\bigcap^n_{j=1}E_{x^{(1)}_{j,\varepsilon},\lambda^{(1)}_{j,\varepsilon}}$ and $\overline{w}_{\varepsilon,2}\bot \displaystyle\bigcap^n_{j=1}E_{x^{(1)}_{j,\varepsilon},\lambda^{(1)}_{j,\varepsilon}}$.  Then
\begin{equation}
\overline{w}_{\varepsilon,1}(x)=\overline{w}_{\varepsilon}(x)-\overline{w}_{\varepsilon,2}=\overline{w}_{\varepsilon}(x) -\sum_{j=1}^n\Big(\alpha_{\varepsilon,j,0}\frac{\partial PU_{x^{(1)}_{j,\varepsilon},\lambda^{(1)}_{j,\varepsilon}}}
 {\partial \lambda}+\sum_{i=1}^N\alpha_{\varepsilon,j,i}
 \frac{\partial PU_{x^{(1)}_{j,\varepsilon},\lambda^{(1)}_{j,\varepsilon}}}
 {\partial x_i}\Big),
 \end{equation}
for some constants $\alpha_{\varepsilon,j,i}$ with $j=1,\cdots,n$ and $i=1,\cdots,N$. 
Moreover, from Lemma \ref{estimate of U-lambda-a and psi-lambda-a 4}
\begin{equation}
\begin{split}
 \alpha_{\varepsilon,j,0}&\Big\|\frac{\partial PU_{x^{(1)}_{j,\varepsilon},\lambda^{(1)}_{j,\varepsilon}}}
 {\partial \lambda}\Big\|^2
    \\=& \Bigl\langle \overline{w}_{\varepsilon}(x),\frac{\partial PU_{x^{(1)}_{j,\varepsilon},\lambda^{(1)}_{j,\varepsilon}}}
 {\partial \lambda} \Bigr\rangle -\sum_{l=1,l\neq j}^n\alpha_{\varepsilon,l,0}\Big\langle \frac{\partial PU_{x^{(1)}_{j,\varepsilon},\lambda^{(1)}_{j,\varepsilon}}}
 {\partial \lambda},\frac{\partial PU_{x^{(1)}_{l,\varepsilon},\lambda^{(1)}_{l,\varepsilon}}}
 {\partial \lambda}\Big\rangle\\&-
 \sum_{l=1}^n\sum_{i=1}^N\alpha_{\varepsilon,l,i}\Big\langle \frac{\partial PU_{x^{(1)}_{j,\varepsilon},\lambda^{(1)}_{j,\varepsilon}}} {\partial \lambda},
 \frac{\partial PU_{x^{(1)}_{l,\varepsilon},\lambda^{(1)}_{l,\varepsilon}}}
 {\partial x_i}\Big\rangle  \\=&-\Bigl\langle w^{(2)}_{\varepsilon},
 \frac{\partial  PU_{x^{(1)}_{j,\varepsilon},\lambda^{(1)}_{j,\varepsilon}}}{\partial \lambda}-
 \frac{\partial  PU_{x^{(2)}_{j,\varepsilon},\lambda^{(2)}_{j,\varepsilon}}}{\partial \lambda}
    \Bigr\rangle+o\Big(\sum_{l=1,l\neq j}^n\frac{|\alpha_{\varepsilon,l,0}|}{\bar{\lambda}^2_\varepsilon}+ \sum_{l=1}^n\sum_{i=1}^N|\alpha_{\varepsilon,l,i}|\Big)\\=&
    O\Big( \big|x_{j,\varepsilon}^{(1)}-x^{(2)}_{j,\varepsilon}\big|+
\frac{\big|\lambda_{j,\varepsilon}^{(1)}-\lambda^{(2)}_{j,\varepsilon}\big|}{\bar{\lambda}^2_\varepsilon}
\Big)\|w^{(2)}_{\varepsilon}\|
    +o\Big(\sum_{l=1,l\neq j}^n\frac{|\alpha_{\varepsilon,l,0}|}{\bar{\lambda}^2_\varepsilon}+ \sum_{l=1}^n\sum_{i=1}^N|\alpha_{\varepsilon,l,i}|\Big),\end{split}
\end{equation}
which together with Lemma \ref{estimate of U-lambda-a and psi-lambda-a 4} again
\begin{equation}\label{lem defference of w-proof-11}
 \alpha_{\varepsilon,j,0}=O\Big( \bar{\lambda}^2_\varepsilon\big|x_{j,\varepsilon}^{(1)}-x^{(2)}_{j,\varepsilon}\big|+
\big|\lambda_{j,\varepsilon}^{(1)}-\lambda^{(2)}_{j,\varepsilon}\big|
\Big)\|w^{(2)}_{\varepsilon}\|+o\Big(\sum_{l=1,l\neq j}^n |\alpha_{\varepsilon,l,0}| + \sum_{l=1}^n\sum_{i=1}^N\bar{\lambda}^2_\varepsilon|\alpha_{\varepsilon,l,i}|\Big).
\end{equation}
Similarly, we have
\begin{equation}
\begin{split}
 \alpha_{\varepsilon,j,i}&\Big\|\frac{\partial PU_{x^{(1)}_{j,\varepsilon},\lambda^{(1)}_{j,\varepsilon}}}
 {\partial x_{i}}\Big\|^2
    \\=& \Bigl\langle \overline{w}_{\varepsilon}(x),\frac{\partial PU_{x^{(1)}_{j,\varepsilon},\lambda^{(1)}_{j,\varepsilon}}}
 {\partial x_{i}} \Bigr\rangle -\sum_{l=1}^n\alpha_{\varepsilon,l,0}\Big\langle \frac{\partial PU_{x^{(1)}_{j,\varepsilon},\lambda^{(1)}_{j,\varepsilon}}}
 {\partial x_{i}},\frac{\partial PU_{x^{(1)}_{l,\varepsilon},\lambda^{(1)}_{l,\varepsilon}}}
 {\partial \lambda}\Big\rangle\\&-
 \sum_{(l,m)\neq (j,i)}\alpha_{\varepsilon,l,m}\Big\langle \frac{\partial PU_{x^{(1)}_{j,\varepsilon},\lambda^{(1)}_{j,\varepsilon}}} {\partial x_{i}},
 \frac{\partial PU_{x^{(1)}_{l,\varepsilon},\lambda^{(1)}_{l,\varepsilon}}}
 {\partial x_m}\Big\rangle  \\=&-\Bigl\langle w^{(2)}_{\varepsilon},
 \frac{\partial  PU_{x^{(1)}_{j,\varepsilon},\lambda^{(1)}_{j,\varepsilon}}}{\partial x_{i}}-
 \frac{\partial  PU_{x^{(2)}_{j,\varepsilon},\lambda^{(2)}_{j,\varepsilon}}}{\partial x_{i}}
    \Bigr\rangle+o\Big(\sum_{l=1}^n|\alpha_{\varepsilon,l,0}|+ \sum_{(l,m)\neq(j,i)}{\bar{\lambda}^2_\varepsilon}|\alpha_{\varepsilon,l,m}|\Big)\\=&
    O\Big( \bar{\lambda}^2_\varepsilon\big|x_{j,\varepsilon}^{(1)}-x^{(2)}_{j,\varepsilon}\big|+
\big|\lambda_{j,\varepsilon}^{(1)}-\lambda^{(2)}_{j,\varepsilon}\big|
\Big)\|w^{(2)}_{\varepsilon}\|
    +o\Big(\sum_{l=1}^n|\alpha_{\varepsilon,l,0}|+ \sum_{(l,m)\neq(j,i)}{\bar{\lambda}^2_\varepsilon}|\alpha_{\varepsilon,l,m}|\Big),\end{split}
\end{equation}
and
\begin{equation}\label{lem defference of w-proof-13}
\begin{split}
 \bar{\lambda}^2_\varepsilon|\alpha_{\varepsilon,j,i}|=
    O\Big( \bar{\lambda}^2_\varepsilon\big|x_{j,\varepsilon}^{(1)}-x^{(2)}_{j,\varepsilon}\big|+
\big|\lambda_{j,\varepsilon}^{(1)}-\lambda^{(2)}_{j,\varepsilon}\big|
\Big)\|w^{(2)}_{\varepsilon}\|+o\Big(\sum_{l=1}^n|\alpha_{\varepsilon,l,0}|+ \sum_{(l,m)\neq(j,i)}{\bar{\lambda}^2_\varepsilon}|\alpha_{\varepsilon,l,m}|\Big).\end{split}
\end{equation}
Hence it follows from \eqref{lem defference of w-proof-11} and \eqref{lem defference of w-proof-13} that
\begin{equation}\label{lem defference of w-proof-14}
\sum_{j=1}^n|\alpha_{\varepsilon,j,0}| + \sum_{j=1}^n\sum_{i=1}^N \bar{\lambda}^2_\varepsilon|\alpha_{\varepsilon,j,i}|= O\Big( \bar{\lambda}^2_\varepsilon\sum_{j=1}^{n}\big|x_{j,\varepsilon}^{(1)}-x^{(2)}_{j,\varepsilon}\big|+\sum_{j=1}^{n}
\big|\lambda_{j,\varepsilon}^{(1)}-\lambda^{(2)}_{j,\varepsilon}\big|
\Big)\|w^{(2)}_{\varepsilon}\|.
\end{equation}
Next, by using the definition of $Q_{\varepsilon}$, $ \overline{w}_{\varepsilon,1}\in \displaystyle\bigcap^n_{j=1}E_{x^{(1)}_{j,\varepsilon},\lambda^{(1)}_{j,\varepsilon}}$, H\"older inequality and \eqref{lem defference of w-proof-14}, we obtain
\begin{equation}
\begin{split}
    &\left\langle Q_{\varepsilon}\,\sum_{j=1}^{n}\Big(\alpha_{\varepsilon,j,0}\frac{\partial PU_{x^{(1)}_{j,\varepsilon},\lambda^{(1)}_{j,\varepsilon}}}
 {\partial \lambda}+\sum_{i=1}^N\alpha_{\varepsilon,j,i}
 \frac{\partial PU_{x^{(1)}_{j,\varepsilon},\lambda^{(1)}_{j,\varepsilon}}}
 {\partial x_i}\Big),\overline{w}_{\varepsilon,1}\right\rangle \\
 &=\sum_{j=1}^{n}\alpha_{\varepsilon,j,0}\left\langle Q_{\varepsilon}\frac{\partial PU_{x^{(1)}_{j,\varepsilon},\lambda^{(1)}_{j,\varepsilon}}}
 {\partial \lambda},\overline{w}_{\varepsilon,1}\right\rangle +\sum_{j=1}^{n}\sum_{i=1}^{N}\alpha_{\varepsilon,j,i}\left\langle Q_{\varepsilon}\frac{\partial PU_{x^{(1)}_{j,\varepsilon},\lambda^{(1)}_{j,\varepsilon}}}
 {\partial x_{i}},\overline{w}_{\varepsilon,1}\right\rangle \\
 &=-(2^*-1)\sum_{j=1}^{n}\alpha_{\varepsilon,j,0}\int_{\Omega}(\sum_{k=1}^{n}PU_{x_{k,\varepsilon}^{(1)},\lambda_{k,\varepsilon}^{(1)}})^{2^*-2}\frac{\partial PU_{x^{(1)}_{j,\varepsilon},\lambda^{(1)}_{j,\varepsilon}}}
 {\partial \lambda}\overline{w}_{\varepsilon,1}\\
 &\quad -(2^*-1)\sum_{j=1}^{n}\sum_{i=1}^{N}\alpha_{\varepsilon,j,i}\int_{\Omega}(\sum_{k=1}^{n}PU_{x_{k,\varepsilon}^{(1)},\lambda_{k,\varepsilon}^{(1)}})^{2^*-2}\frac{\partial PU_{x^{(1)}_{j,\varepsilon},\lambda^{(1)}_{j,\varepsilon}}}
 {\partial x_{i}}\overline{w}_{\varepsilon,1}\\
 &=\frac{1}{\bar{\lambda}_{\varepsilon}}\|\overline{w}_{\varepsilon,1}\|O\left(\sum_{j=1}^{n}|\alpha_{\varepsilon,i,0}|+\sum_{j=1}^{n}\sum_{i=1}^{N}\bar{\lambda}_{\varepsilon}^{2}|\alpha_{\varepsilon,j,i}|\right)\\
 &=O\Big( \bar{\lambda}_\varepsilon\sum_{j=1}^{n}\big|x_{j,\varepsilon}^{(1)}-x^{(2)}_{j,\varepsilon}\big|+\sum_{j=1}^{n}\frac{\big|\lambda_{j,\varepsilon}^{(1)}-\lambda^{(2)}_{j,\varepsilon}\big|}{\bar{\lambda}_\varepsilon}\Big)\|\overline{w}_{\varepsilon,1}\|\|{w}_{\varepsilon}^{(2)}\|.
\end{split} 
\end{equation}
Hence, from Lemma \ref{q var}, we have
\begin{equation}\label{2-15-8}
\begin{split}
\Big \langle Q_{\varepsilon}\overline{w}_{\varepsilon}, \overline{w}_{\varepsilon,1}\Big \rangle = &
    \Big \langle Q_{\varepsilon} \overline{w}_{\varepsilon,1},\overline{w}_{\varepsilon,1} \Big \rangle+ \left\langle Q_{\varepsilon}\,\sum_{j=1}^{n}\Big(\alpha_{\varepsilon,j,0}\frac{\partial PU_{x^{(1)}_{j,\varepsilon},\lambda^{(1)}_{j,\varepsilon}}}
 {\partial \lambda}+\sum_{i=1}^N\alpha_{\varepsilon,j,i}
 \frac{\partial PU_{x^{(1)}_{j,\varepsilon},\lambda^{(1)}_{j,\varepsilon}}}
 {\partial x_i}\Big),\overline{w}_{\varepsilon,1}\right\rangle\\
 &\geq \rho_{0}\|\overline{w}_{\varepsilon,1}\|^{2}-C\Big(\bar{\lambda}_\varepsilon\sum_{j=1}^{n}\big|x_{j,\varepsilon}^{(1)}-x^{(2)}_{j,\varepsilon}\big|+\sum_{j=1}^{n}\frac{\big|\lambda_{j,\varepsilon}^{(1)}-\lambda^{(2)}_{j,\varepsilon}\big|}{\bar{\lambda}_\varepsilon}\Big)\|\overline{w}_{\varepsilon,1}\|\|{w}_{\varepsilon}^{(2)}\|,
\end{split}
 \end{equation}
and thus
\begin{equation} \label{eq difference of w-17}
 \|\overline{w}_{\varepsilon,1}\|^{2}=O\Big(\bar{\lambda}_\varepsilon\sum_{j=1}^{n}\big|x_{j,\varepsilon}^{(1)}-x^{(2)}_{j,\varepsilon}\big|+\sum_{j=1}^{n}\frac{\big|\lambda_{j,\varepsilon}^{(1)}-\lambda^{(2)}_{j,\varepsilon}\big|}{\bar{\lambda}_\varepsilon}\Big)\|\overline{w}_{\varepsilon,1}\|\|{w}_{\varepsilon}^{(2)}\|+O\Big(\Big \langle Q_{\varepsilon}\overline{w}_{\varepsilon}, \overline{w}_{\varepsilon,1}\Big \rangle \Big).   
\end{equation}
In the following, we need to estimate the second term on the right-hand side of \eqref{eq difference of w-17}. Notice that 
\begin{equation}
    U_{x_{j,\varepsilon}^{(1)},\lambda_{j,\varepsilon}^{(1)}}=\Big(1+O\Big({\bar{\lambda}_\varepsilon}{\big|x_{j,\varepsilon}^{(1)}-x^{(2)}_{j,\varepsilon}\big|}
 +\frac{\big|\lambda_{j,\varepsilon}^{(1)}-\lambda^{(2)}_{j,\varepsilon}\big|}{\bar{\lambda}_\varepsilon}\big)\Big)U_{x_{j,\varepsilon}^{(2)},\lambda_{j,\varepsilon}^{(2)}},
\end{equation}
thus by Lemma \ref{Elementary estimate} and H\"older inequality we have
 \begin{equation}\label{eq difference of w-18}
 \begin{split}
    \Big \langle l_{\varepsilon}, \overline{w}_{\varepsilon,1}\Big \rangle=O\Big(
{\bar{\lambda}_\varepsilon}\sum^n_{j=1}{\big|x_{j,\varepsilon}^{(1)}-x^{(2)}_{j,\varepsilon}\big|}
 +\sum^n_{j=1}\frac{\big|\lambda_{j,\varepsilon}^{(1)}-\lambda^{(2)}_{j,\varepsilon}\big|}{\bar{\lambda}_\varepsilon}\Big)
\|w^{(2)}_{\varepsilon}\|\|\overline{w}_{\varepsilon,1}\|,
 \end{split}
 \end{equation}
and similar as the computation in Lemma \ref{L var}, we have
\begin{equation}\label{eq difference of w-19}
   \Big \langle L_{\varepsilon}^{(1)}-L_{\varepsilon}^{(2)}, \overline{w}_{\varepsilon,1}\Big \rangle=O\Big(
{\bar{\lambda}_\varepsilon}\sum^n_{j=1}{\big|x_{j,\varepsilon}^{(1)}-x^{(2)}_{j,\varepsilon}\big|}
 +\sum^n_{j=1}\frac{\big|\lambda_{j,\varepsilon}^{(1)}-\lambda^{(2)}_{j,\varepsilon}\big|}{\bar{\lambda}_\varepsilon}\Big)
\|w_{\varepsilon}^{(2)}\|\|\overline{w}_{\varepsilon,1}\|.
\end{equation}
On the other hand, from Lemma \ref{Elementary estimate} and H\"older inequality, we have
\begin{equation}\label{lem defference of w-proof-21}
    \begin{aligned}
        \|R_{\varepsilon,2^*}^{(l)}(w_{\varepsilon}^{(l)})\|_{H^{-1}}=\begin{cases}
            O(\|w_{\varepsilon}^{(l)}\|^{2}),&\textrm{~~if~~}N=3,4,5,\\
            O(\|w_{\varepsilon}^{(l)}\|^{2^*-1}),&\textrm{~~if~~}N\geq 6,\\
        \end{cases}
    \end{aligned}
\end{equation}
and
\begin{equation}\label{lem defference of w-proof-22}
\begin{aligned}
   \|R_{\varepsilon,q}^{(l)}(w_{\varepsilon}^{(l)})\|_{H^{-1}}=
             O\left(\textrm{F}^{1}_{N,q}(\bar{\lambda}_{\varepsilon})\|w_{\varepsilon}^{(l)}\|+\frac{1}{\bar{\lambda}_{\varepsilon}^{\frac{N-2}{2}q-2}}\|w_{\varepsilon}^{(l)}\|^{q-1}\right),      
\end{aligned}
\end{equation}
where
\begin{equation}
    \textrm{F}^{1}_{N,q}(\bar{\lambda}_{\varepsilon})=\begin{cases}
        \frac{1}{\bar{\lambda}_{\varepsilon}^{^{(N-2)q-N}}},&\textrm{~~if~~}q\in(2,\frac{2N-2}{N-2}),\\
        \frac{(\ln\bar{\lambda}_{\varepsilon})^{\frac{2}{N}}}{\bar{\lambda}_{\varepsilon}^{N-2}},&\textrm{~~if~~}q=\frac{2N-2}{N-2},\\
        \frac{1}{\bar{\lambda}_{\varepsilon}^{N-2}},&\textrm{~~if~~}q\in(\frac{2N-2}{N-2},2^*).\\
    \end{cases}
\end{equation}
Hence from \eqref{lem defference of w-proof-21} and \eqref{lem defference of w-proof-22}, we get
\begin{equation}\label{eq difference of w-23}
\begin{aligned}
    \Big \langle R_{\varepsilon}^{(1)}(w_{\varepsilon}^{(1))}-R_{\varepsilon}^{(2)}(w_{\varepsilon}^{(1)}), \overline{w}_{\varepsilon,1}\Big \rangle&=O(\sum_{l=1}^{2}\|R_{\varepsilon}^{(l)}(w_{\varepsilon}^{(l)})\|_{H^{-1}})\|\overline{w}_{\varepsilon,1}\|\\
    &=o\Big(\bar{\lambda}_\varepsilon\sum_{j=1}^{n}\big|x_{j,\varepsilon}^{(1)}-x^{(2)}_{j,\varepsilon}\big|+\sum_{j=1}^{n}\frac{\big|\lambda_{j,\varepsilon}^{(1)}-\lambda^{(2)}_{j,\varepsilon}\big|}{\bar{\lambda}_\varepsilon}\Big)\|{w}_{\varepsilon}^{(2)}\|\|\overline{w}_{\varepsilon,1}\|.
\end{aligned}
\end{equation}
Combining estimate \eqref{eq difference of w-17}-\eqref{eq difference of w-23} together, we obtain
 \begin{equation}\label{lem defference of w-proof-26}
 \begin{aligned}
     \|\overline{w}_{\varepsilon,1}\|&=O\Big(\bar{\lambda}_\varepsilon\sum_{j=1}^{n}\big|x_{j,\varepsilon}^{(1)}-x^{(2)}_{j,\varepsilon}\big|+\sum_{j=1}^{n}\frac{\big|\lambda_{j,\varepsilon}^{(1)}-\lambda^{(2)}_{j,\varepsilon}\big|}{\bar{\lambda}_\varepsilon}\Big)\|{w}_{\varepsilon}^{(2)}\|,\\
 \end{aligned}  
 \end{equation}
and then from \eqref{lem defference of w-proof-14}, \eqref{lem defference of w-proof-26} and \eqref{perturbation estimate of w-varepsilon}, it follows that
 \begin{equation}
 \begin{split}
\|\overline{w}_{\varepsilon}\|=&O\left( \|\overline{w}_{\varepsilon,1 }\|+\Big\|\sum^n_{j=1}\Big(\alpha_{\varepsilon,j,0}\frac{\partial PU_{x^{(1)}_{j,\varepsilon},\lambda^{(1)}_{j,\varepsilon}}}
 {\partial \lambda}+\sum_{i=1}^N\alpha_{\varepsilon,j,i}
 \frac{\partial PU_{x^{(1)}_{j,\varepsilon},\lambda^{(1)}_{j,\varepsilon}}}
  {\partial x_i}\Big)\Big\|\right)\\ 
     &=O\Big(\bar{\lambda}_\varepsilon\sum_{j=1}^{n}\big|x_{j,\varepsilon}^{(1)}-x^{(2)}_{j,\varepsilon}\big|+\sum_{j=1}^{n}\frac{\big|\lambda_{j,\varepsilon}^{(1)}-\lambda^{(2)}_{j,\varepsilon}\big|}{\bar{\lambda}_\varepsilon}\Big)\|{w}_{\varepsilon}^{(2)}\|,\\
  \end{split}
 \end{equation}
which together with Theorem \ref{thm decomposition} and Theorem \ref{thm blow up rate}, it follows that \eqref{lem defference of w-1} hold.
\end{proof}

\begin{Lem}\label{lem difference of concentration point}
It holds
\begin{equation}\label{lem difference of concentration point-1}
\big|x^{(1)}_{j,\varepsilon}-x^{(2)}_{j,\varepsilon}\big|=o\Big(\frac{1}{\bar \lambda^2_{\varepsilon}}\Big)~\mbox{and}~\big|\lambda^{(1)}_{i,\varepsilon}-
\lambda^{(2)}_{j,\varepsilon}\big|=o\Big(\frac{1}{\bar\lambda^2_{\varepsilon}}\Big),
\end{equation}
for any $j=1,\cdots,n$.
\end{Lem}
\begin{proof}
We define the following quadratic form
\begin{equation}
Q_1(u,v)=-\int_{\partial B_\theta(x^{(1)}_{j,\varepsilon})}\frac{\partial v}{\partial \nu^{(1)}}\frac{\partial u}{\partial x_i}-
\int_{\partial B_\theta(x^{(1)}_{j,\varepsilon})}\frac{\partial u}{\partial \nu^{(1)}}\frac{\partial v}{\partial x_i}
+\int_{\partial B_\theta(x^{(1)}_{j,\varepsilon})}\big\langle \nabla u,\nabla v \big\rangle \nu^{(1)}_i,
\end{equation}
where $\nu^{(1)}(x)=\big(\nu^{(1)}_{1}(x),\cdots,\nu^{(1)}_N(x)\big)$ is the outward unit normal of $\partial B_{\theta}(x^{(1)}_{j,\varepsilon})$. Note that for any $k=1,2$ 
\begin{equation}
u^{(k)}_{\varepsilon}=\sum^n_{j=1} PU_{x^{(k)}_{j,\varepsilon}, \lambda^{(k)}_{j,\varepsilon}}+w^{(k)}_{\varepsilon}, 
\end{equation}
it follows that
\begin{equation}\label{lem difference of concentration point-proof-3}
    \begin{aligned}
     &Q_{1}(u^{(1)}_{\varepsilon},u^{(1)}_{\varepsilon})-Q_{1}(u^{(2)}_{\varepsilon},u^{(2)}_{\varepsilon})\\
        &=\sum^n_{l=1}\sum^n_{m=1}Q_{1}(PU_{x^{(1)}_{l,\varepsilon}, \lambda^{(1)}_{l,\varepsilon}},PU_{x^{(1)}_{m,\varepsilon}, \lambda^{(1)}_{m,\varepsilon}})-\sum^n_{l=1}\sum^n_{m=1}Q_{1}(PU_{x^{(2)}_{l,\varepsilon}, \lambda^{(2)}_{l,\varepsilon}},PU_{x^{(2)}_{m,\varepsilon}, \lambda^{(2)}_{m,\varepsilon}})\\
        &\quad+2\sum^n_{l=1}Q_{1}(PU_{x^{(1)}_{l,\varepsilon}, \lambda^{(1)}_{l,\varepsilon}},w_{\varepsilon}^{(1)})-2\sum^n_{l=1}Q_{1}(PU_{x^{(2)}_{l,\varepsilon}, \lambda^{(2)}_{l,\varepsilon}},w_{\varepsilon}^{(2)})\\
        &\quad+Q_{1}(w_{\varepsilon}^{(1)},w_{\varepsilon}^{(1)})-Q_{1}(w_{\varepsilon}^{(2)},w_{\varepsilon}^{(2)}).
    \end{aligned}
\end{equation}
Next, we estimate the terms on the right-hand side of \eqref{lem difference of concentration point-proof-3}. By using the Green's representation formula with a similar calculations as Proposition \ref{prop estimate for u}, we have
\begin{equation}
 PU_{x^{(k)}_{j,\varepsilon},\lambda^{(k)}_{j,\varepsilon}}(y)
 =A\frac{G\big(x^{(k)}_{j,\varepsilon},y\big)}{(\lambda^{(k)}_{j,\varepsilon})^{\frac{N-2}{2}}}+\delta_{j,\varepsilon}^{(k)}(y),
\end{equation}
with 
\begin{equation}
    \delta_{j,\varepsilon}^{(k)}(y)=O\left(\frac{1}{\bar{\lambda}_{\varepsilon}^{\frac{N+2}{2}}}\right),~~\mbox{in}~ C^{1}(\Omega\setminus B_\theta(x^{(k)}_{j,\varepsilon})),
\end{equation}
for any $k=1,2$ and $j=1,\cdots,n$. Moreover from Theorem \ref{thm blow up rate}, we have
\begin{equation}
    \sum_{i=1}^{n}\big|x_{i,\varepsilon}^{(1)}-x^{(2)}_{i,\varepsilon}\big|+\sum_{i=1}^{n}\big|\lambda_{i,\varepsilon}^{(1)}-\lambda^{(2)}_{i,\varepsilon}\big|=O\left(\frac{1}{\bar{\lambda}_{\varepsilon}^{2}}\right),
\end{equation}
and
 \begin{equation}
    \delta_{j,\varepsilon}^{(1)}(y)=(1+O(\frac{1}{\bar{\lambda}_{\varepsilon}}))\delta_{j,\varepsilon}^{(2)}(y)~~\mbox{in}~ C^{1}(\Omega\setminus B_\theta(x^{(k)}_{j,\varepsilon})).
 \end{equation}
Hence, the first term on the right-hand side of \eqref{lem difference of concentration point-proof-3} can be estimate as follows
\begin{equation}\label{lem difference of concentration point-proof-7}
    \begin{aligned}
        &\sum^n_{l=1}\sum^n_{m=1}Q_{1}(PU_{x^{(1)}_{l,\varepsilon}, \lambda^{(1)}_{l,\varepsilon}},PU_{x^{(1)}_{m,\varepsilon}, \lambda^{(1)}_{m,\varepsilon}})-\sum^n_{l=1}\sum^n_{m=1}Q_{1}(PU_{x^{(2)}_{l,\varepsilon}, \lambda^{(2)}_{l,\varepsilon}},PU_{x^{(2)}_{m,\varepsilon}, \lambda^{(2)}_{m,\varepsilon}})\\
        &=A^2\sum^n_{l=1}\sum^n_{m=1}\frac{Q_1\big(G(x^{(1)}_{m,\varepsilon},x),G(x^{(1)}_{l,\varepsilon},x)\big)}{
 (\lambda^{(1)}_{m,\varepsilon}) ^{(N-2)/2}
 (\lambda^{(1)}_{l,\varepsilon})^{(N-2)/2}} - A^2\sum^n_{l=1}\sum^n_{m=1}\frac{Q_1\big(G(x^{(2)}_{m,\varepsilon},x),G(x^{(2)}_{l,\varepsilon},x)\big)}{
 (\lambda^{(2)}_{m,\varepsilon}) ^{(N-2)/2}
 (\lambda^{(2)}_{l,\varepsilon})^{(N-2)/2}}\\
 &\quad+2A\sum^n_{l=1}\sum^n_{m=1}\frac{Q_{1}(G(x_{l,\varepsilon}^{(1)},x),\delta_{m,\varepsilon}^{(1)}(x))}{(\lambda_{l,\varepsilon}^{(1)})^{\frac{N-2}{2}}}-2A\sum^n_{l=1}\sum^n_{m=1}\frac{Q_{1}(G(x_{l,\varepsilon}^{(2)},x),\delta_{m,\varepsilon}^{(2)}(x))}{(\lambda_{l,\varepsilon}^{(2)})^{\frac{N-2}{2}}}\\
 &\quad+\sum^n_{l=1}\sum^n_{m=1}Q_{1}(\delta_{l,\varepsilon}^{(1)}(x),\delta_{m,\varepsilon}^{(1)}(x))-\sum^n_{l=1}\sum^n_{m=1}Q_{1}(\delta_{l,\varepsilon}^{(2)}(x),\delta_{m,\varepsilon}^{(2)}(x))\\
 &=A^2\sum^n_{l=1}\sum^n_{m=1}\frac{Q_1\big(G(x^{(1)}_{m,\varepsilon},x),G(x^{(1)}_{l,\varepsilon},x)\big)}{
 (\lambda^{(1)}_{m,\varepsilon}) ^{(N-2)/2}
 (\lambda^{(1)}_{l,\varepsilon})^{(N-2)/2}} - A^2\sum^n_{l=1}\sum^n_{m=1}\frac{Q_1\big(G(x^{(2)}_{m,\varepsilon},x),G(x^{(2)}_{l,\varepsilon},x)\big)}{
 (\lambda^{(2)}_{m,\varepsilon}) ^{(N-2)/2}
 (\lambda^{(2)}_{l,\varepsilon})^{(N-2)/2}}+o(\frac{1}{\bar{\lambda}_{\varepsilon}^{N}}).
    \end{aligned}
\end{equation}
Moreover, from Lemma \ref{lem defference of w} and \eqref{perturbation estimate of w-varepsilon}, the second and the third term on the right-hand side of \eqref{lem difference of concentration point-proof-3} can be bounded by
\begin{equation}\label{lem difference of concentration point-proof-8}
\begin{aligned}
    &\sum^n_{l=1}Q_{1}(PU_{x^{(1)}_{l,\varepsilon}, \lambda^{(1)}_{l,\varepsilon}},w_{\varepsilon}^{(1)})-\sum^n_{l=1}Q_{1}(PU_{x^{(2)}_{l,\varepsilon}, \lambda^{(2)}_{l,\varepsilon}},w_{\varepsilon}^{(2)})\\
    &=\sum^n_{l=1}Q_{1}(PU_{x^{(1)}_{l,\varepsilon}, \lambda^{(1)}_{l,\varepsilon}}-PU_{x^{(2)}_{l,\varepsilon}, \lambda^{(2)}_{l,\varepsilon}},w_{\varepsilon}^{(1)})+\sum^n_{l=1}Q_{1}(PU_{x^{(2)}_{l,\varepsilon}, \lambda^{(2)}_{l,\varepsilon}},\overline{w}_{\varepsilon})\\
    &=O(\frac{1}{\bar{\lambda}_{\varepsilon}^{\frac{N}{2}}}\|w_{\varepsilon}^{(1)}\|+\frac{1}{\bar{\lambda}_{\varepsilon}^{\frac{N-2}{2}}}\|\overline{w}_{\varepsilon}\|)=o(\frac{1}{\bar{\lambda}_{\varepsilon}^{N}}),
\end{aligned}    
\end{equation}
and 
\begin{equation}\label{lem difference of concentration point-proof-9}
    Q_{1}(w_{\varepsilon}^{(1)},w_{\varepsilon}^{(1)})-Q_{1}(w_{\varepsilon}^{(2)},w_{\varepsilon}^{(2)})=O(\sum_{k=1}^{2}\|w_{\varepsilon}^{(k})\|^{2})=o(\frac{1}{\bar{\lambda}_{\varepsilon}^{N}}).
\end{equation}
Combining the estimate \eqref{lem difference of concentration point-proof-3}, \eqref{lem difference of concentration point-proof-7}, \eqref{lem difference of concentration point-proof-8} and \eqref{lem difference of concentration point-proof-9} together, we deduce that
\begin{equation}\label{lem difference of concentration point-proof-10}
\begin{aligned}
    &Q_{1}(u^{(1)}_{\varepsilon},u^{(1)}_{\varepsilon})-Q_{1}(u^{(2)}_{\varepsilon},u^{(2)}_{\varepsilon})\\
    &= A^2\sum^n_{l=1}\sum^n_{m=1}\frac{Q_1\big(G(x^{(1)}_{m,\varepsilon},x),G(x^{(1)}_{l,\varepsilon},x)\big)}{
 (\lambda^{(1)}_{m,\varepsilon}) ^{(N-2)/2}
 (\lambda^{(1)}_{l,\varepsilon})^{(N-2)/2}} - A^2\sum^n_{l=1}\sum^n_{m=1}\frac{Q_1\big(G(x^{(2)}_{m,\varepsilon},x),G(x^{(2)}_{l,\varepsilon},x)\big)}{
 (\lambda^{(2)}_{m,\varepsilon}) ^{(N-2)/2}
 (\lambda^{(2)}_{l,\varepsilon})^{(N-2)/2}}+o(\frac{1}{\bar{\lambda}_{\varepsilon}^{N}}).  
\end{aligned}
\end{equation}
On the other hand
\begin{equation}\label{lem difference of concentration point-proof-11}
\begin{aligned}
    \int_{\partial B_\theta\big(x^{(1)}_{j,\varepsilon}\big)} (u_{\varepsilon}^{(1)})^{2^*}\nu^{(1)}_i-
(u_{\varepsilon}^{(2)})^{2^*}\nu^{(1)}_i=O\left( \int_{\partial B_\theta\big(x^{(1)}_{j,\varepsilon}\big)} (u_{\varepsilon}^{(1)})^{2^*-1}|u_{\varepsilon}^{(1)}-u_{\varepsilon}^{(2)}|\right)=o(\frac{1}{\bar{\lambda}_{\varepsilon}^{N}}), 
\end{aligned}   
\end{equation}
and
\begin{equation}\label{lem difference of concentration point-proof-12}
\begin{aligned}
    \varepsilon \int_{\partial B_\theta\big(x^{(1)}_{j,\varepsilon}\big)} (u_{\varepsilon}^{(1)}\Big)^{q}\nu^{(1)}_i-&
\Big(u_{\varepsilon}^{2})^{q}\nu^{(1)}_i=O(\frac{1}{\bar{\lambda}_{\varepsilon}^{(N-2)q}})=o(\frac{1}{\bar{\lambda}_{\varepsilon}^{N}}).
\end{aligned}
\end{equation}
Then from local Poho\v{z}aev identity \eqref{pohozaev identity 2}, \eqref{lem difference of concentration point-proof-10}, \eqref{lem difference of concentration point-proof-11} and \eqref{lem difference of concentration point-proof-12}, we get
\begin{equation}
    A^2\sum^n_{l=1}\sum^n_{m=1}\frac{Q_1\big(G(x^{(1)}_{m,\varepsilon},x),G(x^{(1)}_{l,\varepsilon},x)\big)}{
 (\lambda^{(1)}_{m,\varepsilon}) ^{(N-2)/2}
 (\lambda^{(1)}_{l,\varepsilon})^{(N-2)/2}} - A^2\sum^n_{l=1}\sum^n_{m=1}\frac{Q_1\big(G(x^{(2)}_{m,\varepsilon},x),G(x^{(2)}_{l,\varepsilon},x)\big)}{
 (\lambda^{(2)}_{m,\varepsilon}) ^{(N-2)/2}
 (\lambda^{(2)}_{l,\varepsilon})^{(N-2)/2}}=o(\frac{1}{\bar{\lambda}_{\varepsilon}^{N}}). 
\end{equation}

Let $\Lambda^{(k)}_{j,\varepsilon}:=\Big(\varepsilon ^{\frac{2}{(N-2)q-4}}\lambda^{(k)}_{j,\varepsilon}\Big)^{-1}$ with $k=1,2$ and $j=1,\cdots,n$. Then from Lemma \ref{Lem derivate for reduce function} and Lemma \ref{lem for estimate of quadractic form}
\begin{equation}
\nabla_x\Phi_n(x,\lambda)\big|_{(x,\lambda)
=(x^{(1)}_\varepsilon,\Lambda^{(1)}_\varepsilon)}-
\nabla_x\Phi_n(x,\lambda)\big|_{(x,\lambda)
=(x^{(2)}_\varepsilon,\Lambda^{(2)}_\varepsilon)}=
o\Big(\frac{1}{\bar{\lambda}_\varepsilon^{2}}\Big),
\end{equation}
where $x^{(1)}_\varepsilon=\big(x^{(1)}_{1,\varepsilon},
\cdots,x^{(1)}_{k,\varepsilon}\big)$ and $\Lambda^{(1)}_\varepsilon=\big(\Lambda^{(1)}_{1,\varepsilon},
\cdots,\Lambda^{(1)}_{k,\varepsilon}\big)$.
Similarly, by using local Poho\v{z}aev identity \eqref{pohozaev identity 1}, we can conclude that
\begin{equation}
\nabla_\lambda\Phi_n(x,\lambda)\big|_{(x,\lambda)
=(x^{(1)}_\varepsilon,\Lambda^{(1)}_\varepsilon)}-
\nabla_\lambda\Phi_n(x,\lambda)\big|_{(x,\lambda)
=(x^{(2)}_\varepsilon,\Lambda^{(2)}_\varepsilon)}=
o\Big(\frac{1}{\bar{\lambda}_\varepsilon^{2}}\Big).
\end{equation}
Since the concentration point is a nondegenerate critical point of $\Phi_{n}$, thus \eqref{lem difference of concentration point-1} holds.
\end{proof}

\begin{Rem}
  When $N=3,q\in(4,6)$, $N=4,q\in(2,2^*)$ and $N=5,q\in(2,\frac{7}{3})$, from Theorem \ref{thm decomposition} and Theorem \ref{thm blow up rate}, we can observe that the error estimate and concentration speed is not enough to obtain the above two lemmas.  
\end{Rem}

\subsection{Blow-up analysis}
Now we define
\begin{equation}\label{definition of xi}
\xi_{\varepsilon}(x)=\frac{u_{\varepsilon}^{(1)}(x)-u_{\varepsilon}^{(2)}(x)}
{\|u_{\varepsilon}^{(1)}-u_{\varepsilon}^{(2)}\|_{L^{\infty}(\Omega)}},
\end{equation}
then $\xi_{\varepsilon}(x)$ satisfies $\|\xi_{\varepsilon}\|_{L^{\infty}(\Omega)}=1$ and
\begin{equation}\label{equation for xi}
- \Delta \xi_{\varepsilon}(x)=C_{2^*,\varepsilon}(x)\xi_{\varepsilon}(x)+\varepsilon C_{q,\varepsilon}(x)\xi_{\varepsilon}(x),
\end{equation}
where
\begin{equation}
C_{2^*,\varepsilon}(x)=\Big(2^*-1\Big)\int_{0}^1
\Big(tu_{\varepsilon}^{(1)}(x)+(1-t)u_{\varepsilon}^{(2)}(x)\Big)
^{2^*-2}dt,
\end{equation}
and
\begin{equation}
    C_{q,\varepsilon}(x)=\Big(q-1\Big)\int_{0}^1
\Big(tu_{\varepsilon}^{(1)}(x)+(1-t)u_{\varepsilon}^{(2)}(x)\Big)
^{q-2}dt. 
\end{equation}

\begin{Lem}\label{lem estimate for xi}
It holds
\begin{equation}\label{lem estimate for xi-1}
\int_{\Omega}{\xi}_\varepsilon(x)dx=O\Big(\frac{\ln \bar{\lambda}_\varepsilon}{\bar{\lambda}^{N-2}_\varepsilon}\Big)~\mbox{and}~ \xi_{\varepsilon}(x) =O\Big(\frac{\ln \bar{\lambda}_\varepsilon}{\bar{\lambda}^{N-2}_\varepsilon}\Big),~\mbox{in}~ \Omega\backslash\bigcup_{j=1}^n B_{d}(x_{j,\varepsilon}^{(1)}),
\end{equation}
where $d>0$ is any small fixed constant.
\end{Lem}
\begin{proof}
By using the Green's representation formula, \eqref{equation for xi}, Theorem \ref{thm multibble blowup} and Lemma \ref{useful estimate},  we have
\begin{equation}
\begin{split}
 {\xi}_\varepsilon(x)=& \int_{\Omega}G(y,x)
\big(C_{2^*,\varepsilon}(y)+\varepsilon C_{q,\varepsilon}(y)\big)\xi_{\varepsilon}(y)dy\\=&
O\Big(\sum^n_{j=1}\sum^2_{l=1}\int_{\Omega}\frac{1}{|x-y|^{N-2}} U^{2^*-2}_{x^{(l)}_{j,\varepsilon},\lambda^{(l)}_{j,\varepsilon}}
(y)dy\Big)+O\big(\varepsilon\sum^n_{j=1}\sum^2_{l=1}\int_{\Omega}\frac{1}{|x-y|^{N-2}} U^{q-2}_{x^{(l)}_{j,\varepsilon},\lambda^{(l)}_{j,\varepsilon}}
(y)dy\big)\\=&
O\Big(\sum^n_{j=1}\sum^2_{l=1} \frac{1}{\big(1+\lambda^{(l)}_{j,\varepsilon}|x-x^{(l)}_{j,\varepsilon}|\big)^{2}}\Big)+O(\varepsilon).
\end{split}
\end{equation}
Repeating the above process and notice that $\varepsilon\sim \frac{1}{\bar{\lambda}_{\varepsilon}^{\frac{N-2}{2}q-2}}$,
\begin{equation}
\begin{split}
{\xi}_\varepsilon(x) = &O\left(\int_{\Omega}\frac{1}{|x-y|^{N-2}}
\big(C_{2^*,\varepsilon}(y)+\varepsilon C_{q,\varepsilon}(y)\big)\Big(\sum^n_{j=1}\sum^2_{l=1}
\frac{1}{\big(1+\lambda^{(l)}_{j,\varepsilon}|y-x^{(l)}_{j,\varepsilon}|\big)^{2}}
+\varepsilon\Big)dy\right)\\=&
O\Big(\sum^n_{j=1}\sum^2_{l=1}\frac{1}{\big(1+\lambda^{(l)}_{j,\varepsilon}
|x-x^{(l)}_{j,\varepsilon}|\big)^{4}}\Big)+O\Big(\varepsilon\sum^n_{j=1}\sum^2_{l=1} \frac{1}{\big(1+\lambda^{(l)}_{j,\varepsilon}|x-x^{(l)}_{j,\varepsilon}|\big)^{2}}\Big)+O\Big(\frac{1}{\bar{\lambda}_{\varepsilon}^{N-2}}\Big).
\end{split}
\end{equation}
After repeating the above process for finite number of times, we can obtain that 
\begin{equation}\label{eq estimate of xi-varepsilon}
{\xi}_\varepsilon(x) =
O\Big(\sum^n_{j=1}\sum^2_{l=1}\frac{|\ln ({\lambda}_{j,\varepsilon}^{l}|x-x^{(l)}_{j,\varepsilon}|)|}{\big(1+\lambda^{(l)}_{j,\varepsilon}
|x-x^{(l)}_{j,\varepsilon}|\big)^{N-2}}\Big)+
O\big(\frac{1}{\bar{\lambda}_{\varepsilon}^{N-2}}\big),
\end{equation}
thus \eqref{lem estimate for xi-1} follows.
\end{proof}

\begin{Lem}\label{lem xi var j}
For any $j=1,\cdots,n$, let $\xi_{\varepsilon,j}(x)=\xi_{\varepsilon}(\frac{x}{\lambda^{(1)}_{j,\varepsilon}}+x_{j,\varepsilon}^{(1)})$, after passing a subsequence,  we have
\begin{equation}\label{lem xi var j-1}
\Big|\xi_{\varepsilon,j}(x)-\sum_{i=0}^N c_{j,i}\psi_{i}(x)\Big|=o\big(\frac{1}{\bar \lambda_{\varepsilon}}\big),~\mbox{uniformly in}~C^1\big(B_R(0)\big) ~\mbox{for any}~R>0,
\end{equation}
 where $c_{j,i}$, $i=0,1,\cdots,N$ are some constants  and $$\psi_{0}(x)=\frac{\partial U_{0,\lambda}(x)}{\partial\lambda}\big|_{\lambda=1},~~\psi_{i}(x)=\frac{\partial U_{a,1}(x)}{\partial a_i}\big|_{a=0},~i=1,\cdots,N.
$$
\end{Lem}
\begin{proof}
From \eqref{equation for xi}, we know that $\xi_{\varepsilon,j}$ satisfies 
\begin{equation}
\begin{split}
-\Delta\xi_{\varepsilon,j}(x)=
\frac{1}{(\lambda^{(1)}_{j,\varepsilon})^{2}}C_{2^*,\varepsilon}(\frac{x}{\lambda^{(1)}_{j,\varepsilon}}
+x_{j,\varepsilon}^{(1)})\xi_{\varepsilon,j}(x)+ \frac{\varepsilon }{(\lambda^{(1)}_{j,\varepsilon})^{2}}C_{q,\varepsilon}(\frac{x}{\lambda^{(1)}_{j,\varepsilon}}
+x_{j,\varepsilon}^{(1)})\xi_{\varepsilon,j}(x),
\end{split}
\end{equation}
in $\Omega_{\varepsilon,j}:=\{x\in\R^{N}:\frac{x}{\lambda^{(1)}_{j,\varepsilon}}
+x_{j,\varepsilon}^{(1)}\in\Omega\}\to\R^{N}$. Thus by Theorem \ref{thm multibble blowup} and \eqref{eq estimate of xi-varepsilon}, we have
\begin{equation}
    \begin{split}
        \int_{\Omega_{\varepsilon,j}}|\nabla\xi_{\varepsilon,j}|^{2}dx&=\frac{1}{(\lambda^{(1)}_{j,\varepsilon})^{2}}\int_{\Omega_{\varepsilon,j}}C_{2^*,\varepsilon}(\frac{x}{\lambda^{(1)}_{j,\varepsilon}}
+x_{j,\varepsilon}^{(1)})\xi_{\varepsilon,j}^{2}(x)+ \varepsilon C_{q,\varepsilon}(\frac{x}{\lambda^{(1)}_{j,\varepsilon}}
+x_{j,\varepsilon}^{(1)})\xi_{\varepsilon,j}^{2}(x)dx\\
&\lesssim\int_{\Omega_{\varepsilon,j}}\frac{1}{(1+|x|^{2})^{2}}\left(\frac{(\ln|x|)^{2}}{(1+|x|)^{2N-4}}+\frac{1}{\bar{\lambda}_{\varepsilon}^{2N-4}}\right)dx\\
&\quad+\frac{1}{\bar{\lambda}_{\varepsilon}^{N-2}}\int_{\Omega_{\varepsilon,j}}\frac{1}{(1+|x|^{2})^{\frac{(N-2)(q-2)}{2}}}\left(\frac{(\ln|x|)^{2}}{(1+|x|)^{2N-4}}+\frac{1}{\bar{\lambda}_{\varepsilon}^{2N-4}}\right)dx<\infty,
    \end{split}
\end{equation}
then up to a subsequence, $\xi_{\varepsilon,j}\weakto \xi_{j}$ in $D^{1,2}(\R^{N})$. On the other hand, since $\xi_{\varepsilon,j}(x)$ is bounded, by the elliptic regularity theory \cite{GilbargTrudinger}, after passing to a subsequence, we have
\begin{equation}\label{lem xi var j-proof-2}
  \xi_{\varepsilon,j}(x)\rightarrow\xi_{j}(x)\text{~in~}C^{1}_{loc}(\R^{N}).  
\end{equation}
Next, we need to estimate $C_{2^*,\varepsilon}(x)$ and $C_{q,\varepsilon}(x)$. First, by Lemma \ref{estimate of U-lambda-a and psi-lambda-a 1} and Lemma \ref{lem difference of concentration point}
\begin{equation}
\begin{split}
U&_{x_{j,\varepsilon}^{(1)},\lambda_{j,\varepsilon}^{(1)}}(x)
-
U_{x_{j,\varepsilon}^{(2)},\lambda_{j,\varepsilon}^{(2)}}(x)
 \\=&
O\Big(\big|x_{j,\varepsilon}^{(1)}-x_{j,\varepsilon}^{(2)}\big|\cdot \big(\nabla_y U_{y,\lambda_{j,\varepsilon}^{(1)}}(x)|_{y=x_{j,\varepsilon}^{(1)}}\big)+
\big|\lambda_{j,\varepsilon}^{(1)}-\lambda_{j,\varepsilon}^{(2)}\big|\cdot\big( \nabla_\lambda U_{x_{j,\varepsilon}^{(1)},\lambda}(x) |_{\lambda=\lambda_{j,\varepsilon}^{(1)}}\big)
\Big)\\=&
O\Big(\lambda_{j,\varepsilon}^{(1)}\big|x_{j,\varepsilon}^{(1)}-x_{j,\varepsilon}^{(2)}\big| +
(\lambda_{j,\varepsilon}^{(1)})^{-1}|\lambda_{j,\varepsilon}^{(1)}-\lambda_{j,\varepsilon}^{(2)}| \Big) U_{x_{j,\varepsilon}^{(1)},\lambda_{j,\varepsilon}^{(1)}}(x)
=o\big(\frac{1}{\bar\lambda_{\varepsilon}}\big)U_{x_{j,\varepsilon}^{(1)},\lambda_{j,\varepsilon}^{(1)}}(x).
\end{split}
\end{equation}
Thus from \eqref{eq decomposition for u l l=1,2} and Lemma \ref{estimate of U-lambda-a and psi-lambda-a 3}
\begin{equation}\label{lem xi var j-proof-4}
u_{\varepsilon}^{(1)}(x)-u_{\varepsilon}^{(2)}(x)=
o\big(\frac{1}{\bar\lambda_{\varepsilon}}\big)\Big(\sum_{j=1}^n U_{x_{j,\varepsilon}^{(1)},\lambda_{j,\varepsilon}^{(1)}}(x)\Big)+O\Big(\frac{1}{\bar\lambda_{\varepsilon}^{\frac{N-2}{2}}}\Big)
+O\Big(\sum^2_{l=1}|w_{\varepsilon}^{(l)}(x)|\Big).
\end{equation}
Then for a small fixed $d$ and $x\in B_{d}(x_{j,\varepsilon}^{(1)})$, by Lemma \ref{Elementary estimate} we find 
\begin{equation}\label{lem xi var j-proof-5}
\begin{split}
C_{2^*,\varepsilon}(x) =& \Big(2^*-1+
o\big(\frac{1}{\bar\lambda_{\varepsilon}}\big)\Big)U^{2^*-2}_{x_{j,\varepsilon}^{(1)},\lambda_{j,\varepsilon}^{(1)}}(x)
+O\Big(\frac{1}{\bar\lambda_{\varepsilon}^{2}}\Big)
+O\Big(\sum^2_{l=1}|w_{\varepsilon}^{(l)}(x)|
^{2^*-2}\Big),
\end{split}
\end{equation}
when $N\geq 6$ and we also need to add the error term 
\begin{equation}\label{lem xi var j-proof-7}
 O\Big(\frac{1}{\bar{\lambda}_{\varepsilon}}\Big)+O\Big(\bar{\lambda}_{\varepsilon}^{\frac{1}{2}}\sum^2_{l=1}|w_{\varepsilon}^{(l)}(x)|\Big),
\end{equation}
to the right-hand side of \eqref{lem xi var j-proof-5} for $N=5$. Now, for any given $\Phi(x)\in C_{c}^{\infty}(\R^N)$ with $\text{supp}~\Phi(x) \subset B_{\lambda^{(1)}_{j,\varepsilon}d}(0)
$, from Lemma \ref{lem estimate for xi}, \eqref{perturbation estimate of w-varepsilon}, \eqref{lem xi var j-proof-5}, \eqref{lem xi var j-proof-7} we have
\begin{equation}
\begin{aligned}
\frac{1}{(\lambda^{(1)}_{j,\varepsilon})^{2}}\int_{B_{\lambda^{(1)}_{j,\varepsilon}d}(0)}
&C_{2^*,\varepsilon}\big(\frac{x}{\lambda^{(1)}_{j,\varepsilon}}+x_{j,\varepsilon}^{(1)})\xi_{\varepsilon,j}(x)\Phi(x) dx\\=&\Big(2^*-1\Big)
\int_{\R^N}U^{2^*-2}_{0,1}(x)\xi_{\varepsilon,j}(x)\Phi(x)dx+
o\big(\frac{1}{\bar\lambda_{\varepsilon}}\big).
\end{aligned}
\end{equation}
On the other hand, from the fact that $\varepsilon\sim\frac{1}{\bar{\lambda}_{\varepsilon}^{\frac{N-2}{2}q-2}}$, we can deduce that
\begin{equation}
\begin{split}
\frac{\varepsilon}{(\lambda^{(1)}_{j,\varepsilon})^{2}}\int_{B_{\lambda^{(1)}_{j,\varepsilon}d}(0)}C_{q,\varepsilon}\big(\frac{x}{\lambda^{(1)}_{j,\varepsilon}}+x_{j,\varepsilon}^{(1)})
\xi_{\varepsilon,j}(x)\Phi(x)dx=o\big(\frac{1}{\bar\lambda_{\varepsilon}}\big).
\end{split}
\end{equation}
Hence
\begin{equation}\label{lem xi var j-proof-10}
\begin{split}
\int_{B_{\lambda^{(1)}_{j,\varepsilon}d}(0)}
\nabla\xi_{\varepsilon,j}(x)\cdot\nabla\Phi(x)dx=
\Big(2^*-1\Big)\int_{\R^N}U^{2^*-2}_{0,1}(x)\xi_{\varepsilon,j}(x)\Phi(x)dx+o\big(\frac{1}{\bar\lambda_{\varepsilon}}\big).
\end{split}
\end{equation}
Letting $\varepsilon\rightarrow 0$ in above equation, then from \eqref{lem xi var j-proof-2}, we find that $\xi_{j}(x)$ satisfies
\begin{equation}\label{lem xi var j-proof-11}
-\Delta\xi_{j}(x)=\Big(2^*-1\Big)U_{0,1}^{2^*-2}(x)\xi_{j}(x),~~\textrm{in~}\R^N,
\end{equation}
which together with Lemma \ref{Kernel of Emden-Fowler equation}, implies that $
\xi_{j}(x)=\displaystyle\sum_{i=0}^Nc_{j,i}\psi_i(x)$. Moreover combining \eqref{lem xi var j-proof-10} and \eqref{lem xi var j-proof-11}, we can obtain \eqref{lem xi var j-1}.
\end{proof}

\begin{Lem}\label{lem estimate for xi var}
It holds
\begin{equation}
\xi_{\varepsilon}(x)=\sum^n_{j=1}A_{\varepsilon,j}G(x^{(1)}_{j,\varepsilon},x)+
\sum^n_{j=1}\sum^N_{i=1}B_{\varepsilon,j,i}\partial_iG(x^{(1)}_{j,\varepsilon},x)+O\Big(\frac{\ln \bar{\lambda}_\varepsilon}{\bar{\lambda}_\varepsilon^N}\Big)~\mbox{in}~ C^1
\Big(\Omega\backslash\bigcup^n_{j=1}B_{2d}(x^{(1)}_{j,\varepsilon})\Big),
\end{equation}
where $d>0$ is any small fixed constant, $\partial_iG(y,x)=\frac{\partial G(y,x)}{\partial y_i}$,
\begin{equation}\label{definition of A var and B var}
A_{\varepsilon,j}=
\int_{B_d\big(x^{(1)}_{j,\varepsilon}\big)} C_{2^*,\varepsilon}(x)\xi_{\varepsilon}(x)dx~\mbox{and}~
B_{\varepsilon,j,i}= \int_{B_d\big(x^{(1)}_{j,\varepsilon}\big)}(x_i-x^{(1)}_{j,\varepsilon,i}) C_{2^*,\varepsilon}(x)\xi_{\varepsilon}(x)dx.
\end{equation}
Moreover, we have
\begin{equation}\label{eq estimate for A var j}
A_{\varepsilon,j}=
-\frac{(N-2)Ac_{j,0}}{2(\lambda^{(1)}_{j,\varepsilon})^{N-2}}+
o\Big(\frac{1}{\bar{\lambda}_{\varepsilon}^{N-1}}\Big),
\end{equation}
and
\begin{equation}\label{eq estimate for B var j}
B_{\varepsilon,j,i}=-\frac{N+2}{N}(N(N-2))^{\frac{N+2}{4}}\int_{\R^N}\frac{|x|^2}{(1+|x|^2)^{\frac{N+4}{2}}} \frac{c_{j,i}}{\big(\lambda^{(1)}_{j,\varepsilon}\big)^{N-1}}
+o\Big(\frac{1}{\bar{\lambda}_\varepsilon^{N-1}}\Big),
\end{equation}
for any $j=1,\cdots,n$ and $i=1,\cdots,N$.
\end{Lem}
\begin{proof}
By using the Green's representation formula and \eqref{equation for xi}, we have
\begin{equation}\label{lem estimate for xi var-proof-1}
\begin{split}
 {\xi}_\varepsilon(x)=& \int_{\Omega}G(y,x)
 C_{2^*,\varepsilon}(y) \xi_{\varepsilon}(y)dy+  \varepsilon \int_{\Omega}G(y,x)C_{q,\varepsilon}(y) 
\xi_{\varepsilon}(y)dy.
\end{split}
\end{equation}
For any $d>0$ small enough and $x\in \Omega\backslash\displaystyle\bigcup^n_{j=1}B_{2d}(x^{(1)}_{j,\varepsilon})$,
by Taylor's expansion and Lemma \ref{lem xi var j}, the first term on the right-hand side of \eqref{lem estimate for xi var-proof-1} can be estimated as follows
\begin{equation}\label{lem estimate for xi var-proof-2}
\begin{split}
 \int_{\Omega}G&(y,x)
 C_{2^*,\varepsilon}(y) \xi_{\varepsilon}(y)dy \\=&
 \sum^n_{j=1}\int_{B_d\big(x^{(1)}_{j,\varepsilon}\big)}G(y,x)
 C_{2^*,\varepsilon}(y) \xi_{\varepsilon}(y)dy+
 \int_{\Omega\backslash\bigcup^n_{j=1}B_d\big(x^{(1)}_{j,\varepsilon}\big)}G(y,x)
 C_{2^*,\varepsilon}(y) \xi_{\varepsilon}(y)dy\\=&
 \sum^n_{j=1}A_{\varepsilon,j}G(x^{(1)}_{j,\varepsilon},x)
 +\sum^n_{j=1}\sum^N_{i=1}B_{\varepsilon,j,i}\partial_iG(x^{(1)}_{j,\varepsilon},x)\\&
 +O\Big(\sum^n_{j=1}\int_{B_d\big(x^{(1)}_{j,\varepsilon}\big)}|y-x^{(1)}_{j,\varepsilon}|^2
 C_{2^*,\varepsilon}(y) \xi_{\varepsilon}(y)dy\Big)+O\Big(\frac{\ln \bar{\lambda}_\varepsilon}{\bar\lambda^N_\varepsilon}
 \int_{\Omega\backslash\bigcup^n_{j=1}B_d\big(x^{(1)}_{j,\varepsilon}\big)}G(y,x)dy\Big)
 \\=&
  \sum^n_{j=1}A_{\varepsilon,j}G(x^{(1)}_{j,\varepsilon},x)
 + \sum^n_{j=1}\sum^N_{i=1}B_{\varepsilon,j,i}\partial_iG(x^{(1)}_{j,\varepsilon},x)
 +O\Big(\frac{\ln \bar{\lambda}_\varepsilon}{\bar\lambda^N_\varepsilon}\Big),
\end{split}
\end{equation}
where $A_{\varepsilon,j}$ and $B_{\varepsilon,j,i}$ are defined in \eqref{definition of A var and B var}.
On the other hand, from Lemma \ref{lem estimate for xi} and the fact that $q\geq 2^*-1$ and $\varepsilon\sim\frac{1}{\bar{\lambda}_{\varepsilon}^{\frac{N-2}{2}q-2}}$, $u_{\varepsilon}^{(k)}\lesssim U_{x_{j,\varepsilon}^{(k)},\lambda_{j,\varepsilon}^{(k))}}$ in $B_{d}(x_{j,\varepsilon}^{(k)})$ for any $k=1,2$ and $j=1,\cdots,n$, the second term on the right-hand side of \eqref{lem estimate for xi var-proof-1} can be bounded by
\begin{equation}\label{lem estimate for xi var-proof-3}
\begin{split}
 &\varepsilon\int_{\Omega}G(y,x)C_{q,\varepsilon}(y)\xi_{\varepsilon}(y)dy\\
 &=\varepsilon
 \sum^N_{j=1}\int_{B_d\big(x^{(1)}_{j,\varepsilon}\big)}G(y,x)C_{q,\varepsilon}(y)
 \xi_{\varepsilon}(y)dy+\varepsilon
\int_{\Omega\backslash\bigcup^n_{j=1}B_d\big(x^{(1)}_{j,\varepsilon}\big)}G(y,x)C_{q,\varepsilon}(y)
 \xi_{\varepsilon}(y)dy\\
 &=\frac{\ln \bar{\lambda}_\varepsilon}{\bar{\lambda}^{2N-4}_\varepsilon}
 O\big(\sum^n_{j=1}\int_{B\big(0,\lambda_{j,\varepsilon}^{(1)}d\big)}
 \frac{1}{(1+|x|)^{(N-2)(q-1)}}dy\big)+\frac{\ln \bar{\lambda}_\varepsilon}{\bar{\lambda}^{(N-2)q-2}_\varepsilon}O\Big(
\int_{\Omega }G(y,x)
 dy\Big)\\
 &=\begin{cases}
     \frac{\ln \bar{\lambda}_\varepsilon}{\bar{\lambda}^{2N-4}_\varepsilon},&\text{~if~}(N-2)(q-1)>N,\\
     \frac{(\ln \bar{\lambda}_\varepsilon)^{2}}{\bar{\lambda}^{2N-4}_\varepsilon},&\text{~if~}(N-2)(q-1)=N,\\
     \frac{\ln \bar{\lambda}_\varepsilon}{\bar{\lambda}^{N-4+(N-2)(q-1)}_\varepsilon},&\text{~if~}(N-2)(q-1)<N,\\
 \end{cases}+O\Big(\frac{\ln \bar{\lambda}_\varepsilon}{\bar{\lambda}^{N}_\varepsilon}\Big)\\
 &=O\Big(\frac{\ln \bar{\lambda}_\varepsilon}{\bar\lambda^N_\varepsilon}\Big).
\end{split}
\end{equation}
Hence combining \eqref{lem estimate for xi var-proof-1}-\eqref{lem estimate for xi var-proof-3}, we have
\begin{equation}
\xi_{\varepsilon}(x)=
  \sum^n_{j=1}A_{\varepsilon,j}G(x^{(1)}_{j,\varepsilon},x)
 +\sum^n_{j=1}\sum^N_{i=1}B_{\varepsilon,j,i}\partial_iG(x^{(1)}_{j,\varepsilon},x)
 +O\Big(\frac{\ln \bar{\lambda}_\varepsilon}{\bar\lambda^N_\varepsilon}\Big)~\mbox{for}~ x\in
 \Omega\backslash\bigcup^n_{j=1}B_{2d}(x^{(1)}_{j,\varepsilon}).
\end{equation}
Similarly, we can prove that the above estimate hold in $C^{1}(\Omega\backslash\bigcup^n_{j=1}B_{2d}(x^{(1)}_{j,\varepsilon}))$. Finally, we need to estimate the term $A_{\varepsilon,j}$ and $B_{\varepsilon,j,i}$. By using Lemma \ref{lem estimate for xi}, Lemma \ref{lem xi var j} and \eqref{perturbation estimate of w-varepsilon}
\begin{equation}
\begin{split}
A_{\varepsilon,j}=&\frac{1}{\big(\lambda^{(1)}_{j,\varepsilon}\big)^{N}}\int_{B_{\lambda^{(1)}_{j,\varepsilon}d}(0)} C_{2^*,\varepsilon}(\frac{x}{\lambda^{(1)}_{j,\varepsilon}}+x^{(1)}_{j,\varepsilon})\xi_{\varepsilon,j}(x)dx
\\
&=\frac{2^*-1}{\big(\lambda^{(1)}_{j,\varepsilon}\big)^{N-2}}
\int_{\R^N}U_{0,1}^{2^*-2}\Big(\sum^N_{l=1}c_{j,l} \psi_l(x)\Big)dx+
o\Big(\frac{1}{\bar{\lambda}_\varepsilon^{N-1}}\Big)\\
&=
-\frac{(N-2)Ac_{j,0}}{2(\lambda^{(1)}_{j,\varepsilon})^{N-2}}+
o\Big(\frac{1}{\bar{\lambda}_{\varepsilon}^{N-1}}\Big)
\end{split}
\end{equation}
and
\begin{equation}
\begin{split}
B_{\varepsilon,j,i}=&\frac{1}{\big(\lambda^{(1)}_{j,\varepsilon}\big)^{N+1}}\int_{B_{\lambda^{(1)}_{j,\varepsilon}d}(0)} x_i C_{2^*,\varepsilon}(\frac{x}{\lambda^{(1)}_{j,\varepsilon}}+x^{(1)}_{j,\varepsilon})\xi_{\varepsilon,j}(x)dx
\\=&\frac{2^*-1}{\big(\lambda^{(1)}_{j,\varepsilon}\big)^{N-1}}\int_{\R^N} x_i U_{0,1}^{2^*-2}\Big(\sum^N_{l=1}c_{j,l} \psi_l(x)\Big)dx
+o\Big(\frac{1}{\bar{\lambda}_\varepsilon^{N-1}}\Big)
\\=&-\frac{N+2}{N}(N(N-2))^{\frac{N+2}{2}}\int_{\R^N}\frac{|x|^2}{(1+|x|^2)^{\frac{N+4}{2}}} \frac{c_{j,i}}{\big(\lambda^{(1)}_{j,\varepsilon}\big)^{N-1}}
+o\Big(\frac{1}{\bar{\lambda}_\varepsilon^{N-1}}\Big),
\end{split}\end{equation}
for any $j=1,\cdots,n$ and $i=1,\cdots,N$.  
\end{proof}

\subsection{Proof of Theorem \ref{thm uniqueness}}

\begin{Lem}\label{lem estimate for u l}
Let $u^{(l)}_{\varepsilon}(x)$ with $l=1,2$ be the solutions of $\eqref{p-varepsion}$. Then for small fixed $d>0$, it holds
\begin{equation}\label{lem estimate for u l-1}
u^{(l)}_{\varepsilon}(x)=A\Big(\sum^n_{j=1}\frac{G(x^{(1)}_{j,\varepsilon},x)}{
(\lambda^{(1)}_{j,\varepsilon})^{(N-2)/2}}\Big)
+O\Big(\frac{1}{\bar\lambda_\varepsilon^{(N+2)/2}}\Big)~\mbox{in}~  C^1
\Big(\Omega\backslash\bigcup^n_{j=1}B_{2d}(x^{(1)}_{j,\varepsilon})\Big),
\end{equation}
where $A$ is the constant defined in \eqref{definition of A and B}.
\end{Lem}
\begin{proof}
First, from Proposition \ref{prop estimate for u}, we know that \eqref{lem estimate for u l-1} holds for $l=1$
and
\begin{equation}
u^{(2)}_{\varepsilon}(x)=A\Big(\sum^n_{j=1}\frac{G(x^{(2)}_{j,\varepsilon},x)}{
(\lambda^{(2)}_{j,\varepsilon})^{(N-2)/2}}\Big)+O\Big(\frac{1}{\bar\lambda_\varepsilon^{(N+2)/2}}\Big),
~\mbox{in}~  C^1
\Big(\Omega\backslash\bigcup^n_{j=1}B_{d}(x^{(2)}_{j,\varepsilon})\Big).
\end{equation}
On the other hand, from Theorem \ref{thm blow up rate}
\begin{equation}
\begin{aligned}
    \frac{G(x^{(2)}_{j,\varepsilon},x)}{
(\lambda^{(2)}_{j,\varepsilon})^{(N-2)/2}}&=
\frac{G(x^{(1)}_{j,\varepsilon},x)}{
(\lambda^{(1)}_{j,\varepsilon})^{(N-2)/2}}+O\Big(\frac{
|x^{(1)}_{j,\varepsilon}-x^{(2)}_{j,\varepsilon}|}{\bar{\lambda}_\varepsilon^{(N-2)/2}}\Big)
+O\Big(\frac{|\lambda^{(1)}_{j,\varepsilon}-\lambda^{(2)}_{j,\varepsilon}|}{\bar{\lambda}_\varepsilon^{N/2}}\Big)\\
&=\frac{G(x^{(1)}_{j,\varepsilon},x)}{
(\lambda^{(1)}_{j,\varepsilon})^{(N-2)/2}}+O\Big(\frac{1}{\bar\lambda_\varepsilon^{(N+2)/2}}\Big).
\end{aligned}
\end{equation}
Since $B_{d}(x^{(2)}_{j,\varepsilon})\subset B_{2d}(x^{(1)}_{j,\varepsilon})$ for any $j=1,\cdots,n$ and $\varepsilon$ small enough, thus \eqref{lem estimate for u l-1} hold for $l=2$.
\end{proof}

\begin{Lem}\label{lem local pohozaev for xi and u}
For $\xi_{\varepsilon}$ defined by \eqref{definition of xi}, we have the following local Poho\v{z}aev identities:
\begin{equation}\label{lem local pohozaev for xi and u-2}
\begin{split}
&\frac{1}{2}\int_{\partial \Omega'}
\big\langle \nabla (u_{\varepsilon}^{(1)}+u_{\varepsilon}^{(2)}),  \nabla \xi_{\varepsilon}\big\rangle
\big\langle x-x^{(1)}_{j,\varepsilon},\nu\big\rangle -\int_{\partial \Omega'}\frac{\partial\xi_\varepsilon}{\partial\nu}
\big\langle x-x^{(1)}_{j,\varepsilon},\nabla u_{\varepsilon}^{(1)}\big\rangle\\
&\quad-\int_{\partial \Omega'}\frac{\partial u^{(2)}_\varepsilon}{\partial\nu}
\big\langle x-x^{(1)}_{j,\varepsilon},\nabla \xi_{\varepsilon}\big\rangle
+\frac{2-N}{2}\int_{\partial \Omega'}\frac{\partial\xi_\varepsilon}{\partial\nu}
  u_{\varepsilon}^{(1)}
 +\frac{2-N}{2}\int_{\partial \Omega'}\frac{\partial u_{\varepsilon}^{(2)}}{\partial\nu}\xi_\varepsilon   \\
 &=
  \int_{\partial \Omega'} D_{2^*,\varepsilon}(x)\xi_\varepsilon \big\langle x-x^{(1)}_{j,\varepsilon},\nu\big\rangle  +\varepsilon\int_{\partial \Omega'} D_{q,\varepsilon}\xi_\varepsilon \big\langle x-x^{(1)}_{j,\varepsilon},\nu\big\rangle
  -\varepsilon\left(N-\frac{N-2}{2}q\right)\int_{ \Omega'} D_{q,\varepsilon}\xi_\varepsilon,
\end{split}
\end{equation}
and
\begin{equation}\label{lem local pohozaev for xi and u-1}
 \begin{split}
-\int_{\partial \Omega'}&\frac{\partial \xi_\varepsilon}{\partial \nu}\frac{\partial u^{(1)}_\varepsilon}{\partial x_i}-
\int_{\partial \Omega'}\frac{\partial u^{(2)}_\varepsilon}{\partial \nu}\frac{\partial \xi_\varepsilon}{\partial x_i}
+\frac{1}{2}\int_{\partial \Omega'}\big\langle \nabla (u^{(1)}_\varepsilon+u^{(2)}_\varepsilon),\nabla \xi_\varepsilon \big\rangle \nu_i\\&
= \int_{\partial \Omega'}D_{2^*,\varepsilon}(x) \xi_\varepsilon \nu_i+\varepsilon\int_{\partial \Omega'}D_{q,\varepsilon}(x) \xi_\varepsilon \nu_i,
\end{split}
\end{equation}
where  $\Omega'\subset \Omega$ is a smooth domain, $\nu(x)=\big(\nu_{1}(x),\cdots,\nu_N(x)\big)$ is the outward unit normal of $\partial \Omega'$, $j=1,\cdots,n$, $i=1,\cdots,N$ and
\begin{equation*}
D_{2^*,\varepsilon}(x)= \int_{0}^1\Big(tu_{\varepsilon}^{(1)}(x)+(1-t)u_{\varepsilon}^{(2)}(x)\Big)^{2^*-1}dt,\quad D_{q,\varepsilon}(x)= \int_{0}^1\Big(tu_{\varepsilon}^{(1)}(x)+(1-t)u_{\varepsilon}^{(2)}(x)\Big)^{q-1}dt.
\end{equation*}
\end{Lem}

\begin{proof}
Taking $u=u_{\varepsilon}^{(l)}$ with $l=1,2$ in \eqref{pohozaev identity 1} and \eqref{pohozaev identity 2}, making a difference respectively, and finally, multiplying both sides by $\|u_{\varepsilon}^{(1)}-u_{\varepsilon}^{(2)}\|_{L^{\infty}(\Omega)}^{-1}$, we can obtain \eqref{lem local pohozaev for xi and u-2} and \eqref{lem local pohozaev for xi and u-1}. 
\end{proof}

\begin{Prop}\label{prop c j 0}
It holds
\begin{equation}
c_{j,0}=0,~\mbox{for}~j=1,\cdots,n,
\end{equation}
where $c_{j,0}$ are the constants in Lemma \ref{lem xi var j}.
\end{Prop}
\begin{proof}
In the following, we will use the local Poho\v{z}aev identity \eqref{lem local pohozaev for xi and u-2}. Before going on, we need to define the following quadric form
\begin{equation}
\begin{split}
P_1(u,v)=&- \theta\int_{\partial B_\theta(x^{(1)}_{j,\varepsilon})}
\big\langle \nabla u ,\nu\big\rangle
\big\langle \nabla v,\nu\big\rangle
+ \frac{\theta}{2} \int_{\partial B_\theta(x^{(1)}_{j,\varepsilon})}
\big\langle \nabla u , \nabla v \big\rangle\\&
+\frac{2-N}{4}\int_{\partial B_\theta(x^{(1)}_{j,\varepsilon})}
\big\langle \nabla u ,  \nu \big\rangle v
+\frac{2-N}{4}\int_{\partial B_\theta(x^{(1)}_{j,\varepsilon})}
\big\langle \nabla v ,  \nu \big\rangle u,
\end{split}
\end{equation}
which is well defined and if $u$ and $v$ are harmonic in $ B_d\big(x^{(1)}_{j,\varepsilon}\big)\backslash \{x^{(1)}_{j,\varepsilon}\}$, then $P_1(u,v)$ is independent of $\theta\in (0,d]$, see Lemma \ref{propertity of P Q}. Next, we take $\Omega'=B_\theta\big(x^{(1)}_{j,\varepsilon}\big)$ in \eqref{lem local pohozaev for xi and u-2}, then from Lemma \ref{lem estimate for xi var} and Lemma \ref{lem estimate for u l}, the left-hand side of \eqref{lem local pohozaev for xi and u-2} becomes
\begin{equation}\label{prop c j 0-proof-2}
\begin{split}
 \text{LHS}=&\sum^n_{m=1}\sum^n_{l=1}\frac{2AA_{\varepsilon,l}}{(\lambda^{(1)}_{m,\varepsilon})^{(N-2)/2}}
P_1\Big(G(x^{(1)}_{m,\varepsilon},x),G(x^{(1)}_{l,\varepsilon},x)\Big)
\\&+\sum^n_{m=1}\sum^n_{l=1}\sum^N_{h=1}\frac{2AB_{\varepsilon,l,h}}{(\lambda^{(1)}_{m,\varepsilon})^{(N-2)/2}}
P_1\Big(G(x^{(1)}_{m,\varepsilon},x),\partial_hG(x^{(1)}_{l,\varepsilon},x)\Big)
+O\Big(\frac{\ln\bar{\lambda}_{\varepsilon}}{\bar{\lambda}_\varepsilon^{(3N-2)/2}}\Big).
\end{split}
\end{equation}
Moreover, from Lemma \ref{lem for estimate of quadractic form} and \eqref{eq estimate for A var j}, we can rewrite \eqref{prop c j 0-proof-2} as
\begin{equation}\label{prop c j 0-proof-3}
\begin{split}
 \text{LHS}=&
 {2(N-2)d_{j,\varepsilon}} \Big(\frac{R(x^{(1)}_{j,\varepsilon})}{(\lambda^{(1)}_{j,\varepsilon})^{N-2}}-\sum^n_{l\neq j}\frac{G(x^{(1)}_{j,\varepsilon},x^{(1)}_{l,\varepsilon})}
 {(\lambda^{(1)}_{j,\varepsilon})^{(N-2)/2}(\lambda^{(1)}_{l,\varepsilon})^{(N-2)/2}}
\Big)\\&
+{2(N-2)} \Big(\frac{d_{j,\varepsilon} R(x^{(1)}_{j,\varepsilon})}{(\lambda^{(1)}_{j,\varepsilon})^{N-2}}-\sum^n_{l\neq j}\frac{d_{l,\varepsilon}G(x^{(1)}_{j,\varepsilon},x^{(1)}_{l,\varepsilon})}
 {(\lambda^{(1)}_{j,\varepsilon})^{(N-2)/2}(\lambda^{(1)}_{l,\varepsilon})^{(N-2)/2}}
\Big)\\&
-\frac{N-2}{2}A\sum^N_{h=1} B_{\varepsilon,j,h} \Big(\frac{\partial_hR(x^{(1)}_{j,\varepsilon}) }{(\lambda^{(1)}_{j,\varepsilon})^{(N-2)/2}} -\sum^n_{l\neq j}
\frac{\partial_h G(x^{(1)}_{j,\varepsilon},x^{(1)}_{l,\varepsilon})}{(\lambda^{(1)}_{l,\varepsilon})^{(N-2)/2}}
\Big) \\&
+\frac{N-2}{2} A\sum^N_{h=1}\sum^n_{l\neq j} \frac{B_{\varepsilon,l,h}}{(\lambda^{(1)}_{j,\varepsilon})^{(N-2)/2}}\partial_h G(x^{(1)}_{j,\varepsilon},x^{(1)}_{l,\varepsilon})
+o\Big(\frac{1}{\bar{\lambda}_\varepsilon^{(3N-4)/2}}\Big)\\&
-
\frac{1}{2}A \sum^N_{h=1}B_{\varepsilon,j,h} \Big(\frac{ \partial_hR(x^{(1)}_{j,\varepsilon})}{(\lambda^{(1)}_{j,\varepsilon})^{(N-2)/2}}-2\sum^n_{l\neq j}\frac{ \partial_hG(x^{(1)}_{j,\varepsilon},x^{(1)}_{l,\varepsilon})}
 {(\lambda^{(1)}_{l,\varepsilon})^{(N-2)/2}}
\Big).
\end{split}
\end{equation}
where $
d_{j,\varepsilon}=\frac{(N-2)A^2c_{j,0}}{8(\lambda^{(1)}_{j,\varepsilon})^{(N-2)/2}}$, for $j=1,\cdots,n$.
Furthermore, from \eqref{prop blow up rate 2 proof 4} and \eqref{eq estimate for B var j} 
\begin{equation}\label{prop c j 0-proof-4}
\sum^N_{h=1}B_{\varepsilon,j,h} \Big(\frac{ \partial_hR(x^{(1)}_{j,\varepsilon})}{(\lambda^{(1)}_{j,\varepsilon})^{(N-2)/2}}-2\sum^n_{l\neq j}\frac{ \partial_hG(x^{(1)}_{j,\varepsilon},x^{(1)}_{l,\varepsilon})}
 {(\lambda^{(1)}_{l,\varepsilon})^{(N-2)/2}}
\Big) =O\Big(\frac{1}{\bar{\lambda}_\varepsilon^{3N/2}}\Big).
\end{equation}
Hence combining \eqref{prop c j 0-proof-3} and \eqref{prop c j 0-proof-4}
\begin{equation}\label{prop c j 0-proof-5}
\begin{split}
 \text{LHS}=&
 {2(N-2)d_{j,\varepsilon}} \Big(\frac{R(x^{(1)}_{j,\varepsilon})}{(\lambda^{(1)}_{j,\varepsilon})^{N-2}}-\sum^n_{l\neq j}\frac{G(x^{(1)}_{j,\varepsilon},x^{(1)}_{l,\varepsilon})}
 {(\lambda^{(1)}_{j,\varepsilon})^{(N-2)/2}(\lambda^{(1)}_{l,\varepsilon})^{(N-2)/2}}
\Big)\\&
+{2(N-2)} \Big(\frac{d_{j,\varepsilon} R(x^{(1)}_{j,\varepsilon})}{(\lambda^{(1)}_{j,\varepsilon})^{N-2}}-\sum^n_{l\neq j}\frac{d_{l,\varepsilon}G(x^{(1)}_{j,\varepsilon},x^{(1)}_{l,\varepsilon})}
 {(\lambda^{(1)}_{j,\varepsilon})^{(N-2)/2}(\lambda^{(1)}_{l,\varepsilon})^{(N-2)/2}}
\Big)\\&
-\frac{N-2}{2}A\sum^N_{h=1} B_{\varepsilon,j,h} \Big(\frac{\partial_hR(x^{(1)}_{j,\varepsilon}) }{(\lambda^{(1)}_{j,\varepsilon})^{(N-2)/2}} -\sum^n_{l\neq j}
\frac{\partial_h G(x^{(1)}_{j,\varepsilon},x^{(1)}_{l,\varepsilon})}{(\lambda^{(1)}_{l,\varepsilon})^{(N-2)/2}}
\Big) \\&
+\frac{N-2}{2} A\sum^N_{h=1}\sum^n_{l\neq j} \frac{B_{\varepsilon,l,h}}{(\lambda^{(1)}_{j,\varepsilon})^{(N-2)/2}}\partial_h G(x^{(1)}_{j,\varepsilon},x^{(1)}_{l,\varepsilon})
+o\Big(\frac{1}{\bar{\lambda}_\varepsilon^{(3N-4)/2}}\Big).
\end{split}
\end{equation}
On the other hand, from Lemma \ref{lem estimate for xi}, Lemma \ref{lem xi var j}, Lemma \ref{lem estimate for u l}, Lemma \ref{Elementary estimate}, \eqref{lem xi var j-proof-4}, \eqref{perturbation estimate of w-varepsilon}, \eqref{Prop blow up rate proof 5} and the fact $\varepsilon\sim\frac{1}{\bar{\lambda}_{\varepsilon}^{\frac{N-2}{2}q-2}}$, the right-hand side of \eqref{lem local pohozaev for xi and u-2}
\begin{equation}\label{prop c j 0-proof-6}
\begin{split}
\text{RHS}=& -(N-\frac{N-2}{2}q)\varepsilon\int_{ B_\theta\big(x^{(1)}_{j,\varepsilon}\big)} U_{x^{(1)}_{j,\varepsilon},\lambda^{(1)}_{j,\varepsilon}}^{q-1}(x)\xi_\varepsilon
+o\Big(\frac{1}{\bar{\lambda}_\varepsilon^{(3N-4)/2}}\Big)\\=&
\frac{(2N-(N-2)q)^{2}}{4q}\frac{Bc_{j,0}\varepsilon}{(\lambda^{(1)}_{j,\varepsilon})^{N-\frac{N-2}{2}(q-1)}}
+o\Big(\frac{1}{\bar{\lambda}_\varepsilon^{(3N-4)/2}}\Big)\\=&
2(2N-(N-2)q){d_{j,\varepsilon}} \Big(\frac{R(x^{(1)}_{j,\varepsilon})}{(\lambda^{(1)}_{j,\varepsilon})^{N-2}}-\sum^N_{l\neq j}\frac{G(x^{(1)}_{j,\varepsilon},x^{(1)}_{l,\varepsilon})}
 {(\lambda^{(1)}_{j,\varepsilon})^{(N-2)/2}(\lambda^{(1)}_{l,\varepsilon})^{(N-2)/2}}
\Big)+o\Big(\frac{1}{\bar{\lambda}_\varepsilon^{(3N-4)/2}}\Big).
\end{split}
\end{equation}
Combining \eqref{lem local pohozaev for xi and u-2}, \eqref{prop c j 0-proof-5} and \eqref{prop c j 0-proof-6} together, we can obtain that for $j=1,\cdots,n$,
\begin{equation}\label{prop c j 0-proof-7}
\begin{split}
 &{((N-2)(q+1)-2N)d_{j,\varepsilon}} \Big(\frac{R(x^{(1)}_{j,\varepsilon})}{(\lambda^{(1)}_{j,\varepsilon})^{N-2}}-\sum^n_{l\neq j}\frac{G(x^{(1)}_{j,\varepsilon},x^{(1)}_{l,\varepsilon})}
 {(\lambda^{(1)}_{j,\varepsilon})^{(N-2)/2}(\lambda^{(1)}_{l,\varepsilon})^{(N-2)/2}}
\Big)\\&
+{(N-2)} \Big(\frac{d_{j,\varepsilon} R(x^{(1)}_{j,\varepsilon})}{(\lambda^{(1)}_{j,\varepsilon})^{N-2}}-\sum^n_{l\neq j}\frac{d_{l,\varepsilon}G(x^{(1)}_{j,\varepsilon},x^{(1)}_{l,\varepsilon})}
 {(\lambda^{(1)}_{j,\varepsilon})^{(N-2)/2}(\lambda^{(1)}_{l,\varepsilon})^{(N-2)/2}}
\Big)\\&
-\frac{N-2}{4}A\sum^N_{h=1} B_{\varepsilon,j,h} \Big(\frac{\partial_hR(x^{(1)}_{j,\varepsilon}) }{(\lambda^{(1)}_{j,\varepsilon})^{(N-2)/2}} -\sum^n_{l\neq j}
\frac{\partial_h G(x^{(1)}_{j,\varepsilon},x^{(1)}_{l,\varepsilon})}{(\lambda^{(1)}_{l,\varepsilon})^{(N-2)/2}}
\Big) \\&
+\frac{N-2}{4}A \sum^N_{h=1}\sum^n_{l\neq j} \frac{B_{\varepsilon,l,h}}{(\lambda^{(1)}_{j,\varepsilon})^{(N-2)/2}}\partial_h G(x^{(1)}_{j,\varepsilon},x^{(1)}_{l,\varepsilon})
=o\Big(\frac{1}{\bar{\lambda}_\varepsilon^{(3N-4)/2}}\Big).
\end{split}
\end{equation}
Finally, we define 
\begin{equation}\label{definition D n N}
\begin{split}
   &\widetilde{D}_{1n}=\big(c_{1,0},\cdots,c_{n,0}\big)^T,~
\widetilde{D}_{nn}=\text{diag} \big(\lambda^{\frac{4-N}{2}}_{1},\cdots,\lambda^{\frac{4-N}{2}}_{N}\big),\\
&\widetilde{D}_{nN}=\text{diag}\big((\text{diag}(\lambda^{\frac{2-N}{2}}_{1},\cdots,\lambda^{\frac{2-N}{2}}_{1})_{N\times N},\cdots,\text{diag}(\lambda^{\frac{2-N}{2}}_{n},\cdots,\lambda^{\frac{2-N}{2}}_{n})_{N\times N}\big), 
\end{split}
\end{equation}
and
\begin{equation}\label{definition B var N}
\widetilde{B}_{\varepsilon,n}=\big(\bar B_{\varepsilon,1},\cdots,\bar B_{\varepsilon,nN}\big)^T,~\mbox{with}~
\bar B_{\varepsilon,m}=B_{\varepsilon,j,h},~ m=(j-1)N+h.
\end{equation}
Recall also the definition of $\vec{\lambda}=(\lambda_{1},\cdots,\lambda_{n})$ with
\begin{equation}
\lambda_{j}:=\lim_{\varepsilon\to 0}\Big(\varepsilon ^{\frac{2}{(N-2)q-4}}\lambda_{j,\varepsilon}\Big)^{-1}.
\end{equation}
Then from Lemma \ref{Lem derivate for reduce function}, we rewrite \eqref{prop c j 0-proof-7} as follows:
\begin{equation}\label{prop c j 0-proof-11}
\begin{split}
&\frac{(N-2)A^{2}}{8}\widetilde{M}_{n,\varepsilon} \widetilde{D}_{1n}-\frac{1}{2A}
\widetilde{D}_{nn}
\Big(D^2_{\lambda,x}\Phi_n(x,\lambda)\Big)_{(x,\lambda)=(\vec{a},\vec{\lambda})}\widetilde{D}_{nN}
\widetilde{B}_{\varepsilon,n}\\ 
&=\Big(o\big(\frac{1}{\bar{\lambda}_\varepsilon^{N-1}}\big),
\cdots,o\big(\frac{1}{\bar{\lambda}_\varepsilon^{N-1}}\big)\Big)^T,
\end{split}\end{equation}
where matrix $\widetilde{M}_{n,\varepsilon}=\big(a_{i,j,\varepsilon}\big)_{1\leq i,j\leq n}$ is defined by
\begin{equation}
~a_{i,j,\varepsilon}=
\begin{cases}
\frac{((N-2)(q+2)-2N)R(x^{(1)}_{j,\varepsilon})}{(\lambda^{(1)}_{j,\varepsilon})^{N-2}}-
 \sum^n_{l\neq j}\frac{((N-2)(q+1)-2N)G(x^{(1)}_{j,\varepsilon},x^{(1)}_{l,\varepsilon})}
 {(\lambda^{(1)}_{j,\varepsilon})^{(N-2)/2}(\lambda^{(1)}_{l,\varepsilon})^{(N-2)/2}},& \mbox{for}~i=j,\\
 -\frac{(N-2)G(x^{(1)}_{j,\varepsilon},x^{(1)}_{i,\varepsilon})}
 {(\lambda^{(1)}_{j,\varepsilon})^{(N-2)/2}(\lambda^{(1)}_{i,\varepsilon})^{(N-2)/2}},& \mbox{for}~i\neq j.
\end{cases}
\end{equation}
From \eqref{Prop blow up rate proof 5} and the fact that $q\geq 2^*-1$, we see that $\widetilde{M}_{n,\varepsilon}$ is the main diagonally dominant matrix, thus $\widetilde{M}_{n,\varepsilon}$ is invertible. Moreover, from \eqref{eq estimate for B var j}, we have
\begin{equation}
\widetilde{B}_{\varepsilon,n}=O(\frac{1}{\bar{\lambda}_{\varepsilon}^{N-1}})+o\Big(\frac{1}{\bar{\lambda}_\varepsilon^{N-1}}\Big).
\end{equation}
Hence \eqref{prop c j 0-proof-11} gives that
\begin{equation}
    |c_{j,0}|=O(\frac{1}{\bar{\lambda}_{\varepsilon}}),~\mbox{for}~j=1,\cdots,n,
\end{equation}
thus $c_{j,0}=0$ for any $j=1,\cdots,n$ and then we can obtain that
\begin{equation}\label{estimate for a var j}
A_{\varepsilon,j}=
o\Big(\frac{1}{\bar{\lambda}_{\varepsilon}^{N-1}}\Big),~\mbox{for}~j=1,\cdots,n,
\end{equation}
and
 \begin{equation}\label{estimate for D Phi lamda x}
\begin{split}
\Big(D^2_{\lambda,x}\Phi_n(x,\lambda)\Big)_{(x,\lambda)=(\vec{a},\vec{\Lambda})}\widetilde{D}_{nN}
\widetilde{B}_{\varepsilon,n} =\Big(o\big(\frac{1}{\bar{\lambda}_\varepsilon^{N-1}}\big),
\cdots,o\big(\frac{1}{\bar{\lambda}_\varepsilon^{N-1}}\big)\Big)^T.
\end{split}\end{equation}
\end{proof}

\begin{Prop}\label{prop estimate for c ji}
It holds
\begin{equation}
c_{j,i}=0,~\mbox{for}~j=1,\cdots,n,~\textrm{and}~i=1,\cdots,N,
\end{equation}
where $c_{j,i}$ are the constants in Lemma \ref{lem xi var j}.
\end{Prop}
\begin{proof}
In the following, we will use the local Poho\v{z}aev identity \eqref{lem local pohozaev for xi and u-1}. Taking $\Omega'=B_\theta\big(x^{(1)}_{j,\varepsilon}\big)$ and recall the definition of
quadratic form 
\begin{equation}
Q_1(u,v)=-\int_{\partial B_\theta(x^{(1)}_{j,\varepsilon})}\frac{\partial v}{\partial \nu^{(1)}}\frac{\partial u}{\partial x_i}-
\int_{\partial B_\theta(x^{(1)}_{j,\varepsilon})}\frac{\partial u}{\partial \nu^{(1)}}\frac{\partial v}{\partial x_i}
+\int_{\partial B_\theta(x^{(1)}_{j,\varepsilon})}\big\langle \nabla u,\nabla v \big\rangle \nu^{(1)}_i,
\end{equation}
where $\nu^{(1)}(x)=\big(\nu^{(1)}_{1}(x),\cdots,\nu^{(1)}_N(x)\big)$ is the outward unit normal of $\partial B_{\theta}(x^{(1)}_{j,\varepsilon})$. Then from Lemma \ref{lem estimate for xi var} and Lemma \ref{lem estimate for u l}, the left-hand side of \eqref{lem local pohozaev for xi and u-1} becomes
\begin{equation}\label{prop estimate for c ji-proof-2}
\begin{split}
 \text{LHS}=&\sum^n_{m=1}\sum^n_{l=1}\frac{AA_{\varepsilon,l}Q_1
 \Big(G(x^{(1)}_{m,\varepsilon},x),G(x^{(1)}_{l,\varepsilon},x)\Big)}{(\lambda^{(1)}_{m,\varepsilon})^{(N-2)/2}}
\\&+
\sum^n_{m=1}\sum^n_{l=1}\sum^N_{h=1}\frac{AB_{\varepsilon,l,h}Q_1
\Big(G(x^{(1)}_{m,\varepsilon},x),\partial_hG(x^{(1)}_{l,\varepsilon},x)\Big)}{(\lambda^{(1)}_{m,\varepsilon})^{(N-2)/2}}
+O\Big(\frac{\ln \bar{\lambda}_\varepsilon}{\bar{\lambda}_\varepsilon^{(3N-2)/2}}\Big).
\end{split}
\end{equation}
On the other hand, from Lemma \ref{lem estimate for xi} and Lemma \ref{lem estimate for u l} with the fact $\varepsilon\sim\frac{1}{\bar{\lambda}_{\varepsilon}^{\frac{N-2}{2}q-2}}$ and $q\geq 2^*-1$, the right-hand side of \eqref{lem local pohozaev for xi and u-1} becomes
\begin{equation}\label{prop estimate for c ji-proof-3}
\begin{split}
 \text{RHS}= O\Big(\frac{\ln\bar{\lambda}_{\varepsilon}}{\bar{\lambda}_\varepsilon^{(3N-2)/2}}\Big).
\end{split}
\end{equation}
Then from Lemma \ref{lem for estimate of quadractic form}, \eqref{prop estimate for c ji-proof-2}, \eqref{prop estimate for c ji-proof-3} and \eqref{estimate for a var j} we can deduce that
\begin{equation}
\begin{split}
\sum^n_{m=1}\sum^n_{l=1}\sum^N_{h=1}\frac{B_{\varepsilon,l,h}Q_1
\Big(G(x^{(1)}_{m,\varepsilon},x),\partial_hG(x^{(1)}_{l,\varepsilon},x)\Big)}{(\lambda^{(1)}_{m,\varepsilon})^{(N-2)/2}}
=o\Big(\frac{1}{\bar{\lambda}_\varepsilon^{(3N-4)/2}}\Big).
\end{split}
\end{equation}
Moreover, from Lemma \ref{lem for estimate of quadractic form}, we have
\begin{equation}\label{prop estimate for c ji-proof-5}
\begin{split}
\sum^N_{h=1} B_{\varepsilon,j,h}&\Big(\frac{\partial^2_{x_ix_h}R(x^{(1)}_{j,\varepsilon})}
{(\lambda^{(1)}_{j,\varepsilon})^{(N-2)/2}}-2\sum_{m\neq j}^{n}\frac{
 D^2_{x_i,x_h}G(x^{(1)}_{m,\varepsilon},x^{(1)}_{l,\varepsilon})}{{(\lambda^{(1)}_{m,\varepsilon})^{(N-2)/2}}}\Big)\\&
-2\sum^N_{h=1}\sum_{l\neq j}^{n}\frac{B_{\varepsilon,l,h}
D_{x_h}\partial_{x_i}G(x^{(1)}_{j,\varepsilon},x^{(1)}_{l,\varepsilon})}
{{(\lambda^{(1)}_{j,\varepsilon})^{(N-2)/2}}}=o\Big(\frac{1}{\bar{\lambda}_\varepsilon^{(3N-4)/2}}\Big).
\end{split}
\end{equation}
Furthermore, by Lemma \ref{Lem derivate for reduce function} we can rewrite \eqref{prop estimate for c ji-proof-5} as follows
\begin{equation}
\Big(D^2_{xx}\Phi_n(x,\lambda)\Big)_{(x,\lambda)=(\vec{a},\vec{\lambda})}
\widetilde{D}_{nN}\widetilde{B}_{\varepsilon,n}=\Big(o\big(\frac{1}{\bar{\lambda}_\varepsilon^{N-1}}\big),\cdots,
o\big(\frac{1}{\bar{\lambda}_\varepsilon^{N-1}}\big)\Big)^T,
\end{equation}
where
$\widetilde{D}_{nN},\widetilde{B}_{\varepsilon,n}$ are defined in \eqref{definition D n N} and \eqref{definition B var N} respectively.
Noting that $(\vec{a},\vec{\lambda})$ is a nondegenerate  critical point of $\Phi_{n}$, then
\begin{equation}
\mbox{Rank}~\Big(D^2_{x,x}\Phi_n(x,\lambda);D^2_{\lambda,x}\Phi_n(x,\lambda)\Big)_{(x,\lambda)=(\vec{a},\vec{\lambda})}=nN,
\end{equation}
which together with \eqref{estimate for D Phi lamda x}, it follows  that
\begin{equation}
B_{\varepsilon,j,h}=o\Big(\frac{1}{\bar{\lambda}_\varepsilon^{N-1}}\Big), ~\mbox{for}~ j=1,\cdots,n~\mbox{and}~
h=1,\cdots,N.
\end{equation}
Note that from \eqref{eq estimate for B var j}, it holds 
\begin{equation}
B_{\varepsilon,j,h}\sim c_{j,h}\frac{1}{\bar{\lambda}_{\varepsilon}^{N-1}}+o\Big(\frac{1}{\bar{\lambda}_\varepsilon^{N-1}}\Big).
\end{equation}
Thus $c_{j,h}=0$ for any $j=1,\cdots,n$ and $h=1,\dots,N$.
\end{proof}

\begin{proof}[Proof of Theorem \ref{thm uniqueness}]
From Lemma \ref{lem estimate for xi}, we know that
\begin{equation}
|\xi_{\varepsilon}(x)|=O\Big(\frac{1}{R^2}\Big)+O(\varepsilon),~\mbox{for any}~
x\in\Omega\backslash\bigcup_{j=1}^n B_{R(\lambda^{(1)}_{j,\varepsilon})^{-1}}(x_{j,\varepsilon}^{(1)}).
\end{equation}
Thus for any fixed $\gamma\in (0,1)$ and small $\varepsilon$, there exists $R_1>0$ such that
\begin{equation}
|\xi_{\varepsilon}(x)|\leq \gamma,~ \text{for any~~}x\in\Omega\backslash\bigcup_{j=1}^n B_{R_1(\lambda^{(1)}_{j,\varepsilon})^{-1}}(x_{j,\varepsilon}^{(1)}).
\end{equation}
On the other hand, for the above fixed $R_1$, from Lemma \ref{lem xi var j}, Proposition \ref{prop c j 0} and Proposition \ref{prop estimate for c ji}, we have
\begin{equation}
\xi_{\varepsilon,j}(x)=o(1)~\mbox{in}~ B_{R_1}(0),~j=1,\cdots,n.
\end{equation}
Note that $\xi_{\varepsilon,j}(x)=\xi_{\varepsilon}(
\frac{x}{\lambda^{(1)}_{j,\varepsilon}}+x_{j,\varepsilon}^{(1)})$, thus
\begin{equation}
\xi_{\varepsilon}(x)=o(1),~x\in \bigcup_{j=1}^n B_{R_1(\lambda^{(1)}_{j,\varepsilon})^{-1}}(x_{j,\varepsilon}^{(1)}).
\end{equation}
Hence for any fixed $\gamma\in (0,1)$ and small $\varepsilon$, we have
$|\xi_{\varepsilon}(x)|\leq \gamma$ for all $x\in \Omega$,
this make a contradiction with $\|\xi_{\varepsilon}\|_{L^{\infty}(\Omega)}=1$. Thus $u^{(1)}_{\varepsilon}(x)\equiv u^{(2)}_{\varepsilon}(x)$ for small $\varepsilon$.
\end{proof}

\begin{proof}[Proof of Corollary \ref{thm number}]
Since $\Omega$ satisfies Assumption B, then from the results in \cite{Pistoia2004} (see also \cite{Molle2003}), we have
$$
\mbox{the number of solutions to \eqref{p-varepsion}}~~\geq\displaystyle\sum^{n_0}_{n=1}\sharp T_n,$$
for any $\varepsilon>0$ small enough. Moreover, Since $\Omega$ satisfies Assumption A, then by Theorem \ref{thm multibble blowup} and Theorem \ref{thm uniqueness}, we can deduce that
$$
\mbox{the number of solutions to \eqref{p-varepsion}}~~=\displaystyle\sum^{n_0}_{n=1}\sharp T_n,$$
for any $\varepsilon>0$ small enough.
\end{proof}

\section*{Appendix}

\appendix
\renewcommand{\theequation}{A.\arabic{equation}}
\setcounter{equation}{0}

\section{Preliminaries}

\subsection{Elementary inequality}
 
 \begin{Lem}\cite[Lemma 6.1.1]{Cao_Peng_Yan_2021}\label{Elementary estimate}
  For $a,b>0$, we have the following estimate
  \begin{equation}
 \begin{aligned}
     &(a+b)^{p}=a^{p}+pa^{p-1}b+O(b^{p}),\text{~if~} p\in(1,2],\\
    &(a+b)^{p}=a^{p}+pa^{p-1}b+\frac{p(p-1)}{2}a^{p-2}b^{2}+O(b^{p}),\text{~if~} p>2.
 \end{aligned}
 \end{equation}
\end{Lem}
\begin{Lem}\cite[Lemma 6.1.2]{Cao_Peng_Yan_2021}\label{Elementary estimate 2}
  For any $m\in\N$, we have the following estimate
  \begin{equation}
  (\sum_{i=1}^{m}a_{i})^{p}-\sum_{i=1}^{m}a_{i}^{p}=\begin{cases}
      O(\sum_{i\neq j}a_{i}^{\frac{p}{2}}a_{j}^{\frac{p}{2}}),&\text{~if~} p\in(1,2],\\
       O(\sum_{i\neq j}a_{i}^{p-1}a_{j}),&\text{~if~} p>2.
  \end{cases}
 \end{equation}
\end{Lem}
\begin{Lem}\cite[Lemma B.2]{Wei2010InfinitelyMS}\label{useful estimate}
    For any constant $\theta>0$, there exist a constant $C>0$ such that
    \begin{equation}
        \int_{\R^{N}}\frac{1}{|y-z|^{N-2}}\frac{1}{(1+|z|)^{2+\theta}}dz\leq \begin{cases}
            C(1+|y|)^{-\theta},&\text{~if~}\theta<N-2,\\
            C|\log|y||(1+|y|)^{-\theta},&\text{~if~}\theta\geq N-2.\\
        \end{cases}
    \end{equation}
\end{Lem}

\subsection{Projection of bubble}
We recall some useful estimates from \cite{bahri1988,Rey1990}.
\begin{Lem}\label{estimate of two different bubbles}
For any $\lambda_{1}, \lambda_{2}\in \R^{+}$ and $a_{1},a_{2}\in\R$, we have
\begin{equation}
    \int_{\R^{N}}U_{a_{1},\lambda_{1}}^{2^*-1}U_{a_{2},\lambda_{2}}=O\left(\frac{\lambda_{1}}{\lambda_{2}}+\frac{\lambda_{2}}{\lambda_{1}}+\lambda_{1}\lambda_{2}|a_{1}-a_{2}|^{2}\right)^{-\frac{N-2}{2}}.
\end{equation}
\end{Lem}

\begin{Lem}\label{estimate of U-lambda-a and psi-lambda-a 1}
Assume that $a\in\Omega$ and $\lambda\in \R^{+}$, then we have the following properties    \begin{equation}\begin{aligned}
&\frac{\partial U_{a,\lambda}(x)}{\partial a_j} =\alpha_{N}(N-2)\lambda^{\frac{N+2}{2}}\frac{x_{j}-a_{j}}{\left(1+\lambda^{2}|x-a|^{2}\right)^{\frac{N}{2}}}=O\big(\lambda U_{a,\lambda}\big), \\
&\frac{\partial U_{a,\lambda}(x)}{\partial\lambda} =\alpha_{N}\frac{N-2}2\lambda^{\frac{N-4}2}\frac{1-\lambda^2|x-a|^2}{(1+\lambda^2|x-a|^2)^{\frac N2}}=O\Big(\frac{U_{a,\lambda}}\lambda\Big),
\end{aligned}\end{equation}
where $j=1,\cdots,N$ and  $d=\text{dist}(a,\partial\Omega)$ is the distance between $a$ and boundary $\partial\Omega$.
\end{Lem}

\begin{Lem}\label{estimate of U-lambda-a and psi-lambda-a 2}
Assume that $a\in\Omega$ and $\lambda\in \R^{+}$, then we have the following properties
\begin{equation}
 \begin{aligned}
PU_{a,\lambda}& =U_{a,\lambda}-\psi_{a,\lambda}\geq 0, \\
\psi_{a,\lambda}(x)& =\alpha_{N}\frac{(N-2)\omega_N}{\lambda^{\frac{N-2}2}}H(a,x)+O\Big(\frac1{\lambda^{\frac{N+2}2}d^N}\Big), \\
\frac{\partial\psi_{a,\lambda}(x)}{\partial\lambda}& =-\alpha_{N}\frac{(N-2)\omega_N}{\lambda^{\frac N2}}H(a,x)+O\Big(\frac1{\lambda^{\frac{N+4}2}d^N}\Big), \\
\frac{\partial\psi_{a,\lambda}(x)}{\partial a_j}& =\alpha_{N}\frac{(N-2)\omega_N}{\lambda^{\frac{N-2}2}}\frac{\partial H(a,x)}{\partial a_j}+O\Big(\frac{1}{\lambda^{\frac{N+2}2}d^{N+1}}\Big),
\end{aligned}      
\end{equation}
 where $j=1,\cdots,N$ and $d=\text{dist}(a,\partial\Omega)$ is the distance between $a$ and boundary $\partial\Omega$.
\end{Lem}

\begin{Lem}\label{estimate of U-lambda-a and psi-lambda-a 3}
Assume that $a\in\Omega$ and $\lambda\in \R^{+}$, then we have the following properties
    \begin{equation}
        \begin{aligned}
 &\parallel\psi_{a,\lambda}\parallel_{L^{2^{*}}}=O\Big(\frac{1}{(\lambda d)^{\frac{N-2}{2}}}\Big),~~\|\psi_{a,\lambda}\|_{L^{\infty}} =O\left(\frac1{\lambda^{\frac{N-2}2}d^{N-2}}\right),\\
 &\|\frac{\psi_{a,\lambda}}{\partial a_{j}}\|_{L^{2^*}}=O\left(\frac{1}{\lambda^{\frac{N-2}{2}}d^{\frac{N}{2}}}\right),~~\|\frac{\partial \psi_{a,\lambda}}{\partial a_{j}}\|_{L^{\infty}} =O\left(\frac1{\lambda^{\frac{N-2}2}d^{N-1}}\right),\\
 &\|\frac{\psi_{a,\lambda}}{\partial \lambda}\|_{L^{2^*}}=O\left(\frac{1}{\lambda^{\frac{N}{2}}d^{\frac{N-2}{2}}}\right),~~\|\frac{\partial \psi_{a,\lambda}}{\partial \lambda}\|_{L^{\infty}} =O\left(\frac1{\lambda^{\frac{N}2}d^{N-2}}\right),\\
        \end{aligned}
    \end{equation}
where $j=1,\cdots,N$ and $d=\text{dist}(a,\partial\Omega)$ is the distance between $a$ and boundary $\partial\Omega$.
\end{Lem}

\begin{Lem}\label{estimate of U-lambda-a and psi-lambda-a 4}
There exist constants $B,C>0$ dependent only on $N$ such that 
\begin{equation}
    \left\langle PU_{a,\lambda},\frac{\partial PU_{a,\lambda}}{\partial\lambda}\right\rangle=O\left(\frac{1}{\lambda^{N-1}}\right),~~\left\langle PU_{a,\lambda},\frac{\partial PU_{a,\lambda}}{\partial a_{i}}\right\rangle=O\left(\frac{1}{\lambda^{N-2}}\right),
\end{equation}
\begin{equation}
    \left\langle\frac{\partial PU_{a,\lambda}}{\partial\lambda},\frac{\partial PU_{a,\lambda}}{\partial a_{i}}\right\rangle=O(\frac{1}{\lambda^{N-1}}),~~\left\langle\frac{\partial PU_{a,\lambda}}{\partial a_{i}},\frac{\partial PU_{a,\lambda}}{\partial a_{j}}\right\rangle=O(\frac{1}{\lambda^{N-1}}),\textrm{~where~} i\neq j,
\end{equation}
\begin{equation}
    \|\frac{\partial PU_{a,\lambda}}{\partial\lambda}\|_{H^{1}_{0}(\Omega)}^{2}=B\frac{1}{\lambda^{2}}+O(\frac{1}{\lambda^{N}}),~~\|{\frac{\partial PU_{a,\lambda}}{\partial a_{i}}}\|_{H^{1}_{0}(\Omega)}^{2}=C\lambda^{2}+O(\frac{1}{\lambda^{N-2}}), 
\end{equation}
\begin{equation}
    \|\frac{\partial^{2}PU_{a,\lambda}}{\partial\lambda^{2}}\|_{H^{1}_{0}(\Omega)}=O(\frac{1}{\lambda_{{\varepsilon}}^{2}}),~\|\frac{\partial^{2}PU_{a,\lambda}}{\partial\lambda\partial a_{i}}\||_{H^{1}_{0}(\Omega)}=O(1),~\|\frac{\partial^{2}PU_{a,\lambda}}{\partial a_{i}^{2}}\|_{H^{1}_{0}(\Omega)}=O(\lambda^{2}).
\end{equation}
\end{Lem}

\subsection{Green function and Robin function}
From the definition of Green function $G(x,y)$ and Robin function $R(x)$, we know that $G(x,y)$ and $H(x,y)$ are symmetric in $x$ and $y$. Moreover, by using the strong maximum principle, we have

 \begin{equation}
     0<G(x,y)< S(x,y), ~~x,y\in\Omega. 
 \end{equation}

\begin{Lem}\cite[Proposition 6.7.1]{Cao_Peng_Yan_2021}\label{estimate of robin function up to doundary}
Let $d=d(x,\partial\Omega)$ for $x\in \Omega$. Then as $d\to 0$, we have
\begin{equation}
R(x)=\frac{1}{(N-2)\omega_{N}}\frac{1}{(2d)^{N-2}}(1+O(d))    
\end{equation}
and 
\begin{equation}
    \nabla R(x)=\frac{2}{\omega_{N}}\frac{1}{(2d)^{N-1}}\frac{x^{\prime}-x}{d}+O\left(\frac{1}{d^{N-2}}\right),
    \end{equation}
where $x'\in\partial\Omega$ is the unique point, satisfying $d(x,\partial\Omega)=|x-x'|$.
\end{Lem}
\begin{Thm}$($\cite{Caffarelli1985Convexity,Cardaliaguet2002convexity}$)$\label{strictly convex of Robin function}
Let us assume that $\Omega$ is a convex domain of $\R^{N}$, $N\geq 3$. Then the Robin function $R$ is strictly convex and it has unique critical point which is a strict minimum.    
\end{Thm}

\begin{Thm}$($\cite{Grossi2010Nonexistence}$)$\label{nonexistence of multi-peak solution}
    Let $\Omega$ be a smooth bounded domain in $\R^{N}$, $N\geq 3$, and let $l\geq 2$ be an integer. Set $\Omega^{l}=\Omega\times\cdots \times \Omega$ ($l$ times) and $\Delta=\{(\xi_{1},\cdots,\xi_{l})\in\Omega^{l}|\xi_{i}=\xi_{j}\text{~for some~} i\neq j\}.$ For constants $A,B>0$ and $\Lambda=(\Lambda_{1},\cdots,\Lambda_{l})$, $\Lambda_{i}>0$, $1\leq i\leq l$, define a function $\mathcal{F}_{\Lambda}:\Omega^{l}\setminus \Delta\to \R,$ 
\begin{equation}
    \mathcal{F}_{\Lambda}(\xi_{1},\cdots,\xi_{l})=A\sum_{i=1}^{l}(R(\xi_{i})+K(\xi_{i}))\Lambda_{i}^{2}-B\sum_{\substack{i\neq j\\1\leq i,j\leq l}}G(\xi_{i},\xi_{j})\Lambda_{i}\Lambda_{j},
\end{equation}
    where $K\in C^{2}(\Omega)$ is such that $R+K$ is a convex function on $\Omega$.
    
    Assume $\Omega$ is convex, Then there does not exist any critical point $(a_{1},\cdots,a_{l})$ of $ \mathcal{F}_{\Lambda}$ in $\Omega\setminus \Delta.$ That is, there does not exist $(a_{1},\cdots,a_{l})\in\Omega\setminus \Delta$ such that
    \begin{equation}
        \frac{1}{2}A(\nabla R(a_{i})+\nabla K(a_{i}))\Lambda_{i}^{2}-B\sum_{j=1,j\neq i}^{l}\nabla_{x}G(a_{i},a_{j})\Lambda_{i}\Lambda_{j}=\vec{0},
    \end{equation}
    for $i=1,2,\cdots,l.$
\end{Thm}

\begin{Lem}\label{Lem derivate for reduce function}
Let $n\in\N$, $q\in(\max\{2,\frac{N}{N-2}\},2^*)$, $\vec{\lambda}=(\lambda_{1},\cdots,\lambda_{n})\in (0,\infty)^{n}$, $\vec{x}=(x_{1},\cdots,x_{n})$ $=(y_{1},\cdots,y_{nN})$ with $x_{i}\in\Omega \subset\R^{N}$, $y_{j}\in\R$ and $\Phi_{n}$ defined in \eqref{reduce function}.
\begin{enumerate}
    \item For any $i=1,\cdots,n$ and $j\in [(i-1)N+1,iN]\cap \N$, we have
    \begin{equation}
       \partial_{y_{j}}\Phi_{n}(\vec{x},\vec{\lambda})= \frac{A^{2}}{2}\left(\lambda_{i}^{N-2}\partial_{y_{j}}R(x_{i})-2\sum_{l\neq i,l=1}^{n}\lambda_{l}^{\frac{N-2}{2}}\lambda_{i}^{\frac{N-2}{2}}\partial_{y_{j}}G(x_{i},x_{l})\right).
    \end{equation}
    \item For any $i=1,\cdots,n$, we have
    \begin{equation}
    \begin{aligned}
      D_{\lambda_{i}}\Phi_{n}(\vec{x},\vec{\lambda})&=\frac{(N-2)A^{2}}{2\lambda_{i}}\left(\lambda_{i}^{N-2}R(x_{i})-\sum_{l\neq i,l=1}^{n}\lambda_{i}^{\frac{N-2}{2}}\lambda_{l}^{\frac{N-2}{2}}G(x_{i},x_{l})\right)\\
      &\quad -\frac{(N-2)(2^*-q)B}{2q\lambda_{i}}\lambda_{i}^{\frac{N-2}{2}(2^*-q)}.
    \end{aligned}
    \end{equation}
    \item For any $j\in[(i-1)N+1,iN]\cap\N$ for some $i=1,\cdots,n$, we have that if $s\in[(i-1)N+1,iN]\cap\N$
    \begin{equation}
        \partial^{2}_{y_{j},y_{s}}\Phi_{n}(\vec{x},\vec{\lambda})=\frac{A^{2}}{2}\left(\lambda_{i}^{N-2}\partial^{2}_{y_{j},y_{s}}R(x_{i})-2\sum_{l\neq i,l=1}^{n}\lambda_{i}^{\frac{N-2}{2}}\lambda_{l}^{\frac{N-2}{2}}\partial^{2}_{y_{j},y_{s}}G(x_{i},x_{l})\right),
    \end{equation}
    while if $s\in[(k-1)N+1,kN]\cap\N$ and $k\neq i$,
    \begin{equation}
        \partial^{2}_{y_{j},y_{s}}\Phi_{n}(\vec{x},\vec{\lambda})=-A^{2}\lambda_{k}^{\frac{N-2}{2}}\lambda_{i}^{\frac{N-2}{2}}\partial_{y_{j}}D_{y_{s}}G(x_{i},x_{k}).
    \end{equation}
    \item For $i=1,\cdots,n$ and $j=1,\cdots,nN$ we have that if $j\in[(i-1)N+1,iN]\cap\N$,
    \begin{equation}
        \partial_{y_{j}}D_{\lambda_{i}}\Phi_{n}(\vec{x},\vec{\lambda})=\frac{(N-2)A^{2}}{2}\lambda_{i}^{\frac{N-4}{2}}\left(\lambda_{i}^{\frac{N-2}{2}}\partial_{y_{j}}R(x_{i})-\sum_{l\neq i,l=1}^{n}\lambda_{l}^{\frac{N-2}{2}}\partial_{y_{j}}G(x_{i},x_{l})\right),
    \end{equation}
    while if $j\in[(k-1)N+1,kN]\cap\N$ and $k\neq i$
    \begin{equation}
        \partial_{y_{j}}D_{\lambda_{i}}\Phi_{n}(\vec{x},\vec{\lambda})=-\frac{(N-2)A^{2}}{2}\lambda_{i}^{\frac{N-4}{2}}\lambda_{k}^{\frac{N-2}{2}}D_{y_{j}}G(x_{i},x_{k}).
    \end{equation}
\end{enumerate}
where we use $\partial_{x_{i}}$ to denote the partial derivative for any function $f(x, y): \R^N \times \R^N \to \R$ with respect to $x$ and use $D_{y_{j}}$ to denote the partial derivative with respect to $y$.
\end{Lem}
\begin{proof}
    We can prove this Lemma by direct calculation and we omit the details.
\end{proof}

\subsection{Poho\v{z}aev identity}
\label{Pohozaevidentities}
In the following, we derive several forms of the Poho\v{z}aev identity as in \cite{Pohozaev1965,YanYanLi1995} and show some important properties.

Assume that $u $ is a $C^2$ solution of the following equation
 \begin{equation}
     -\Delta u = u^{2^*-1}+\varepsilon u^{q-1}, \hbox{~~in~} \Omega.
 \end{equation}
Multiplying this equation by $\langle x-x_{0},\nabla u \rangle$ and $\frac{\partial u}{\partial x_{i}}$ on both sides and integrating by parts respectively, we have

\begin{equation}\label{pohozaev identity 1}
\begin{split}
-&\int_{\partial \Omega}\frac{\partial u}{\partial\nu}
\big\langle x-x_{0},\nabla u\big\rangle
+\frac{1}{2}\int_{\partial \Omega}
|\nabla  u|^2
\big\langle x-x_{0},\nu\big\rangle
-\frac{N-2}{2}\int_{\partial \Omega}\frac{\partial u}{\partial\nu}
  u \\&=
  \frac{1}{2^*}\int_{\partial \Omega} u^{2^*} \big\langle x-x_{0},\nu\big\rangle +\frac{\varepsilon}{q
  }\int_{\partial \Omega}u^{q}\big\langle x-x_{0},\nu\big\rangle-\varepsilon\frac{2N-(N-2)q}{2q}\int_{ \Omega} u^q,
\end{split}
\end{equation}
and
\begin{equation}\label{pohozaev identity 2}
-\int_{\partial \Omega}\frac{\partial u}{\partial \nu}\frac{\partial u}{\partial x_i}
+\frac{1}{2}\int_{\partial \Omega}|\nabla u_\varepsilon|^2\nu_i=\frac{1}{2^*}\int_{\partial \Omega}  u^{2^*}\nu_i+ \frac{\varepsilon}{q}\int_{\partial \Omega} u_\varepsilon^q\nu_i,
\end{equation}
where $\nu(x)=\big(\nu_{1}(x),\cdots,\nu_N(x)\big)$ is the outward unit normal of $\partial \Omega$. In particularly, when $u$ satisfies that
\begin{equation}
    u=0,\text{~~on~~}\partial\Omega,
\end{equation}
from \eqref{pohozaev identity 1} we can obtain the following classical Poho\v{z}aev identity
\begin{equation}\label{pohozaev identity 3}
 \frac{1}{2N}\int_{\partial \Omega}
|\nabla  u|^2
\big\langle x-x_{0},\nu\big\rangle=\left(\frac{1}{q}-\frac{1}{2^{*}}\right)\varepsilon\int_{\Omega}u^{q}.
\end{equation} 

For any $\sigma>0$ and $x\in\partial B(0,\sigma)$, we define
\begin{equation}\label{definition of B-x-sigma-u}
    B(x,\sigma,u,\nabla u)=\sigma\left(\frac{\partial u}{\partial\nu}
\right)^{2}-\frac{\sigma}{2}
|\nabla  u|^2+\frac{N-2}{2}\frac{\partial u}{\partial\nu}u .
\end{equation}
Let us recall the properties of $B(x,\sigma,u,\nabla u)$ from \cite[Proposition 1.1]{YanYanLi1995}.
\begin{Lem}\label{lem properties of B}
    \begin{enumerate}
        \item If $u(x)=|x|^{2-N}$, then
        \begin{equation}
            B(x,\sigma,u,\nabla u)=0,\text{~~for any~~}x\in\partial B(0,\sigma).
        \end{equation}
        \item If $u(x)=a|x|^{2-N}+A+\alpha(x)$, where $a, A>0$ are two positive constants and $\alpha(x)$ is some function differentiable near the origin with $\alpha(0)=0$. Then there exists some positive constant $\sigma^*>0$ such that for any $0<\sigma<\sigma^*$, we have
        \begin{equation}
            B(x,\sigma,u,\nabla u)<0,\text{~~for any~~}x\in\partial B(0,\sigma),
        \end{equation}
        and
        \begin{equation}
            \lim_{\sigma\to 0}\int_{\partial B(0,\sigma)}B(x,\sigma,u,\nabla u)=-\frac{(N-2)^{2}}{2}aA\omega_{N}.
        \end{equation}
        where $\omega_{N}$ denotes the measure of the unit sphere in $\R^{N}$.
    \end{enumerate}
\end{Lem}

For any fixed point $x_{*}\in\Omega$, we define the following two quadratic forms
\begin{equation}\label{defin for quadratic P}
\begin{split}
P(u,v,\theta)=&- \theta\int_{\partial B_\theta(x_{*})}
\big\langle \nabla u ,\nu\big\rangle
\big\langle \nabla v,\nu\big\rangle
+\frac{\theta}{2}\int_{\partial B_\theta(x_{*})}
\big\langle \nabla u , \nabla v \big\rangle
\\&
+\frac{2-N}{4}\int_{\partial B_\theta(x_{*})}
\big\langle \nabla u ,  \nu \big\rangle v
+\frac{2-N}{4}\int_{\partial B_\theta(x_{*})}
\big\langle \nabla v ,  \nu \big\rangle u.
\end{split}
\end{equation}
and
\begin{equation}\label{defin for quadratic Q}
Q(u,v,\theta)=-\int_{\partial B_\theta(x_{*})}\frac{\partial v}{\partial \nu}\frac{\partial u}{\partial x_i}-
\int_{\partial B_\theta(x_{*})}\frac{\partial u}{\partial \nu}\frac{\partial v}{\partial x_i}
+\int_{\partial B_\theta(x_{*})}\big\langle \nabla u,\nabla v \big\rangle \nu_i,
\end{equation}
where $u,v\in C^{2}(\Omega)$ and $\theta>0$ such that $B_{2\theta}(x_{*})\subset\Omega$.
\begin{Lem}\cite[Lemma A.5]{chen2024}\label{propertity of P Q}
    If $u$ and $v$ are harmonic in $ B_\theta(x_{*})\backslash \{x_{*}\}$, then $P(u,v,\theta)$ and $Q(u,v,\theta)$ are independent of $\theta$.
\end{Lem}
\begin{Lem}\label{lem for estimate of quadractic form}\cite[Section 5]{Cao2021Trans}
For $i,h=1,\cdots,N$, we have the following computations about $P,Q$ and Green function.
\begin{equation}
P\Big(G(y_{*},x), G(z_{*},x)\Big)=
\begin{cases}
-\frac{(N-2)R(x_{*})}{2},~&\mbox{if}~y_{*}=x_{*},z_{*}=x_{*},\\[1mm]
\frac{(N-2)G(x_{*},z_{*})}{4},~
&\mbox{if}~y_{*}=x_{*},z_{*}\neq x_{*},\\[1mm]
\frac{(N-2)G(x_{*},y_{*})}{4},~&\mbox{if}~y_{*}\neq x_{*},z_{*}= x_{*},\\[1mm]
0, ~&\mbox{if}~y_{*}\neq x_{*},z_{*}\neq x_{*}.
\end{cases}
\end{equation}
\begin{equation}
P\Big(G(y_{*},x),\partial_hG(z_{*},x)\Big)=
\begin{cases}
-\frac{(N-1)\partial_h R(x_{*})}{4},~&\mbox{if}~y_{*}=x_{*},z_{*}=x_{*},\\[1mm]
\frac{(N-2)\partial_hG(z_{*},x_{*})}{4},~
&\mbox{if}~y_{*}=x_{*},z_{*}\neq x_{*},\\[1mm]
\frac{N\partial_h G(x_{*},y_{*})}{4},~&\mbox{if}~y_{*}\neq x_{*},z_{*}= x_{*},\\[1mm]
0, ~&\mbox{if}~y_{*}\neq x_{*},z_{*}\neq x_{*}.
\end{cases}
\end{equation}

\begin{equation}
Q\Big(G(y_{*},x),G(z_{*},x)\Big)=
\begin{cases}
-\partial_{i} R(x_{*}),~&\mbox{if}~y_{*}=x_{*},z_{*}=x_{*},\\[1mm]
D_{i}G(z_{*},x_{*}),
 ~
&\mbox{if}~y_{*}=x_{*},z_{*}\neq x_{*},\\[1mm]
D_{i}G(y_{*},x_{*})~&\mbox{if}~y_{*}\neq x_{*},z_{*}= x_{*},\\[1mm]
0, ~&\mbox{if}~y_{*}\neq x_{*},z_{*}\neq x_{*}.
\end{cases}
\end{equation}
\begin{equation}
Q\Big(G(y_{*},x),\partial_h G(z_{*},x)\Big)=
\begin{cases}
-\frac{1}{2}\partial^{2}_{i,h} R(x_{*}),~&\mbox{if}~y_{*}=x_{*},z_{*}=x_{*},\\[1mm]
D_{i}\partial_{h}G(z_{*},x_{*}),
 ~
&\mbox{if}~y_{*}=x_{*},z_{*}\neq x_{*},\\[1mm]
\partial^{2}_{i,h}G(x_{*},y_{*})~&\mbox{if}~y_{*}\neq x_{*},z_{*}= x_{*},\\[1mm]
0, ~&\mbox{if}~y_{*}\neq x_{*},z_{*}\neq x_{*}.
\end{cases}
\end{equation}
Here we use $\partial_{i}$ to denote the partial derivative for any function $f(x, y): \R^N \times \R^N \to \R$ with respect to $x$ and use $D_{i}$ to denote the partial derivative with respect to $y$.
\end{Lem}

\subsection{Kernel of Linear operator}
\begin{Lem}\cite[Lemma A.1]{Bianchi1991ANO}\label{Kernel of Emden-Fowler equation}
If $u$ is a solution of the following equation
\begin{equation}
\begin{cases}
    -\Delta u=(2^*-1)U_{0,1}^{2^*-2}u, \textrm{~~in~~}\R^{N},\\
    u\in\mathcal{D}^{1,2}(\R^{N}),
\end{cases}    
\end{equation}   
then there exists $a_{i}\in\R$, $i=0,1,\cdots,N$ such that
\begin{equation}
u(x)=a_{0}\frac{1-|x|^{2}}{(1+|x|^{2})^{N/2}}+\sum_{i=1}^{N}a_{i}\frac{x_{i}}{(1+|x|^{2})^{N/2}}.
\end{equation}
\end{Lem}

\section{Local estimates around the blow-up points}\label{Local estimates for blow-up points and proof of Theorem}
In the following, we always assume that $N\geq3$, $q\in(\max\{2,\frac{4}{N-2}\},2^*)$, $\Omega$ is a smooth and bounded domain in $\R^{N}$. We follow \cite{YanYanLi1995} and \cite{Cerqueti2001Localestimates} with refinement.

\subsection{Various types of blow-up points}
We present the definitions of blow up point, isolated blow up point and isolated simple blow up point which were introduced by Schoen.

\begin{Def}
    Let $\{\varepsilon_{n}\}_{n\in\N}$ be a sequence of positive small numbers such that $\varepsilon_{n}\to0$ as $n\to\infty$ and $\{u_{n}=u_{\varepsilon_{n}}\}_{n\in\N}$ be a sequence of solutions of \eqref{p-varepsion}.
    \begin{enumerate}
        \item A point $\bar{x}\in \Omega$ is called a blow-up point of functions $\{u_{n}\}_{n\in\N}$, if there exists a sequence of points $\{x_{n}\}\subset\Omega$ such that $x_{n}$ is a local maximum point of $u_{n}$, 
    \begin{equation}
        x_{n}\to \bar{x} \text{~~and~~} u_{n}(x_{n})\to\infty, \text{~as~}n\to\infty.
    \end{equation}
    \item A point $\bar{x}\in \Omega$ is called an isolated blow-up point of functions $\{u_{n}\}_{n\in\N}$, if $\bar{x}$ is a blow-up point and there exist numbers $0<\bar{r}<\text{dist}(\bar{x},\partial\Omega)$ small and $\bar{C}>0$ large (independent of $n\in\N$) such that for large $n$
    \begin{equation}
        u_{n}(x)\leq \bar{C}|x-x_{n}|^{-\frac{N-2}{2}}, ~\text{for all}~ x\in B(x_{n},\bar{r}).
    \end{equation}
    \item 
Let $\bar{x}\in \Omega$ be an isolated blow-up point of functions $\{u_{n}\}_{n\in\N}$. We define the spherical average of $u_{n}$ for $r\in(0,\bar{r})$
\begin{equation}
        \bar{u}_{n}(r)=\frac{1}{|\partial B(x_{n},r)|}\int_{\partial B(x_{n},r)}u_{n}~\text{and}~\bar{w}_{n}(r)=r^{\frac{N-2}{2}}\bar{u}_{n}(r).
    \end{equation}
Then the point $\bar{x}\in \Omega$ is called an isolated simple blow-up point of functions $\{u_{n}\}_{n\in\N}$, if $\bar{w}_{n}$ has precisely one critical point in $(0,\bar{r})$ for large $n$, after the value of $\bar{r}>0$ is suitably reduced.
    \end{enumerate}
\end{Def}

\subsection{Isolated blow-up point} First, we give some property of the isolated blow-up point and the proof is contained in  \cite[Section 2.1]{Cerqueti2001Localestimates} up to a minor modifications.

\begin{Lem}\label{harnack inequality}
Let $x_{n}\to\bar{x}$ be an isolated blow-up point of $\{u_{n}\}_{n\in\N}$. Then there exists $C>0$ independent of $n\in\N$ such that  
\begin{equation}\label{harnack inequality 1}
   \max_{B(x_{n},2r)\setminus B(x_n,\frac{1}{2}r)} u_{n}(x)\leq C\min_{B(x_{n},2r)\setminus B(x_n,\frac{1}{2}r)} u_{n}(x).
   \end{equation}
\end{Lem}

\begin{Cor}\label{cor harnack inequality}
  Let $x_{n}\to\bar{x}$ be an isolated blow-up point of $\{u_{n}\}_{n\in\N}$. Then there exists $C>0$ independent of $n\in\N$ such that  
\begin{equation}\label{harnack inequality 2}
    \max_{\partial B(x_{n},r)} u_{n}(x)\leq C\min_{\partial B(x_{n},r)} u_{n}(x),
\end{equation}
for all $r\in(0,\bar{r})$ and $n\in\N$.  
\end{Cor}

The following Proposition shows that a renormalization of $u_{n}$ tends to the standard bubble $U$ around an isolated blow-up point.
\begin{Prop}\label{lem isolated blow up point 2}
Let $x_{n}\to\bar{x}$ be an isolated blow-up point of $\{u_{n}\}_{n\in\N}$. Then for any couple of sequences $\{R_{n}\}_{n\in\N}$ and $\{\eta_{n}\}_{n\in\N}$ satisfying that $R_{n}\to\infty$ and $\eta_{n}\to 0$ as $n\to\infty$, after possibly passing to a subsequence, we have

\begin{enumerate}
   \item $\|u_{n}^{-1}(x_{n})u_{n}(u_{n}^{-\frac{2}{N-2}}(x_{n})\cdot+x_{n})-U\|_{C^{2}(\bar{B}(0,2R_{n}))\cap H^{1}(\bar{B}(0,2R_{n}))} \leq \eta_{n}$,
   \item $R_{n}u_{n}^{-\frac{2}{N-2}}(x_{n})\to 0$, as $n\to\infty.$
\end{enumerate}
\end{Prop}

In the following, we fix first the sequence $R_{n}\to\infty$ and then fix the sequence $\eta_{n}\to0$ such that 
\begin{equation}
  0<\eta_{n}<\frac{1}{2}\left(\frac{N(N-2)}{(N(N-2)+R_{n}^{2}}\right)^{\frac{N-2}{2}}.
\end{equation}
 Moreover, we define
\begin{equation}\label{e217}
    r_{n}:=R_{n}u_{n}^{-\frac{2}{N-2}}(x_{n})\to 0,\text{~~as~~}n\to\infty.
\end{equation}

\begin{Cor}\label{cor below estimate}
  Let $x_{n}\to\bar{x}$ be an isolated blow-up point of $\{u_{n}\}_{n\in\N}$. Then we have
 \begin{enumerate}
     \item $\bar{w}_{n}$ has an unique critical $\bar{r}_{n}$ in $(0,r_{n})$. Moreover, if $\bar{x}$ is a isolated simple blow-up point, then $\bar{w}_{n}$ is strictly decreasing for $r\in (\bar{r}_{n},\bar{r})$.  
     \item There exists a constant $C>0$ independent of $n$ such that
     \begin{enumerate}
         \item \begin{equation}
              u_{n}(x)\leq C\frac{1}{u_{n}(x_{n})|x-x_{n}|^{N-2}},\text{~~for any ~~} x\in \bar{B}(x_{n},r_{n}).
         \end{equation}
         \item \begin{equation}
             \quad u_{n}(x)\geq C\frac{u_{n}(x_{n})}{((N(N-2)+u_{n}^{\frac{4}{N-2}}(x_{n})|x-x_{n}|^{2})^{\frac{N-2}{2}}},\text{~~for any ~~} x\in \bar{B}(x_{n},r_{n}).
         \end{equation}
     \end{enumerate}
 \end{enumerate}
\end{Cor}

Next, we can extend the lower bound estimate $(b)$ in Corollary \ref{cor below estimate} to a fixed ball independent of $n$.

\begin{Prop}\label{prop lower bounded of solutions}
Let $x_{n}\to\bar{x}$ be an isolated blow-up point of $\{u_{n}\}_{n\in\N}$. Then there exists numbers $C>0$ and $0<\bar{r}_{1}<\bar{r}$ which are independent of $n$ such that
\begin{equation}
     u_{n}(x)\geq C\frac{u_{n}(x_{n})}{(N(N-2)+u_{n}^{\frac{4}{N-2}}(x_{n})|x-x_{n}|^{2})^{\frac{N-2}{2}}},\text{~~for any ~~} x\in \bar{B}(x_n,\bar{r}_{1}).
\end{equation}
\end{Prop}

\subsection{Isolated simple blow-up point}
The main results in this subsection state as follows.
\begin{Prop}\label{prop properties of isolated simple blow-up point}
Let $x_{n}\to\bar{x}$ be an isolated simple blow-up point of $\{u_{n}\}_{n\in\N}$. Then there exists constants $C=C(N,q,\bar{r},\bar{C})>0$ such that
\begin{equation}\label{eq upper bound of isolated simple blow-up point}
     u_{n}(x)\leq C\frac{1}{u_{n}(x_{n})|x-x_{n}|^{N-2}},\text{~~for any ~~} x\in \bar{B}(x_{n},\bar{r}/{2}),
\end{equation}
Furthermore, after passing to a subsequence, we have
\begin{equation}\label{eq convergence of isolated simple blow-up point}
     u_{n}(x_{n})u_{n}(x)\to\frac{\alpha_{N}^{2}}{|x-\bar{x}|^{N-2}}+b(x),\text{~~for any ~~} x\in C^{2}_{loc}({B}(\bar{x},\bar{r}/2)),
\end{equation}
where $b(x)$ is a harmonic function in ${B}(\bar{x},\bar{r}/2)$.
\end{Prop}

Next, we will establish this proposition through a series of lemmas. First, similar as the proof of \cite[Lemma 3.3, Lemma 3.4]{Cerqueti2001Localestimates}, we can obtain the following two lemmas.
\begin{Lem}\label{lem decay estimate}
Let $x_{n}\to\bar{x}$ be an isolated simple blow-up point of $\{u_{n}\}_{n\in\N}$. Then for any given $r_{1}\in(0,\bar{r})$ and $\delta>0$ small, there exists a positive constant $C_{2}=C_{2}(N,\bar{r},\bar{C})>0$ such that for $n$ large enough
\begin{equation}\label{lem decay estimate-1}
    u_{n}(x)\leq \frac{C_{2}}{u_{n}^{\lambda}(x_{n})|x-x_{n}|^{N-2-\delta}},\text{~~for any~~}r_{n}\leq |x-x_{n}|\leq r_{1},
\end{equation}
where $r_n$ is defined in \eqref{e217} and $\lambda=1-\frac{2\delta}{N-2}.$
\end{Lem}

\begin{Lem}\label{lem properties of isolated simple point-2}
Let $x_{n}\to\bar{x}$ be an isolated simple blow-up point of $\{u_{n}\}_{n\in\N}$ and $e\in\R^{N}$ such that $|e|=\frac{1}{2}\bar{r}$. Then there exist a constant $a_{1}>0$ and a harmonic function $b_{1}(x)$ in $B(\bar{x},\bar{r})$ such that
\begin{equation}
    \frac{u_{n}(x)}{u_{n}(x_{n}+e)}\to \frac{a_{1}}{|x-\bar{x}|^{N-2}}+b_{1}(x),\text{~in~}C^{2}_{loc}(B(\bar{x},\bar{r})\setminus\{\bar{x}\}).
\end{equation}
\end{Lem}

\begin{Lem}\label{lem blou-up rate estimate}
 Let $x_{n}\to\bar{x}$ be an isolated simple blow-up point of $\{u_{n}\}_{n\in\N}$ and $e\in\R^{N}$ such that $|e|=\frac{1}{2}\bar{r}$. Then there exist a constant $C=C(N,q,\bar{C},\bar{r})>0$ such that   
 \begin{equation}
     \varepsilon_{n}u_{n}(x_{n})^{q+2-2^*}\leq Cu_{n}(x_{n})^{2}u_{n}(x_{n}+e)^{2}+o(1).
 \end{equation}
\end{Lem}
\begin{proof}
Recall the Local Poho\v{z}aev identity of $u_{n}$ in $B(\bar{x},\frac{1}{2}\bar{r})$ (see \eqref{pohozaev identity 1}),
\begin{equation}\label{lem blou-up rate estimate-proof-1}
\begin{split}
\frac{\bar{r}}{2}&\int_{\partial B(\bar{x},\frac{1}{2}\bar{r})}\left(\frac{\partial u_{n}}{\partial\nu}\right)^{2}
-\frac{\bar{r}}{4}\int_{\partial B(\bar{x},\frac{1}{2}\bar{r})}
|\nabla  u_{n}|^2
+\frac{N-2}{2}\int_{\partial B(\bar{x},\frac{1}{2}\bar{r})}\frac{\partial u_{n}}{\partial\nu}u_{n} \\
&=\varepsilon_{n}\frac{2N-(N-2)q}{2q}\int_{ B(\bar{x},\frac{1}{2}\bar{r})} u_{n}^q
-\frac{\bar{r}}{22^*}\int_{\partial B(\bar{x},\frac{1}{2}\bar{r})} u_{n}^{2^*} -\frac{\varepsilon_{n}\bar{r}}{2q
  }\int_{\partial  B(\bar{x},\frac{1}{2}\bar{r})}u_{n}^{q}.
\end{split}
\end{equation}

Now, we consider a annular domain $D=\{\sigma_{1}<|x-x_{n}|<\sigma_{2}\}$ with $0<\sigma_{1}<\frac{1}{2}\bar{r}<\sigma_{2}<\bar{r}$, then by Gradient estimate \cite[Theorem 3.9]{GilbargTrudinger} we have
\begin{equation}
    \sup_{|x-x_{n}|=\frac{1}{2}\bar{r}}|\nabla u_{n}|\leq C(N,\bar{r})\sup_{D}(u_{n}+u_{n}^{2^*-1}+\varepsilon_{n}u_{n}^{q-1}),
\end{equation}
and from Lemma \ref{harnack inequality} and Lemma \ref{lem decay estimate}
\begin{equation}
    \sup_{D}u_{n}\leq C(N,\bar{C})\inf_{D}u_{n}\leq C(N,\bar{C})u_{n}(x_{n}+e)\to0, \text{~as~}n\to\infty.
\end{equation}
This implies that
\begin{equation}
    \sup_{|x-x_{n}|=\frac{1}{2}\bar{r}}|\nabla u_{n}|\leq C(N,\bar{C},\bar{r})u_{n}(x_{n}+e).
\end{equation}
Thus the left-hand side of \eqref{lem blou-up rate estimate-proof-1} can be estimated as 
\begin{equation}\label{lem blou-up rate estimate-proof-5}
    \mathrm{LHS}\leq C(N,\bar{C},\bar{r})u_n(x_n+e)^2.
\end{equation}

Next, we estimate each term in the right-hand side of \eqref{lem blou-up rate estimate-proof-1}
\begin{equation}
\begin{aligned}\mathrm{J}_{1}=\varepsilon_{n}\frac{2N-(N-2)q}{2q}\int_{B(\bar{x},\frac{1}{2}\bar{r})}u_{n}^{q}&=\varepsilon_{{n}}\frac{2N-(N-2)q}{2q}\int_{|x-x_{n}|<r_{n}}u_{n}^{q}(x)\\
&+\varepsilon_{n}\frac{2N-(N-2)q}{2q}\int_{\{r_{n}\leq|x-x_{n}|<\frac{1}{2}\bar{r}\}\cap B(\bar{x},\frac{1}{2}\bar{r})}u_{n}^{q}(x)\\
&\geq\varepsilon_{{n}}\frac{2N-(N-2)q}{2q}\int_{|x-x_{n}|<r_{n}}u_{n}^{q}(x).
\end{aligned}
\end{equation}
By Proposition \ref{lem isolated blow up point 2},
\begin{equation}
\begin{aligned}
   &\varepsilon_{{n}}\frac{2N-(N-2)q}{2q}\int_{|x-x_{n}|<r_{n}}u_{n}^{q}(x)\\
   &=\varepsilon_{{n}}\frac{2N-(N-2)q}{2q}\frac{u_n(x_n)^{q+2-2^*}}{u_n(x_n)^2}\left(\int_{\mathbf{R}^N}\left(\frac{N(N-2)}{N(N-2)+|x|^2}\right)^{\frac{N-2}{2}q}+o(1)\right). 
\end{aligned}   
\end{equation}
Note that $q>\frac{N}{N-2}$, thus there exist a constant $C=C(N,q)>0$ such that
\begin{equation}
   \mathrm{J}_1\geq C(N,q)\varepsilon_{{n}}\frac{u_n(x_n)^{q+2-2^*}}{u_n(x_n)^2}. 
\end{equation}
Moreover, from Lemma \ref{lem decay estimate}, the last two estimates can be estimated as
\begin{equation}
    \mathrm J_2=\frac{\bar{r}}{22^*}\int_{\partial B(\bar{x},\frac{1}{2}\bar{r})}u_n^{2^*}\leq\frac{o(1)}{u_n(x_n)^2},
\end{equation}
\begin{equation}
    \mathrm{J}_3=\frac{\varepsilon_{n}\bar{r}}{2q
  }\int_{\partial B(\bar{x},\frac{1}{2}\bar{r})}u_n^q\leq\varepsilon_{n}\frac{u_n(x_n)^{q+2-2^*}}{u_n(x_n)^2}o(1).
\end{equation}
Finally, the right-hand side of \ref{lem blou-up rate estimate-proof-1} 
\begin{equation}\label{lem blou-up rate estimate-proof-12}
\begin{aligned}
     \textrm{RHS}&\geq \text{J}_{1}-\text{J}_{2}-\text{J}_{3}\\
     &\geq C(N,q)\varepsilon_{{n}}\frac{u_n(x_n)^{q+2-2^*}}{u_n(x_n)^2}-\frac{o(1)}{u_n(x_n)^2}-\varepsilon_{n}\frac{u_n(x_n)^{q+2-2^*}}{u_n(x_n)^2}o(1)\\
     &=\varepsilon_{n}\frac{u_n(x_n)^{q+2-2^*}}{u_n(x_n)^2}(C(N,q)-o(1))-\frac{o(1)}{u_n(x_n)^2}.
\end{aligned}
\end{equation}
Now, combining \eqref{lem blou-up rate estimate-proof-1}, \eqref{lem blou-up rate estimate-proof-5} and \eqref{lem blou-up rate estimate-proof-12} together, we can obtain that
\begin{equation}
     \varepsilon_{n}u_{n}(x_{n})^{q+2-2^*}\leq C(N,q,\bar{r},\bar{C})u_{n}(x_{n})^{2}u_{n}(x_{n}+e)^{2}+o(1).
\end{equation}
This complete the proof.
\end{proof}

\begin{Lem}\label{lem properties of isolated simple point 3}
   Let $x_{n}\to\bar{x}$ be an isolated simple blow-up point of $\{u_{n}\}_{n\in\N}$ and $e\in\R^{N}$ such that $|e|=\frac{1}{2}\bar{r}$. Then there exist a constant $C=C(N,q,\bar{C},\bar{r})>0$ such that   
 \begin{equation}
    u_{n}(x_{n})u_{n}(x_{n}+e)\leq C.
 \end{equation}  
\end{Lem}

\begin{proof}
By contradiction, suppose that, up to a subsequence,
\begin{equation}
    \lim_{n\to\infty}u_n(x_n+e)u_n(x_n)=\infty.
\end{equation}
Multiplying the equation \eqref{p-varepsion} by $u_n(x_n+e)^{-1}$ and integrating by parts on $B(\bar{x},\frac{1}{2}\bar{r})$, we have
\begin{equation}\label{lem properties of isolated simple point 3-proof-1}
    -\int_{\partial B(\bar{x},\frac{1}{2}\bar{r})}\frac{\partial v_{n}}{\partial\nu}=u_{n}(x_{n}+e)^{-1}\int_{B(\bar{x},\frac{1}{2}\bar{r})}u_{n}^{2^*-1}+u_{n}(x_{n}+e)^{-1}\int_{B(\bar{x},\frac{1}{2}\bar{r})}\varepsilon_{n}u_{n}^{q-1},
\end{equation}
where $v_{n}=\frac{u_{n}}{u_{n}(x_{n}+e)}$. From Lemma \ref{lem properties of isolated simple point-2} and  dominated convergence theorem, the left-hand side of \eqref{lem properties of isolated simple point 3-proof-1} satisfies
\begin{equation}\label{lem properties of isolated simple point 3-proof-3}
   \lim_{n\to\infty}\mathrm{LHS}=\lim_{n\to\infty}-\int_{\partial B(\bar{x},\frac{1}{2}\bar{r})}\frac{\partial v_{n}}{\partial\nu}=-a_{1}(2-N)\omega_{N}>0. 
\end{equation}
Next, we consider the right-hand side of \eqref{lem properties of isolated simple point 3-proof-1}
\begin{equation}\label{lem properties of isolated simple point 3-proof-4}
    \begin{aligned}\mathrm{RHS}&=u_{n}(x_{n}+e)^{-1}\left(\int_{|x-x_{n}|\leq r_{n}}u_{n}^{2^*-1}+\int_{\{r_{n}\leq|x-x_{n}|\leq\frac{1}{2}\bar{r}\}\cap B(\bar{x},\frac{1}{2}\bar{r})}u_{n}^{2^*-1}\right)\\
    &+u_{n}(x_{n}+e)^{-1}\left(\int_{|x-x_{n}|\leq r_{n}}\varepsilon_{n}u_{n}^{q-1}+\int_{\{r_{n}\leq|x-x_{n}|\leq\frac{1}{2}\bar{r}\}\cap B(\bar{x},\frac{1}{2}\bar{r})}\varepsilon_{n}u_{n}^{q-1}\right).
    \end{aligned}
\end{equation}
Let us estimate each term in \eqref{lem properties of isolated simple point 3-proof-4}. First
\begin{equation}
    \begin{aligned}\mathrm{I}_{1}&=\int_{|x-x_{n}|\leq r_{n}}u_{n}^{2^*-1}\\&=\int_{|y|\leq R_{n}}u_{n}^{2^*-1}\left(\frac{y}{u_{i}(x_{n})^{\frac{2}{N-2}}}+x_{n}\right)\frac{dy}{u_{n}(x_{n})^{2^*}}\\&:=\frac{1}{u_{n}(x_{n})}\int_{|y|<R_{n}}\tilde{u}_{n}^{2^*-1}(y)dy.\end{aligned}
\end{equation}
By the Proposition \ref{lem isolated blow up point 2}, we have
\begin{equation}
\int_{|y|<R_{n}}\tilde{u}_{n}^{2^*-1}(y)\to\int_{\mathbf{R}^{N}}U^{2^*-1}=\frac{\alpha_{N}^{2^*}\omega_{N}}{N}.
\end{equation}
Therefore,
\begin{equation}
    \mathrm I_1=\frac{\alpha_{N}^{2^*}\omega_{N}}{Nu_n(x_n)}+o(1).
\end{equation}
By Lemma \ref{lem decay estimate} we have
\begin{equation}
    \begin{aligned}\mathrm{I}_{2}=\int_{r_{n}\leq|x-x_{n}|\leq\frac{1}{2}\bar{r}}u_{n}^{2^*-1}&\leq\frac{C(N,\bar{r},\bar{C})}{2-\delta(2^*-1)}\frac{1}{u_{n}(x_{n})^{\lambda(2^*-1)}}(r_{n}^{\delta(2^*-1)-2}-(\frac{1}{2}\bar{r})^{\delta(2^*-1)-2})\\
    &\leq \frac{C(N,\bar{r},\bar{C})}{R_{n}^{2-\delta (2^*-1)}}\frac{1}{u_{n}(x_{n})}=\frac{o(1)}{u_{n}(x_{n})},\end{aligned}
\end{equation}
for $\delta>0$ small enough and $n$ large enough.
From Corollary \ref{cor below estimate}, we have
\begin{equation}
   \begin{aligned}\mathrm{I}_{3}&=\:\varepsilon_{n}\int_{|x-x_{n}|\leq r_{n}}u_{n}^{q-1}\\
   &=\varepsilon_{n}\int_{|x|<R_{n}}\left(\frac{1}{u_{n}(x_{n})}u_{n}\left(\frac{y}{u_{n}(x_{n})^{\frac{2}{N-2}}}+x_{n}\right)\right)^{q-1}\frac{dy}{u_{n}(x_{n})^{2^*-q+1}}\\
   &\leq C(N,q)\frac{\varepsilon_{n}}{u_{n}(x_{n})^{2^*-q+1}}\left(\int_{|y|\leq R_{n}}\frac{dy}{(1+|y|^{2})^{\frac{N-2}{2}(q-1)}}+o(1)\right).\end{aligned} 
\end{equation}
Note that, by direct computation we have
\begin{equation}
        \int_{|y|\leq R_n}\frac{dy}{\left(1+|y|^2\right)^{\frac{N-2}{2}(q-1)}}=\begin{cases}O(1),&\text{~if~}(N-2)(q-1)>N,\\
        O(\ln R_{n}),&\text{~if~}(N-2)(q-1)=N,\\
        O(R_{n}^{N-(N-2)(q-1)}),&\text{~if~}(N-2)(q-1)<N.
        \end{cases}
\end{equation}
Hence, by $R_{n}u_{n}(x_{n})^{-\frac{2}{N-2}}\to 0$ as $n\to\infty$  and the above estimates 
\begin{equation}
    \mathrm I_3=\frac{o(1)}{u_n(x_n)}.
\end{equation}
Finally, by Lemma \ref{lem decay estimate}, we have
\begin{equation}
\begin{aligned}
    \mathrm I_4=\varepsilon_n\int_{r_n\le|x-x_n|\le\frac{1}{2}\bar{r}}u_n^{q-1}&\leq\varepsilon_{n}\frac{C(N,\bar{r},\bar{C})}{u_{n}(x_{n})^{\lambda(q-1)}}\int_{r_n\le|x-x_n|\le\frac{1}{2}\bar{r}}\frac{1}{|x-x_{n}|^{(N-2-\delta)(q-1)}}\\
    &\leq \begin{cases}
        \varepsilon_{n}\frac{C(N,\bar{r},\bar{C})}{u_{n}(x_{n})^{\lambda(q-1)}},&\text{~if~}(N-2)(q-1)\leq N,\\
        \varepsilon_{n}R_{n}^{N-(N-2)(q-1)+\delta(q-1)}\frac{C(N,\bar{r},\bar{C})}{u_{n}(x_{n})^{2^*-q+1}},&\text{~if~}(N-2)(q-1)> N,\\
    \end{cases}\\
    &=\frac{o(1)}{u_{n}(x_{n})}.
\end{aligned}
\end{equation}
for any $\delta>0$ small enough and $n$ large enough.
Therefore, by \eqref{lem properties of isolated simple point 3-proof-4}, we obtain that
\begin{equation}
    \mathrm{RHS}\to0,\quad n\to\infty.
\end{equation}
This is make a contradiction with \eqref{lem properties of isolated simple point 3-proof-3} and thus we complete the proof.
\end{proof}

\begin{Cor}
  Let $x_{n}\to\bar{x}$ be an isolated simple blow-up point of $\{u_{n}\}_{n\in\N}$. Then there exists a positive constant $C=C(N,q,\bar{C},\bar{r})$ such that
    \begin{equation}
        \varepsilon_{n}u_{n}(x_{n})^{q+2-2^*}\leq C.
    \end{equation}
\end{Cor}

\begin{proof}[Proof of Proposition \ref{prop properties of isolated simple blow-up point}]
We remark that from Corollary \ref{cor below estimate} and Lemma \ref{lem properties of isolated simple point 3} the assertion \eqref{eq upper bound of isolated simple blow-up point} is already proved for $|x-x_n|\leq r_n$ and for $|x-x_n|=\frac{1}{2}\bar{r}.$ Next, for any $\rho\in(r_{n},\frac{1}{2}\bar{r})$, we consider the following two cases.

If $\rho\leq|x-x_n|\leq\frac{1}{2}\bar{r}$. It follows lemma \ref{harnack inequality} that there exists a constant $C>0$ a depending only on $N,\bar{C},\bar{r}$, such that
\begin{equation}
    \begin{aligned}
    |x-x_{n}|^{N-2}u_{n}(x_{n})u_{n}(x)&\leq|x-x_{n}|^{N-2}u_{n}(x_{n})\max_{\rho\leq|x-x_{n}|\leq\frac{1}{2}\bar{r}}u_{n}(x)\\
    &\leq|x-x_{n}|^{N-2}u_{n}(x_{n})C\min_{\rho\leq|x-x_{n}|\leq\frac{1}{2}\bar{r}}u_{n}(x)\\
    &\leq C u_{n}(x_{n})u_{n}(x_{n}+e)\leq C.
    \end{aligned}
\end{equation}

If $r_n\leq|x-x_n|\leq\rho.$ By contradiction, suppose
that there exists $\tilde{x}_n$ satisfying $r_n\leq|\tilde{x}_n-x_n|\leq\rho$, such that
\begin{equation}\label{eq isolate simple blow-up pioint prop 1-proof-2}
    |\tilde{x}_n-x_n|^{N-2}u_n(x_n)u_n(\tilde{x}_n)\to\infty,\text{~as~}n\to\infty.
\end{equation}
We set $\tilde{r}_n=|\tilde{x}_n-x_n|$ and we can suppose that $\tilde{r}_n\to0$, otherwise, there exists $\tilde{r}>0$ such that $\tilde{r}<|\tilde{x}_n-x_n|< \frac{1}{2}\bar{r}$, so we will get immediately a contradiction. Consider now the function $u_n(x)$ in the ball $|x-{x_{n}}|\leq\frac{1}{2}\bar{r}$ and for $|y|<\bar{r}/2\tilde{r}_{n},{\mathrm{~the~function}}$
\begin{equation}
    \tilde{u}_n(y)=\tilde{r}_n^{\frac{N-2}{2}}u_n(x_n+\tilde{r}_n y).
\end{equation}
It solves
\begin{equation}
   -\Delta\tilde{u}_{n}(y)=\tilde{u}_{n}^{2^*-1}(y)+\varepsilon_{n}\tilde{r}_{n}^{\frac{N-2}{2}(2^*-q)}\tilde{u}_{n}^{q-1}(y)\quad|y|<\bar{r}/2\tilde{r}_{n}. 
\end{equation}
and the point $y=0$ is an blow up point of $\tilde{u}_{n}$. Moreover,
since $x_{n}\to\bar{x}$ is an isolate blow-up point of $u_{n}$, thus
\begin{equation}
    \tilde{u}_n(y)\leq\frac{\bar{C}}{|y|^{\frac{N-2}{2}}},\quad|y|<\bar{r}/2\tilde{r}_n,
\end{equation}
and the point $y=0$ is an isolated blow up point of $\tilde{u}_{n}$.
On the other hand, since
\begin{equation}
   r^{\frac{N-2}{2}}\bar{\tilde{u}}_n(r)=(r\tilde{r}_n)^{\frac{N-2}{2}}\bar{u}_n(r\tilde{r}_n)=\bar{w}_n(r\tilde{r}_n), 
\end{equation}
then $y= 0$ is an isolated simple blow up point of $\tilde{u}_{n}$.  By the Lemma \ref{lem properties of isolated simple point 3}.
\begin{equation}
    \tilde{u}_n(0)\tilde{u}_n\left(\frac{\tilde{x}_n-x_n}{\tilde{r}_n}\right)\leq C.
\end{equation}
This implies
\begin{equation}
    |\tilde{x}_{n}-x_n|^{N-2}u_n(x_n)u_n(\tilde{x}_n)\leq C.
\end{equation}
This make a contradiction with \eqref{eq isolate simple blow-up pioint prop 1-proof-2} and thus the assertion \eqref{eq upper bound of isolated simple blow-up point} hold.

Next, we prove \eqref{eq convergence of isolated simple blow-up point} hold. Multiplying the equation \eqref{p-varepsion} by $u_{n}(x_{n})$ and integrating by parts on the ball $B(\bar{x},\frac{1}{2}\bar{r})$, we get
\begin{equation}\label{eq isolate simple blow-up pioint prop 2-proof-1}
    \begin{aligned}-\int_{B(\bar{x},\frac{1}{2}\bar{r})}u_{n}(x_{n})\Delta u_{n}=\int_{B(\bar{x},\frac{1}{2}\bar{r})}u_{n}(x_{n})u_{n}(x)^{2^*-1}+\varepsilon_{n}\int_{B(\bar{x},\frac{1}{2}\bar{r})}u_{n}(x_{n})u_{n}(x)^{q-1}.\end{aligned}
\end{equation}
From Proposition \ref{prop lower bounded of solutions} and  Lemma \ref{lem properties of isolated simple point 3},  after passing a subsequence, there exists a constant $\Lambda>0$ such that,
\begin{equation}
  \lim_{n\to\infty}u_{n}(x_{n})u_{n}(x_{n}+e)=\Lambda.  
\end{equation}
Thus, by Lemma \ref{lem properties of isolated simple point-2}, the left-hand side of \eqref{eq isolate simple blow-up pioint prop 2-proof-1} satisfies
\begin{equation}
    \begin{aligned}\mathrm{LHS}&=-\int_{B(\bar{x},\frac{1}{2}\bar{r})}u_{n}(x_{n})u_{n}(x_{n}+e)\Delta v_{n}\\
    &=-\int_{\partial B(\bar{x},\frac{1}{2}\bar{r})}u_{n}(x_{n})u_{n}(x_{n}+e)\frac{\partial v_{n}}{\partial\nu} \to (N-2)\:a_{1}\Lambda\omega_{N},\end{aligned}
\end{equation}
where $v_{n}=\frac{u_{n}}{u_{n}(x_{n}+e)}$. On the other hands, similar as the proof of Lemma \ref{lem properties of isolated simple point 3}, we have
\begin{equation}
    \mathrm{RHS}\to\frac{\alpha_{N}^{2^*}\omega_{N}}{N}.
\end{equation}
Comparing both sides of \eqref{eq isolate simple blow-up pioint prop 2-proof-1}, we have
\begin{equation}
    a_{1}\Lambda=\alpha_{N}^{2}.
\end{equation}
Thus, by Lemma \ref{lem properties of isolated simple point-2}, we have
\begin{equation}
\begin{aligned}\operatorname*{lim}_{n\to\infty}u_{n}(x_{n})u_{n}(x)&=\:\operatorname*{lim}_{n\to\infty}u_{n}(x_{n})u_{n}(x_{n}+e)v_{n}(x)\\&=\:\frac{\Lambda a_{1}}{|x-\bar{x}|^{N-2}}+\Lambda b_{1}(x)=\frac{\alpha_{N}^{2}}{|x-\bar{x}|^{N-2}}+b(x),\end{aligned}    
\end{equation}
where $b(x)=\Lambda b_1(x).$
\end{proof}

\subsection{Isolated blow-up point must be simple}
\begin{Prop}\label{prop isolate blow-up point is simple}
    If $x_{n}\to\bar{x}$ is an isolated blow-up point of $\{u_{n}\}_{n\in\N}$, then $\bar{x}$ muse be an isolated simple blow-up point of $\{u_{n}\}_{n\in\N}$.
\end{Prop}

\begin{proof}
From Corollary \ref{cor below estimate}, we know that $\bar{w}_{n}$ only one critical point in $(0,r_{n})$. By contradiction, suppose that there is a second critical point $\mu_n\geq r_n$ and we can further assume that 
\begin{equation}
    \mu_{n}\to0,\text{~~as~~}n\to\infty,
\end{equation}
otherwise, the proof is trivial. Next, we define the rescaled function
\begin{equation}
   \xi_n(y)=\mu_n^{\frac{N-2}{2}}u_n(\mu_n y+x_n),\text{~~for any~~}|y|<\frac{\bar{r}}{2\mu_n}. 
\end{equation}
Then we have
\begin{equation}
    \xi_{n}(0)=\mu_{n}^{\frac{N-2}{2}}u_{n}(x_{n})\geq r_{n}^{\frac{N-2}{2}}u_{n}(x_{n})=R_{n}^{\frac{N-2}{2}}\to\infty.
\end{equation}
Moreover, we can verify that
\begin{equation}
    -\Delta\xi_{n}=\xi_{n}^{2^*-1}+\varepsilon_{n}\mu_{n}^{\frac{N-2}{2}(2^*-q)}\xi_{n}^{q-1},\quad\mathrm{in}\:|y|<\frac{\bar{r}}{2\mu_{n}},
\end{equation}
and
\begin{equation}
    \xi_{n}(y)\leq\frac{\bar{C}}{|y|^{\frac{N-2}{2}}}\quad\mathrm{in}\quad|y|<\frac{\bar{r}}{2\mu_{n}}.
\end{equation}
Note that $\mu_{n}\to0$ as $n\to\infty$, thus the point $y=0$ is an isolated blow up point for $\xi_n$ in any compact of $\R^N.$
Moreover, from
\begin{equation}
    r^{\frac{N-2}{2}}\bar{\xi}_n(r)=(r\mu_{n})^{\frac{N-2}{2}}\bar{u}_n(r\mu_{n}),
\end{equation}
we can see that $r^{\frac{N-2}{2}}\bar{\xi}_n(r)$ has only one critical point for $0<r<1$ and
\begin{equation}
    \frac{d}{dr}(r^{\frac{N-2}{2}}\bar{\xi}_{n}(r))|_{r=1}=0.
\end{equation}
This implies that $y=0$ is an isolated and simple blow up point for $\xi_n$ with $\bar{r}=1.$ Then from the Proposition \ref{prop properties of isolated simple blow-up point} and Lemma \ref{harnack inequality}, we obtain that there exist a constant $C>0$ independent on $n$ such that for any $\sigma>0$
\begin{equation}
    \xi_{n}(0)\xi_{n}(y)\leq\frac{C}{|y|^{N-2}},\quad\textrm{for any}\:|y|\leq\sigma,
\end{equation}
which together with the Proposition \ref{prop properties of isolated simple blow-up point} (applied on any compact of $\R^N)$ that there exists a function $b=b(y)$, which is harmonic in all ${\R}^N$, such that
\begin{equation}
    \xi_{n}(0)\xi_{n}(y)\to h(y)=\frac{\alpha_{N}^{2}}{|y|^{N-2}}+b(y),\quad\mathrm{in}\:C_{loc}^{2}(\R^{N}\backslash\{0\}).
\end{equation}
Note that $h(y)\geq0$, then by the maximum principle, we have that $b(y)\geq0$ in all $\mathbf{R}^N,$ which together with the Liouville Theorem that $b(y)=b_0$ is constant. Moreover, from
\begin{equation}
    \frac{d}{dr}(r^{\frac{N-2}{2}}\bar{\xi}_{n}(r))|_{r=1}=0=\xi_{n}(0)\frac{d}{dr}(r^{\frac{N-2}{2}}\bar{\xi}_{n}(r))|_{r=1},
\end{equation}
and $\xi_n(0)\bar{\xi}_n(r)\to h(r)=\alpha_{N}^{2}r^{2-N}+b_0$,
\begin{equation}
    \frac{d}{dr}(r^{\frac{N-2}{2}}h(r))|_{r=1}=0,
\end{equation}
and it follows that
\begin{equation}
    b_{0}=\alpha_{N}^{2}.
\end{equation}
Then, we can obtain that
\begin{equation}\label{eq isolate blow-up point is simple-proof-14}
   \xi_{n}(0)\xi_{n}(y)\to h(y)=\frac{\alpha_{N}^{2}}{|y|^{N-2}}+\alpha_{N}^{2},\quad\mathrm{in}\:C_{loc}^{2}(\mathbf{R}^{N}\backslash\{0\}). 
\end{equation}

Finally, let us recall the Poho\v{z}aev identity of function $\xi_n$ on $B(0,\sigma)$ from \eqref{pohozaev identity 1} and
multiply both sides by $\xi_n(0)^2.$ We obtain,
\begin{equation}\label{eq isolate blow-up point is simple-proof-15}
    \begin{aligned}&\varepsilon_{n}\mu_{n}^{\frac{N-2}{2}(2^*-q)}\frac{2N-(N-2)q}{2q}\int_{B(0,\sigma)}\xi_{n}(0)^{2}\xi_{n}^{q}-\frac{\sigma}{2^*}\int_{\partial B(0,\sigma)}\xi_{n}(0)^{2}\xi_{n}^{2^*}\\
    &-\varepsilon_{n}\mu_{n}^{\frac{N-2}{2}(2^*-q)}\frac{\sigma}{q}\int_{\partial B(0,\sigma)}\xi_{n}(0)^{2}\xi_{n}^{q}=\int_{\partial B(0,\sigma)}B(x,\sigma,\xi_{n}(0)\xi_{n},\nabla(\xi_{n}(0)\xi_{n})).\end{aligned}
\end{equation}
where $B(x,\sigma,\xi_{n}(0)\xi_{n},\nabla(\xi_{n}(0)\xi_{n}))$ defined in \eqref{definition of B-x-sigma-u}. Similar as the proof of the Lemma \ref{lem blou-up rate estimate}, we obtain that for any $\sigma>0$, the left-hand side of \eqref{eq isolate blow-up point is simple-proof-15}
\begin{equation}\label{eq isolate blow-up point is simple-proof-16}
\begin{aligned}
    \mathrm{LHS}&\geq\varepsilon_{n}\mu_{n}^{\frac{N-2}{2}(2^*-q)}\xi_n(0)^{q+2-2^*}(C(N,q)-o(1))-o(1).
\end{aligned}    
\end{equation}
where $C(N,q)>0$ is a constant. Now, we estimate the right-hand side of \eqref{eq isolate blow-up point is simple-proof-15}. First, by using that Gradient estimates (with the same argument of the proof of Lemma \ref{lem blou-up rate estimate}) and Lemma \ref{lem properties of isolated simple point 3}. Then, there exist a constant $C>0$, such that
\begin{equation}
    B(y,\sigma,\xi_{n}(0)\xi_{n},\nabla(\xi_{n}(0)\xi_{n}))\leq C.
\end{equation}
Thus by the dominate convergence theorem and \eqref{eq isolate blow-up point is simple-proof-14}, we have
\begin{equation}
   \lim_{n\to0} \int_{\partial B(0,\sigma)}B(y,\sigma,\xi_{n}(0)\xi_{n},\nabla\xi_{n}(0)\xi_{n})=\int_{\partial B(0,\sigma)}B(y,\sigma,h(y),\nabla h(y)).
\end{equation}
Moreover, by lemma \ref{lem properties of B}, we have
\begin{equation}
    \lim_{\sigma\to0}\int_{\partial B(0,\sigma)}B(y,\sigma,h(y),\nabla h(y))=-\frac{(N-2)^{2}\alpha_{N}^{4}\omega_{N}}{2}<0.
\end{equation}
This make a contradiction with \eqref{eq isolate blow-up point is simple-proof-16} and we complete the proof.
\end{proof}
\smallskip

\noindent{\bf Ethics   declaration:}  Not   applicable.

\smallskip

\noindent{\bf Conflict   of   Interest:} The  authors  state  that  there  is  no  conflict  of  interest.

\smallskip

\noindent{\bf Data  availability  statements:} Data  sharing  not  applicable  to  this  article  as  no  datasets  were  generated  or 
 analysed  during  the  current  study.


\end{document}